\documentclass[12pt]{amsart}

\usepackage[T1]{fontenc}
\usepackage{times}
\usepackage{amssymb,epsfig,verbatim,xypic,relsize,stmaryrd}
\usepackage{calligra,mathrsfs}
\theoremstyle{plain}
\usepackage[T1]{fontenc}
\usepackage[utf8]{inputenc}
\usepackage[english]{babel}
\usepackage{amssymb}
\usepackage{amsthm}
\usepackage{graphics}
\usepackage{amsmath}
\usepackage{amstext}
\usepackage{multirow}
\usepackage{enumerate}

\usepackage{hyperref}
\usepackage{graphicx}
\usepackage{amsmath}
\usepackage{amssymb}
\usepackage{amsthm}
\usepackage{amstext}
\usepackage{epstopdf}
\usepackage{mathrsfs}
\usepackage{amscd}
\usepackage{tikz}
\usetikzlibrary{cd}
\AtBeginDocument{%
   \def\MR#1{}
}

\newtheorem{thm}{Theorem}[section]
\newtheorem{lemma}[thm]{Lemma}
\newtheorem{prop}[thm]{Proposition}

\newtheorem{cor}[thm]{Corollary}

\theoremstyle{remark}
\newtheorem{rmk}[thm]{Remark}
\newtheorem{ex}[thm]{Example}
\newtheorem{question}[thm]{Question}

\newtheorem{THM}{Theorem}

\newtheorem{COR}[THM]{Corollary}
\newtheorem{QUESTION}[THM]{Question}

\theoremstyle{definition}
\newtheorem{dfn}[thm]{Definition}

\theoremstyle{remark}
\newtheorem{remark}[thm]{Remark}

\newcommand{\mb}{\mathbb}

\newcommand{\C}{\mb C}

\newcommand{\Z}{\mb Z}
\newcommand{\Q}{\mb Q}

\newcommand\restr[2]{{
  \left.\kern-\nulldelimiterspace 
  #1 
  \vphantom{\big|} 
  \right|_{#2} 
  }}

\DeclareMathOperator{\codim}{codim}
\DeclareMathOperator{\sing}{Sing}

\DeclareMathOperator{\Spec}{Spec}

\DeclareMathOperator{\Hom}{Hom}

\DeclareMathOperator{\HomSheaf}{\mathscr{H}\text{\kern -3pt {\calligra\large om}}\,}

\DeclareMathOperator{\Gr}{\mathrm{Gr}}
\DeclareMathOperator{\SLie}{\mathrm{SLie}}
\DeclareMathOperator{\Commuting}{\mathrm{C}}

\newcommand{\ca}{\mathcal{A}}
\newcommand{\Si}{\Sigma}
\newcommand{\la}{\lambda}
\newcommand{\ov}[1]{\overline{#1}}
\newcommand{\PP}{\mathbb{P}}
\newcommand{\om}{\omega}

\newcommand{\wh}[1]{\widehat{#1}}
\newcommand{\g}[1]{\mathfrak{#1}}
\newcommand{\aro}{\longrightarrow}
\newcommand{\arou}[1]{\stackrel{#1}{\longrightarrow}}
\newcommand{\mm}[1]{\mathrm{#1}}
\newcommand{\cv}{\mathcal{V}}
\newcommand{\cw}{\mathcal{W}}
\newcommand{\cf}{\mathcal{F}}
\newcommand{\co}{\mathcal{O}}
\newcommand{\cm}{\mathcal{M}}

\newcommand{\cb}{\mathcal{B}}

\newcommand{\ck}{\mathcal{K}}

\newcommand{\ot}{\otimes}
\newcommand{\ce}{\mathcal{E}}
\newcommand{\ti}{\times}

\numberwithin{equation}{section}

\addtocounter{section}{0}             
\numberwithin{equation}{section}       

\title[Locally free  tangent sheaf]{Distributions with locally free tangent sheaf}

\author[J.V. Pereira]{Jorge Vit\'orio Pereira}
\address{IMPA, Estrada Dona Castorina, 110, Horto, 22460-320,  Rio de
Janeiro, Brasil}

\author[J.P. dos Santos]{Jo\~ao Pedro   dos Santos}
\address{Institut Montpelli\'erain Alexander Grothendieck, Universit\'e de Montpellier,
Place Eug\`ene Bataillon, 34090 Montpellier, France
}

\date{\today}

\dedicatory{Dedicated to Fernando Cukierman on the occasion of his 70th birthday}

\begin{document}


\begin{abstract}
In \cite{cukierman-pereira} one finds a  study of the locus $\mathrm{Dec}$  where the tangent sheaf of a {\it family} of foliations in $\mathbb{P}_\C^n$ is {\it decomposable}, i.e. a sum of line bundles. A  prime conclusion is an  ``openness'' result:  once the singular locus has sufficiently large codimension, $\mathrm{Dec}$ turns out to be open. 
In the present paper,  we study the locus $\mathrm{LF}$ of points of a family of distributions where the tangent sheaf is {\it locally free}. Through general Commutative Algebra, we show that $\mathrm{LF}$ is open provided that singularities have codimension at least three. When dealing with foliations rather than distributions, the condition on the lower bound of  the singular set can be weakened by the introduction of ``Kupka'' points. 
We apply the available ``openness'' results to    families in $\mathbb{P}^n_\C$ and in $\mathcal B$, the variety of  Borel subgroups  of a simple group. By establishing    a theorem putting in bijection irreducible components of the space of two-dimensional  subalgebras   of a given semi-simple Lie algebra and its nilpotent orbits, we conclude that the space of foliations of {\it rank two}  on $\mathbb{P}_\C^n$ and $\mathcal B$, may have quite many irreducible components as $n$ and $\dim \mathcal B$ grow. We also set in place several algebro-geometric foundations for the theory of families of distributions in two appendices. 
\end{abstract}

\maketitle
\setcounter{tocdepth}{1}
\tableofcontents

\section{Introduction}

In this paper, we wish to throw light on the  distributions and foliations admitting a locally free tangent sheaf. More precisely, we wish to study how this property may vary in a family with the goal of   showing that    whenever  singularities have   ``large codimension'', this  property  is ``open''.

Now, 
if we understand foliations as a generalisation of vector fields or 1-forms, then the idea of studying  which topological  properties are preserved by small deformations has a distinguished history and goes under the name of ``structural stability'' \cite{andronov-pontryagin37,peixoto59,kupka,medeiros77}. On the other hand, it is then only natural to look at {\it algebraic} output properties, as  was 
 done in 
   \cite{cukierman-pereira}  in the case of forms on $\PP^n_\C$. There,  the authors use the precise description of $q$-forms on $\PP^n_\C$ in terms of homogeneous polynomial  forms in order to show that if $\omega\in H^0(\PP_\C^n,\Omega^q(m))$ defines a singular foliation   such that  $\codim\sing \ge3$   and the tangent sheaf $T_\omega$ associated   to it  is decomposable (i.e. a sum of invertible sheaves), then any $q$-form (inducing a foliation)  sufficiently close to $\omega$ also produces   a decomposable tangent sheaf.  
  See Theorem \ref{CP-theorem} below for a restatement. 
   The present work seeks to expand the scope of this result in the following directions:  
 \begin{enumerate}[(i)]  \item consider locally free sheaves which are not necessarily sums of invertible ones; 
\item  consider  varieties other than $\PP^n$ (and Lie algebras other than $\g{sl}_{n+1}$); 
\item consider  base fields other than $\C$; 
\item consider the ``Kupka'' condition on singularities for higher codimensions and more general ambient spaces. 
\end{enumerate}

In addition, we 

\begin{enumerate}[(i)] \item[(v)] wish
to find cases where  the assumption on the codimension of  singularities is naturally met and, once this is done,
\item[(vi)]use the results to find irreducible components of the spaces of integrable  forms (or of foliations). 
\end{enumerate} 

In the ensuing sections, we describe more precisely the themes and results which are treated in this work. 

\subsection{``Openess'' results}\label{introduction_openess_results}
The results obtained under this rubric are the raison d'\^etre of this paper. They generalise \cite[Theorem 2]{cukierman-pereira} (reproduced as Theorem \ref{CP-theorem} below) and are explained in Section  \ref{20.05.2024--1}.

The concept of a holomorphic (algebraic) singular  foliation $\mathcal F$ on a given   (algebraic) complex manifold $X$ is well-established and it has become in the past decades   crucial --- or at least standard ---  to engage it by means of sheaf theory.  So, the several qualifications that a sheaf can have  become important indicators of the singular foliation itself. In this part of the Introduction, we wish to explain our results showing that the property of local freeness of tangent sheaves is ``stable'', or ``open''.

Even a superficial contact with the theory of foliations shows that the usual operations which one makes with a sheaf do not immediately correspond to the operations on the foliations side (see Example \ref{counter_example} for a sample of this). This adds complexity to the   study of {\it families of foliations} and the application of known results of Commutative 
Algebra to this setting requires special  efforts. Indeed, 
even  in a very comfortable context, pulling back foliations is not synonym with pulling back sheaves. For example, let   $f:X\to S$ be a smooth morphism of smooth complex algebraic schemes and let $\cf$ be a singular foliation on $X$ whose tangent sheaf $T_\cf$ 
is a subsheaf of $T_f=T_{X/S}$, the relative tangent sheaf. 
Then, 
given a closed point $s\in S(\C)$, the restricted foliation $\cf_s$ induced on the fibre  $X_s$ {\it does not}
  have, in general,   the restriction $T_\cf|_{X_s}$ as tangent sheaf. 
  On the other hand,  the link between $T_{\cf_s}$ and $T_\cf|_{X_s}$ can be given through Commutative Algebra, which is used to prove the following result.  For unexplained notation, the reader is directed to Appendix \ref{generalities}; here we content ourselves with pointing out that $\mathrm{Sing}$ stands for the singular set of a relative distribution $\cv$, i.e. the set of points where   $T_f/T_\cv$ fails to be free.

\begin{THM}\label{T:locally free}
    Let $f:   X \to S$ be a proper and smooth  morphism of  Noetherian schemes. Let $\mathcal V$ be a relative distribution on $X$, i.e. a strongly saturated coherent submodule of the relative tangent sheaf $T_f$. 
For each $s\in S$, let $X_s$ denote the (schematic) fibre of $X$ above $s$ and denote by $\cv_s$ the pull-back of $\cv$ to $X_s$.  
Assume that for a certain $o\in S$,  we have 
\begin{enumerate}[(1)]
\item that $\codim(\mathrm{Sing}(\cv)\cap X_o,X_o)\ge3$ 
and that 
\item  the tangent sheaf $T_{\cv_o}$ is locally free. 
\end{enumerate}
 Then for all $s$ in a certain neighbourhood of $o$,  the tangent sheaf $T_{\cv_s}$ is also locally free. 
\end{THM}

In giving a proof of Theorem \ref{T:locally free}, we employed a somewhat weaker hypothesis  than the ``smallness'' of the singular set (explained above), which is the  ``smallness'' of the set of {\it tangent singularities} $\mathrm{TSing}$. (See Theorem \ref{07.11.2023--1}.)
By definition, tangent singularities  are points   where the tangent sheaf fails to be free. 
Now, as explained in Lemma \ref{08.11.2023--1}, for any distribution $\cv$, we have an inclusion $\mathrm{TSing}(\cv)\subset\mathrm{Sing}(\cv)$  
 and one of the best tools to obtain points of $\mathrm{Sing}(\cv)\setminus\mathrm{TSing}(\cv)$ is, when $\cv$ is involutive, by means   of ``Kupka''
singularities $\mathrm{Kup}(\cv)$  (cf. Section \ref{kupka_singularities}). Indeed, as is shown in Proposition \ref{23.05.2024--1} at least when $S$ is locally factorial, the inclusion $\mathrm{NKup}(\cv)\supset  \mathrm{TSing}(\cv)$ holds true. (Here $\mathrm{NKup}(\cv)$ is the set of non-Kupka'' singularities $\sing(\cv)\setminus\mathrm{Kup}(\cv)$.)   
Hence, we can state a variant of  Theorem \ref{T:locally free} which reads.

\begin{THM}\label{THM:Kupka}Let $f:   X \to S$ be a proper and smooth morphism of  Noetherian  schemes. Let $q\ge1$ be an integer, $L$ a line bundle on $X$ and $\omega\in \Gamma(X,\Omega^q_f\otimes L)$ an LDS form (cf. Section \ref{twisted_form}). Denote the associated distribution by $\cv$ and suppose that $\cv$ is in fact   involutive.  
 
For each $s\in S$, let $X_s$ denote the fibre of $X$ above $s$ and denote by $\cv_s$ the pull-back of $\cv$ to $X_s$.  
Assume that for a certain $o\in S$,  
\begin{enumerate}[(1)]\item  we have  $\codim(\sing(\cv)\cap X_o,X_o)\ge2$, 
\item  if   $\mm{NKup}(\cv)$ stands for the  ``non-Kupka locus'' of $\sing(\cv)$, then 
\[
\mathrm{codim}(\mm{NKup}(\cv)\cap X_o\,,\,X_o)\ge3,\]
and
\item the sheaf $T_{\cv_o}$ is locally free.  
\end{enumerate}
Then for all $s$ in a certain neighbourhood of $o$,  the tangent sheaf $T_{\cv_s}$ is also locally free. 
 \end{THM}      
 
\subsection{The space of  Lie subalgebras}
The parts   under this rubric are   in Section \ref{Variety_of_Lie_subalgebras};  their relevance to this work  lies in the fact that one of its outputs (Theorem \ref{THM:A}) shall  enable us to give   lower bounds to the number of irreducible components of the space of foliations (see Section \ref{20.05.2024--3}). But we highlight that the theme  ``spaces of subalgebras'' is intrinsically  an interesting one and may be pursued for its own sake. So much so that Section \ref{Variety_of_Lie_subalgebras}
can be read independently and only requires basic knowledge of Algebraic Geometry and Lie algebras. 

An important class of foliations is given by (certain) actions of a Lie algebra. More precisely, given a field $\Bbbk$ and a smooth $\Bbbk$-scheme $X$, any subalgebra $\g h$ of the Lie algebra $H^0(X,T_{X/\Bbbk})$ gives rise to an $\co_X$-linear map $\gamma:\g h\ot\co_X\to T_{X/\Bbbk}$ compatible with   Lie brackets.  Granted reasonable hypothesis,   this gives rise to  a foliation $\mathcal A(\g h)$ with     tangent sheaf $\g h\ot\co_X$,  so that  ``openness results'' (cf.  in Section \ref{introduction_openess_results}) may lead to   information  about the ``size'' of foliations by actions inside a given family. This idea shall be  explored in detail (see Section  \ref{20.05.2024--3}).

Let $\Bbbk$ be a field and let    $\g g$ be a finite dimensional Lie algebra over $\Bbbk$. For a given integer  $1\le r<\dim\g g$, it is not difficult to construct a closed subscheme $\mathbf{SLie}(r,\g g)$ of the   Grassmann variety $\mathrm{Gr}(r,\g g)$ of $r$-dimensional subspaces in $\g g$, whose $\Bbbk$-valued points  correspond  to $r$-dimensional  subalgebras of $\g g$. This beautiful space was introduced by Richardson \cite{richardson67} and has  unfortunately drawn little attention from   mathematicians. In  Section \ref{Variety_of_Lie_subalgebras}, we show :
\begin{THM}\label{THM:A}Let $\Bbbk$ be algebraically closed and of characteristic zero, and  let  $\mathfrak g$ be   a  semi-simple Lie algebra over $\Bbbk$ other than $\g{sl}_2$.   
The irreducible components of  
$\mathbf{SLie}(2,\g g)$ are in {\it bijection} with
    the set of conjugacy classes of nilpotent elements of $\mathfrak g$.
\end{THM}

The veracity of this  result can be grasped by the two observations: (1)   the ``abelian'' $\Bbbk$-points of $\SLie(2,\g g)$ are in close relation with the ``variety of commuting elements'', which is irreducible by     \cite{richardson79},  and (2)   given a non-abelian 2-dimensional  subalgebra $\g h$, the derived algebra $[\g h,\g h]$ is a line of {\it nilpotent elements} (cf. Section \ref{10.06.2024--1}),  and is hence contained in a {\it single} nilpotent orbit \cite[VIII.11.2-3]{bourbaki_lie}.

\subsection{Irreducible components of the spaces of foliations}\label{20.05.2024--3}
Let $m\ge0$ and $n\ge1$ be  integers and consider the quasi-projective complex algebraic set 
\[
\mathrm{Fol}^{1}(\mathbb P_\C^n,m)
=\left\{\omega\in\mathbb P\left(H^0(\mathbb P^n_\C,\Omega^1  (m+2))\right)\,:\,\begin{array}{l}\mm d\omega\wedge\omega=0\\\dim\mathrm{Sing}(\omega)\le n-2
\end{array}
\right\},  
\] 
   the {\it space of foliations of codimension one and degree $m$ on $\mathbb P^n$},  or the space of  {\it irreducible and integrable algebraic Pfaff equations} \cite[Ch. 2]{jouanolou79}. Given its encompassing geometric content, $\mathrm{Fol}^1(\mathbb P_\C^n,m)$ has always attracted attention of geometers and the determination of the number of its  irreducible components is an affordable  and beautiful problem.
    For the sake of discussion, let us write 
\[\mathrm{ic}^1 (\PP_\C^n,m):=\#\,\text{irreducible components of $\mathrm{Fol}^1(\mathbb P^n,m)$}.\]
(As traditional in many areas of Geometry, in our notations, upper-indices refer to ``codimension'' while lower ones to ``dimension''.)    
   
 If $m\le2$, it is known that  
 \[\mathrm{ic}^1(\PP_\C^3,m)=\mathrm{ic}^1(\PP_\C^4,m)=\cdots.\] 
 More precisely,   $\mathrm{ic}^1(\PP^n_\C,0)=1$   (this is a simple exercise),  while  $\mathrm{ic}^1(\PP_\C^n,1)=2$  provided that $n\ge3$  \cite[Ch. 1, Proposition 3.5.1]{jouanolou79}.   A   central result of the theory \cite{cerveau-linsneto} affirms that \[\mathrm{ic}^1(\PP_\C^n,2)=6\quad\text{whenever $n\ge3$.}\] On the other hand, an    upper bound for $\mathrm{ic}^1(\PP_\C^n,m)$ was obtained in \cite[Theorem 1.2]{ballico}. 
 
     Needless to say, this circle of ideas is also relevant in higher codimension. 
Let $m\ge0$ and $q\in\{1,\ldots,n\}$ be given, and for each $\omega\in H^0(\PP^n_\C,\Omega^q(m+q+1))$, define 
\[
\mathrm{Ker}_\omega=
 \mathrm{Ker}\left(\xymatrix{
 T \ar[r]^-{\mathrm{id}\ot\omega}& T \otimes \Omega^q (m+q+1) \ar[rr]^{\text{contract}\ot\mathrm{id}}&&\Omega^{q-1}(m+q+1)
}\right).
\]
(Here $T$ stands for the tangent sheaf of $\PP^n_\C$.)
Then,   inside $\PP H^0(\PP^n_\C,\Omega^{q}(m+q+1))$, there exists  a locally closed algebraic subset 
\[
\mm{Fol}^q(\PP^n,m)=\begin{array}{c}\text{``Space of singular foliations }\\
\text{on $\PP^n_\C$ of codimension $q$}
\\
\text{ and degree $m$'',}
\end{array}
\]
such that     \[\omega\longmapsto\mathrm{Ker}_\omega\]
establishes a bijection between $\mm{Fol}^q(\PP^n,m)
$ and the saturated, involutive $\co$-submodules of $T$ of rank $n-q$, i.e. singular foliations of codimension $q$. Now, writing $r=n-q$ and putting $\mm{Fol}_r=\mm{Fol}^q$, let 
\[\begin{split}
\mathrm{ic}_{r}(\PP^n_\C,m)&=\mathrm{ic}^q(\PP^n_\C,m)\\&=\#\,\text{irreducible components of $\mathrm{Fol}_r(\mathbb P^n,m)$}.\end{split}
\]
Then, provided that $r<n$, it is known that  
\[
\mathrm{ic}_r(\PP^n_\C,0)=1\qquad\text{and}\qquad \mm{ic}_r(\PP^n_\C,1)=2.
\]
The first equality can be read between the lines in \cite{cerveau-deserti}, while the second is  \cite[Corollary 6.3]{loray-pereira-touzet13}. In sharp contrast, we deduce  the following from Theorem \ref{THM:A} and \cite[Corollary 6.1]{cukierman-pereira}.

\begin{COR}\label{20.05.2024--4}Let $\mathrm{ic}_{2}(\PP^n_\C,2)$ be the number of irreducible components of the space $\mathrm{Fol}_2(\mathbb P_\C^n,2)$  of  foliations on $\mathbb P_\C^n$ of degree $2$ and dimension $2$.  Then, for $n\ge4$, we have 
 \[
\mathrm{ic}_{2}(\PP^n_\C,2)\ge \#\left\{\text{partitions of $n-4$}\right\}. \]
 In particular, by ``Hardy-Ramanujan'' (an explicit lower bound is in \cite[Corollary 3.1]{maroti}), we have
 \[\mm{ic}_{2}(\PP^n_\C,2)>\frac {e^{2\sqrt {n-4}}}{14}.\]  
\end{COR}

In order to apply \cite[Corollary 6.1]{cukierman-pereira} and Theorem \ref{THM:A} to prove Corollary \ref{20.05.2024--4}, we are required to   construct foliations on $\PP^n_\C$ of {\it rank two} 
whose singular set has codimension {\it at least three}. This we obtain by direct computations with Lie algebras of vector fields in $\PP^n_\C$, see Corollary  \ref{06.05.2024--1}. Now, since Theorem \ref{T:locally free} generalizes \cite[Corollary 6.1]{cukierman-pereira}, the search for Lie algebras of vector fields with ``small'' singular sets becomes important. 

\begin{QUESTION}\label{question_intro}Let $\Bbbk$ be a field, $X$ a smooth and projective variety, $\g g$ the Lie algebra of global vector fields on $X$, and $\gamma:\co_X\ot\g g\to T_{X}$ the natural morphism of $\co_X$-modules. 
Are there subalgebras $\g h\subset \g g$ such that the ``singular set'' 
\[
\left\{p\in X\,:\,
\begin{array}{c}
(\co_X\ot\g h)(p)\stackrel\gamma\to T_X(p)
\\ \text{fails to be injective}
\end{array}
\right\}
\]
has large codimension? 
\end{QUESTION} 

In order to have positive answers   en masse, we are led to  consider varieties of Borel subgroups (or Borel varieties), see   Section \ref{22.02.2024--3}. The advantage of looking at these comes from the possibility of drawing from the vast literature established by group theorists   on the theme  of ``regular elements'' and ``Springer fibres''. Indeed,  early in the theory of algebraic groups,   the importance of having information on the ``number''
of Borel subgroups containing a given element was noticed \cite{springer69}. This produced an impressive body of work \cite{mcgovern02},   allowing us to have deep results at hand.  It then becomes possible to give satisfactory answers to Question \ref{question_intro} and then to 
 apply our ``openess results'' (Theorem \ref{T:locally free}) in order to conclude prove: 

\begin{THM}Let $\Bbbk$ be an algebraically closed field of characteristic zero. Let $G$ be a linear algebraic group of adjoint type with simple Lie algebra $\g g$. Let $\Phi$ be the reduced root system associated to the choice of a Borel subgroup   of $G$ and $X$ the corresponding   Borel variety. Assume that $\Phi$ is not of type $\mathrm A_1,\mathrm A_2,\mathrm A_3,\mathrm B_2,\mathrm C_3$ or $\mathrm G_2$. 

\begin{enumerate}[(1)]
\item Given $\g h$ in $\SLie(2,\g g)(\Bbbk)$, the natural morphism of $\co_X$--modules
\[\gamma:\co_X\ot\g h\aro T_X\]
is injective on an open subscheme whose complement has codimension at least three. (See Section \ref{01.04.2024--2}.)
\item There exists a morphism of reduced schemes 
\[\psi:\SLie(2,\g g)\aro \mathrm{Fol}_2(X,\det T_X)\]
with the following properties. 

\begin{enumerate}[(a)] 
 \item For $\g h\in\SLie(2,\g g)(\Bbbk)$, the point $\psi(\g h)$ is the $(\dim X-2)$-form associated to the foliation deduced from 
$\co_X\ot\g h\to T_X$. (See Section \ref{01.04.2024--2}.)
\item The morphism $\psi$ is open. (See Theorem \ref{10.06.2024--2}) 
\end{enumerate}
\item The number of irreducible components of $\mathrm{Fol}_2(X,\det T_X)$ is at least the number of nilpotent conjugacy classes in  $\g g$. 
\end{enumerate}\end{THM}

\subsection{Generalities}
In the last four decades  there has been an interest in studying ``spaces of distributions'' or ``spaces of foliations'' using three different methods. The first  tends to identify distributions (or foliations)  with differential forms and finds its modern introduction in  the influential text \cite[4.1]{jouanolou79}.
Once we deal with distributions of codimension higher than 1, this approach depends on the correspondence between  distributions and forms \cite{medeiros77,medeiros}. Another point of view, introduced by \cite{pourcin87} and adopted in \cite{quallbrunn15,correa-jardim-muniz22,velazquez24}, uses Grothendieck's Quot-scheme in order to study distributions by means of their sheaves. Finally, another method is that of \cite{barlet},  which changes the usage of the Quot-scheme by the Chow scheme. 
  These approaches are not identical and comparing them more carefully is to be done elsewhere.  Here, we shall work with ``distributions spaces'' 
as ``spaces of twisted forms'', staying thus in the line of \cite{jouanolou79,cerveau-linsneto,medeiros,cukierman-pereira}. 
This decision has a cost, which we now explain. 

 The reader 
entering the bulk of our arguments will probably realize that a   certain amount of non-trivial Commutative Algebra is necessary to attain our goals. Indeed, already a working  definition of a family of distributions requires some care, since  flatness ``on the base'' (which amounts to the approach through the Quot-scheme \cite{quallbrunn15}) cannot be taken unreservedly (see for example \cite[Section 2]{quallbrunn15}).   
 This is the reason for which we end the paper   with   generalities in Commutative Algebra and in the theory of distributions  in the algebraic case:   Appendices \ref{generalities} and \ref{generalitiesbis}. In them, we  profit  to introduce a number of techniques from Commutative Algebra in the context of distributions, hoping that they may be useful not only as a guide to the non-initiated, but also as a set of first steps into a possible more harmonious   theory for ``families'' of distributions. 
That this is a worthy enterprise may be justified by the fact that very little is known about the Commutative Algebra of the   ``spaces of foliations'';  having enough elbow-room to deal with these spaces in a uniform manner should be useful elsewhere.

\subsection*{Notations and conventions}

\begin{enumerate}[(1)]
\item All rings are commutative and unital and    $\Bbbk$ always stands for a field.
\item All Lie algebras are of finite dimension. 
 \item An open subset $U$ of a Noetherian scheme is called {\it big} if $\mathrm{codim}(X\setminus U,X)\ge2$. 

\item For a vector space $V$, we let $\mathbb P(V)$ be the projective space of {\it lines} in $V$. 
\item Given a ring $A$ and an $A$-module $M$, we let $M^\vee$ be $\mathrm{Hom}_A(M,A)$. A similar notation is employed for sheaves of modules.
\item Given a point $x$ on a scheme $X$, we let $\boldsymbol k(x)$ stand for its residue field, i.e. the residue field of the local ring $\co_x$. 
\item Given a morphism of schemes $f:Y\to X$, the fibre of $f$ above a point $x\in X$, often denoted by $Y_x$, is the scheme $Y\times_{X}\mathrm{Spec}\,\boldsymbol k(x)$. This is {\it homeomorphic} to the subspace $f^{-1}(x)$ of  $Y$ \cite[I, 3.6.1]{ega}.
\item By a vector bundle on a scheme, we mean a locally free sheaf of finite and fixed rank. 
\item For a smooth morphism of schemes $f:X\to S$, we shall write $T_f$ for the relative tangent sheaf. (One usually writes $T_{X/S}$, which we avoid for the sake of simplicity.)
\item Let $\mathcal F$ be a vector bundle and $\mathcal E$ be an $\co$-submodule. We say that $\mathcal E$ is a {\it subbundle} of $\mathcal F$ if  $\mathcal E$ and $\mathcal F/\mathcal E$ are vector bundles. 
\item Given a noetherian scheme $X$ and a coherent $\co_X$-module $\cm$ on it, its singular scheme $\sing(\cm)$ is the closed set of points $x\in X$ where $\cm_x$ fails to be a free $\co_x$-module.
\item Given a noetherian scheme $X$ and $\gamma:E\to F$ a morphism of vector bundles, we define $\mathrm{Sing}(\gamma)$ as being the set of points $x\in X$ where the map induced on the $\boldsymbol k(x)$-spaces  $\gamma(x):E(x)\to F(x)$    fails to be {\it injective}. This is a closed set. If $s$ is a section of $E$, then $\mathrm{Sing}(s)$ is the singular set of the obvious  morphism   $s:\co_X\to E$. 
\end{enumerate}

\subsection*{Acknowledgments} We would like to thank Reimundo Heluani for insightful discussions and M. Brion for pointing out reference \cite{demazure}.
The first named author extends his heartfelt gratitude to Fernando Cukierman for the inspiring conversations   and collaborative efforts around the theme of spaces of foliations.  He also  acknowledges the support from  CAPES/COFECUB, CNPq Projeto Universal 408687/2023-1 ``Geometria das Equa\c{c}\~oes Diferenciais Alg\'ebricas'', CNPq (Grant number 304690/2023-6), and FAPERJ (Grant number E26/200.550/2023). The second named author would like to thank the CNRS and the  IMPA. The 
former for granting him a one year research leave and the latter for
the peaceful and inspiring  environment where this work was carried out.

\section{The scheme of Lie subalgebras}\label{Variety_of_Lie_subalgebras}
In this section, we go over the theory of the scheme of Lie subalgebras of a given   Lie algebra. In order to give fluidity to the exposition, the schemes in sight shall all be algebraic (i.e. of finite type) over a given algebraically closed field $\Bbbk$ \cite[$\mathrm{I}$, 6.4.1]{ega}. In addition, $\Bbbk$ shall soon be taken to be of characteristic zero. Points on schemes are always closed, i.e. $\Bbbk$-points. 

\subsection{The scheme $\mathbf{SLie}$ parametrising subalgebras {\cite{richardson67}}}\label{01.01.2024--2}
Let $\g g$ be a Lie algebra over   $\Bbbk$ and, given $1\le r<\dim\g g$, let $\mathrm{Gr}(r,\g g)$ stand for the Grassmann variety of $r$-dimensional subspaces of $\g g$. By construction, we have an exact sequence of locally free $\co_{\mathrm{Gr}(r,\g g)}$-modules 
\begin{equation}\label{universal_grassmann}
0\aro \mathcal U\aro \co_{\mathrm{Gr}(r,\g g)}\ot\g g\aro \mathcal Q\aro0,
\end{equation} 
where $\mathcal U$ is the universal subbundle. Now, if $\beta:\wedge^2\mathcal U\to\mathcal Q$ stands for the composition  
\[ 
\wedge^2\mathcal U\aro \wedge^2\left( \co_{\mathrm{Gr}(r,\g g)}\ot\g g\right)\stackrel{[\cdot,\cdot]} \aro \co_{\mathrm{Gr}(r,\g g)}\ot\g g\aro
\mathcal Q,  
\] 
 we define $\mathbf{SLie}(r,\g g)$ as the closed subscheme of $\mathrm{Gr}(r,\g g)$ cut out by the vanishing of $\beta$.   Hence, in particular, a point  $V\in\mathrm{Gr}(r,\g g)$ lies in $\mathbf{SLie}(r,\g g)$ exclusively when $V$ is a subalgebra of $\g g$. 
We call   $\mathbf{SLie}(r,\mathfrak g)$   the {\it scheme of $r$-dimensional Lie subalgebras of $\mathfrak g$} \cite[Sections 3 and 4]{richardson67}. 
Since we shall be mostly concerned with topological properties of $\mathbf{SLie}(r,\g g)$, it is convenient to  denote by  $\SLie(r,\g g)$ the {\it  reduced subscheme  of $\mathbf{SLie}(r,\g g)$} associated to $\mathbf{SLie}(r,\g g)$. 

The tangent space of $\Gr(r,\mathfrak g)$ at a point    representing a    subspace  $V$ (recall that points are always closed at this part of the text) is naturally identified with
\[
\Hom(V, \mathfrak g/V). 
\]
When $V=\mathfrak h$ is a subalgebra of $\mathfrak g$ then the adjoint action of $\mathfrak g$ on itself induces
a $\mathfrak h$-module structure on $\mathfrak g/\mathfrak h$. If $(\mathrm{Hom}_\Bbbk(\wedge^\bullet\mathfrak h, \mathfrak g/\mathfrak h),d)$
denotes the Chevalley-Eillenberg complex of the $\mathfrak h$-module $\mathfrak g/\mathfrak h$ \cite[Corollary 7.7.3, p.240]{weibel} then, by definition,
$\Hom(\mathfrak h,\mathfrak g/\mathfrak h) = C^1(\g h, \mathfrak g/\mathfrak h)$. As shown by Richardson in \cite[Proposition 6.1]{richardson67}, the Zariski tangent space of $\mathbf{SLie}(r,\mathfrak g)$ at $\mathfrak h$ equals \[Z^1(\mathfrak h, \mathfrak g / \mathfrak h)=\{\delta:\g h\to\g g/\g h\,:\,\delta([xy])=x\delta(y)-y\delta(x)\},\] the subsapce of all derivations \cite[7.4.3]{weibel} or crossed homomorphisms. 

Now,   the Zariski tangent space of the orbit of $\mathfrak h \in \mathbf{SLie}(r,\g g)$ under the action of $G$ (induced by $\mathrm{Ad}:G\to\mathrm{GL}(\g g)$) equals, according to \cite[Proposition 7.1]{richardson67}, the space of coboundaries or inner derivations  \[B^1(\mathfrak h, \mathfrak g/ \mathfrak h)=\{dm:\g h\to\g g/\g h\,:\,m\in\g g/\g h\},\]
where $dm(x)=xm=[x,m]\mod\g h$ cf. \cite[7.4.4]{weibel}. 
 See  also \cite{crainic-schatz-struchiner14}  for a recent presentation of the results quoted above.

\subsection{Commuting varieties}
Let $\mathfrak g$ be a finite-dimensional   Lie algebra and let $r\ge 2$ be an integer. The variety of
commuting $r$-tuples of $\mathfrak g$ is, by definition,
\[
    \Commuting_r(\mathfrak g) = \left\{ (x_1, \ldots, x_r) \in \mathfrak g^r \, ; \, [ x_i , x_j ]= 0,\text{for all  $1 \le i,j\le r$}\right\} \, .
\]
Clearly, $\Commuting_r(\mathfrak g)$ is a closed   subscheme of the affine space  $\mathfrak g^r$ and questions about its nature immediately come to the fore. 
Motzkin and  Taussky \cite[Theorem 5]{motzkin-taussky55} were the first to show, when $\Bbbk=\C$ and by using the Euclidean topology,   that {\it   $\Commuting_2(\mathfrak {gl}_n)$ is  irreducible}.
Gerstenhaber, apparently unaware of the work by Motzkin and Taussky,   extended this result to {\it  arbitrary algebraically closed fields}, see \cite[Ch. II, Theorem 1]{gerstenhaber61}. In \cite[p. 342]{gerstenhaber61}, Gerstenhaber explains that the irreducibility of  $\Commuting_2(\mathfrak{gl}_n)$ was conjectured by Goto, who also observed that {\it $\Commuting_r(\mathfrak{gl}_n)$ has more than one irreducible component when $4\le r<n$}. More recently, Guralnick \cite{guralnick} proved that $\Commuting_r(\mathfrak{gl}_n)$  {\it cannot be irreducible}
when $n\ge 4$ and $r\ge 4$  or when  $n=3$ and $r \ge 32$.  The above list of works is a small sample of those discussing   irreducibility of  
 $\Commuting_r(\mathfrak{gl}_n)$ and, to the best of our knowledge, the complete list of  values $(r,n)$ for which $\Commuting_r(\mathfrak{gl}_n)$ is irreducible is still unknown. Although the study of commuting varieties has a classical flavour, it is an active area of research still today  with
a number of open questions.
For Lie algebras other than $\g {gl}_n$,  Richardson's result  below (cf. \cite[Corollary 2.5]{richardson79}) is fundamental  and  is a key ingredient   in the proof of Theorem \ref{THM:A}.

\begin{thm}[Richardson]\label{T:Richardson}
    If $\mathrm{char}(\Bbbk)=0$ and $\mathfrak g$ is   reductive, then
    $\Commuting_2(\mathfrak g)$   is irreducible. 
\end{thm}

In the following Lemma, we shall require the notion of semi-simple and nilpotent elements in a semi-simple Lie algebra \cite[5.4 and 6.4]{hum}. Hence, from now on, we \[\text{{\it suppose that $\mathrm{char}(\Bbbk)=0$}.}\]
For brevity, we   let $\g g_{\rm nil}$ stand for the subset of nilpotent elements and $\g g_{\rm ss}$ for  the set of semi-simple ones. 

\begin{lemma}\label{L:is abelian}
    Let $\mathfrak g$ be a semi-simple   Lie algebra over the algebraically closed field $\Bbbk$ of characteristic zero. If $\mathfrak h \subset \mathfrak g$ is a subalgebra  such that $\g g_{\rm{nil}}\cap \g h=\{0\}$,   then $\mathfrak h$ is abelian.
\end{lemma}
\begin{proof}We first prove that $\g h$ is solvable by assuming otherwise and arriving at a contradiction. 
This being so,  $\g h$  contains  a semi-simple subalgebra  $\g r\not=0$     (a Levi subalgebra \cite[Theorem 7.8.13]{weibel}). If $\g b\subset \g r$ is a Borel subalgebra (cf. \cite[16.3]{hum} or \cite[VIII.3.3]{bourbaki_lie}), then $[\g b,\g b]\not=0$ only contains nilpotent elements of $\g r$. Since $\g r_{\rm nil}\subset\g g_{\rm nil}$, we conclude that $[\g b,\g b]=0$, which forces $\g r$ to have no roots. This contradiction shows that $\g r$ cannot exist and $\g h$ is solvable. Now, let $\g b$ 
be a  Borel subalgebra     
    of $\mathfrak g$ containing $\g h$; since $[\g b,\g b]$ only contains nilpotent elements we conclude that $[\g h,\g h]=0$. 
\end{proof}

\begin{cor}\label{06.02.2024--1}
   Let  $\mathfrak g\not=\g{sl}_2$ be a semi-simple   Lie algebra over the algebraically closed field of characteristic zero $\Bbbk$. Let $\Sigma$ be the set of all points $\g h\in\SLie(2,\g g)$ for which $\g h$ is abelian. Then  $\Sigma$ is closed and is an irreducible component of $\SLie(2,\g g)$.   
\end{cor}
\begin{proof}First, as we exclude $\g {sl}_2$, the set 
$\Sigma$ is non-empty (a maximal toral subagebra will have always dimension at least two). 
Adopting the notations employed in eq. \eqref{universal_grassmann}, $\Sigma$ is the  set where the composition
\[
\wedge^2\mathcal U\aro \wedge^2\left(\mathcal O_{\Gr(r,\g g)}\ot\g g\right)\aro \mathcal O_{\Gr(r,\g g)}\ot\g g 
\]
vanishes and is hence closed. 
    Theorem \ref{T:Richardson} implies that $\Sigma$ is also irreducible. It remains to show that $\Sigma$ is an irreducible component
    of $\SLie(2,\mathfrak g)$, which is done by proving that $\Sigma$ contains an open and non-empty subset of $\SLie(2,\g g)$. For that, note that $A=\{V\in\mathrm{Gr}(2,\g g)\,:\,V\cap\g g_{\rm nil}=0\}$ is open because of Bertini's theorem and the fact that the dimension of the closed algebraic subset   $\g g_{\rm nil}$ is $\dim \g g-\mathrm{rank}\,\g g$ \cite[Theorem 4.1]{mcgovern02} and this is $\le\dim \g g-2$. Then, employing
      Lemma \ref{L:is abelian}, we see that  $\SLie(2,\g g)\cap A$ is contained in $\Sigma$ and is non-empty.
\end{proof}

\subsection{Two dimensional non-abelian subalgebras and nilpotent orbits}\label{10.06.2024--1}
We keep on assuming that $\Bbbk$ is algebraically closed and of characteristic zero and $\g g$ is a semi-simple Lie algebra over $\Bbbk$. We now analyse the structure of the non-abelian elements of $\SLie_2(\g g)$. 

Let  $\g h\in \SLie_2(\g g)$ be non-abelian.  As is well-known \cite[1.4]{hum}, such an algebra is spanned by $\{e,h\}$  satisfying $[h,e]=2e$.
  Moreover, as $\g g$ is assumed semi-simple, then $\g h$ must contain a non-zero nilpotent  element $n$ of $
\g g$ (Lemma \ref{L:is abelian}). 
Using the  nilpotency of $\mathrm{ad}_n:\g h\to \g h$, it is not difficult to see that $n\in\Bbbk e$ and hence 
\begin{equation}\label{01.01.2024--1}
\Bbbk n=[\g h,\g h]=\Bbbk e.
\end{equation}
In particular, $e\in\g g_{\rm nil}$ and in fact  $\g g_{\rm nil}\cap \g h=[\g h,\g h]$.

Let us fix a connected linear algebraic group $G$ of adjoint type, i.e. a linear connected linear algebraic group    $G$ with Lie algebra $\g g$ and such that $\mathrm{Ad}:G\to \mathrm{GL}(\g g)$ is  a closed immersion. We observe that $\Bbbk n\setminus\{0\}=[\g h,\g h]\setminus\{0\}$ is contained in single $G$-orbit, which we call \[\mathrm{Orb}(\g h).\]
(This is a consequence of   \cite[proof of 4.3.1]{collingwood-mcgovern93},  or of the Jacobson-Morozov Theorem \cite[VIII.11.2]{bourbaki_lie} plus  \cite[VIII.11.3, Proposition 6]{bourbaki_lie}.)  

For the sake of uniformity of notation, for any given {\it abelian} $\g a\in\SLie(2,\g g)$, we let 
\[\mathrm{Orb}(\g a) = \text{The orbit \{0\}}.\]

\begin{thm}[Theorem \ref{THM:A} of the Introduction]\label{20.06.2024--2} 
    Let $\mathfrak g\not=\g{sl}_2$ be a   semi-simple Lie algebra over $\Bbbk$. The irreducible components of $\mathbf{SLie}(2,\mathfrak g)$  are in bijection with
the set of conjugacy classes of nilpotent elements of $\mathfrak g$ under the action of $G$.
More precisely, the bijection is obtained from the map  \[\mathrm{Orb}:\SLie(2,\g g)\longrightarrow \{\text{Nilpotent orbits}\}\] constructed above.
\end{thm}
\begin{proof} 
Let us abbreviate $\SLie(2,\g g)$ to $L$ and  let $L^{\rm na}=\{\g h\in L\,:\,  [\g h,\g h]\not=0\}$ be the (quasi-projective algebraic)  subset of non-abelian subalgebras. We already know that $L\setminus L^{\rm na}$ is an irreducible component (Corollary  \ref{06.02.2024--1}) and we are left with showing that $\mathrm{Orb}$ establishes a bijection between the irreducible components of $L^{\rm na}$ and the non-trivial nilpotent orbits.

Let $N$ stand for the    set of   nilpotent orbits in $\g g\setminus\{0\}$ and for each $n\in N$, we pick an element $e_n\in n$; note that  $e_n\not=0$. A fundamental result  in the theory says that $N$ is {\it finite} (cf. \cite[Corollary 3.2.15]{collingwood-mcgovern93} or \cite[Theorem 3.11]{mcgovern02}). 
Define  
\[A_n:=\{x\in \g g     \,:\,[x,e_n]=2e_n\};\]
this is a linear subspace of $\g g$ and by the  Jacobson-Morozov Theorem (cf. \cite[VIII.11.2]{bourbaki_lie} or   \cite[Theorem 3.4]{mcgovern02}) it is also non-empty.  (Note that $A_n$ is an affine space over the vector space $Z_{\g g}(e_n)$.)
Let now 
\[
\varphi_n:G\times A_n\longrightarrow L^{\rm na}
\]
be the map $(g,x)\mapsto \Bbbk g(x)+\Bbbk g(e_n)$. Note that $\mathrm{Orb}(\varphi_n(g,x))=n$. 
Then  
\[L^{\rm na}=\bigsqcup_{n\in N}\mathrm{Im}(\varphi_n).\] 
Indeed, if $\g h\in L^{\rm na}$, then $[\g h,\g h]=\Bbbk g(e_n)$ (see the discussion around eq. \eqref{01.01.2024--1}) for some $g\in G$ and $n\in N$. Now there exists $x\in \g h$ such that $[x,g(e_n)]=g(e_n)$; then $y:=g^{-1}(x)\in A_n$ and $\g h=\Bbbk  g(y)+\Bbbk  g(e_n)\in\mathrm{Im}(\varphi_n)$. Since each $G\times A_n$ is irreducible, the same is true about the constructible set 
\[
B_n:=\mathrm{Im}(\varphi_n).
\] It takes little effort to see that,  
\[
B_n=\{\g h\in L^{\rm na}\,:\,\mathrm{Orb}(\g h)=n\}.
\]

Let $\Sigma \subset L^{\rm na}$ be an irreducible component. From   $\Sigma=\cup_n\Sigma\cap\overline B_n$, it follows that 
\[
\text{  $\Sigma= \overline B_{\gamma(\Sigma)}$ for a certain $\gamma(\Sigma)\in N$.}
\] Consequently, constructibility assures that $B_{\gamma(\Sigma)}$ contains an open and dense subset of $\Sigma$   \cite[$0_{\rm III}$. 9.2.3]{ega}. In addition, if $\nu\in N\setminus\{\gamma(\Sigma)\}$, then $B_{\nu}$ cannot contain an open dense subset of $\Sigma$:  otherwise $\Sigma\cap B_{\nu}\cap B_{\gamma(\Sigma)}\not=\varnothing$. In conclusion, there exists {\it a unique $\gamma(\Sigma)\in N$ such that $B_{\gamma(\Sigma)}$  contains an open and dense subset of $\Sigma$}, and we arrive at a map
\[
\gamma:\left\{\begin{array}{c}\text{irreducible  }\\  \text{components of $L^{\rm na}$}\end{array}\right\}\longrightarrow N.
\]   

We note that  $B_{\gamma(\Sigma)}\subset\Sigma$ and that 
$\Sigma$ is   the {\it only irreducible component of $L^{\rm na}$ containing 
it}.  This shows as   that   the map $\gamma$ is injective. 

To prove surjectivity, it suffices to show that each $B_n$ contains an open and non-empty subset. 
This is the reason for the next result. 

\begin{prop}\label{26.06.2024--1}   Let   $e$ be any non-zero nilpotent element of $\mathfrak g$ in the orbit $n$. Let $(e, h,f)\in\g g$ be an $\g{sl}_2$-triple, i.e. 
\[        [h,e] = 2e, \quad [h,f] = -2f, \quad [e,f] = h \, .\]
 (The existence is assured by the Jacobson-Morozov theorem \cite[VIII.11.2]{bourbaki_lie}.)       Let $\mathfrak h=\Bbbk h+\Bbbk e$. Then there exits a smooth $\Bbbk$-scheme $U$ and a smooth morphism $\psi:U\to \mathbf{SLie}(2,\g g)$ whose image is contained in $B_n$.  
\end{prop}  
\begin{proof}
Let $\mathfrak s=\Bbbk e+\Bbbk f+\Bbbk h$. (Clearly $\g s\simeq \g{sl}_2$.)
  We will describe the infinitesimal deformations of $\mathfrak h$ inside $\mathfrak g$ and show that they are all realized by actual deformations. From Richardson's description of the tangent spaces in $\mathbf{SLie}(2,\g g)$ (see Section \ref{01.01.2024--2}), we are required to study  $Z^1(\g h,\g g/\g h)$.

Let us give $\g g$ and $M:=\g g/\g h$ the obvious structures of $\g s$-module and of $\g h$-module. We write  $x_M$  for the    endomorphism of $M$ obtained from the action of     $x\in \g h$ on it. 
From the standard representation theory of $\g{sl}_2$ \cite[7.2]{hum}, we know that
$h$ acts in a diagonal way on $\g g$  having only integer eigenvalues. Since $\g h\subset\g g$ is invariant under $\mathrm{ad}_h$, it follows that $h_M$ acts diagonally on $M$ with only integer eigenvalues.  Hence, for each $i\in\Z$, let $\g g_i$, resp. $M_i$, stand for the eigenspace associated to the eigenvalue $i$.

We now define an injective  linear map $\delta:M^{\g h}\to Z^1(\g h,M)$ by \[\delta_x(e)=0\quad\text{and}\quad \delta_x(h)=x.
\]
We claim that $\delta$ induces an isomorphism 
\[
M^{\g h}\simeq H^1(\mathfrak h, M).\]
Since injectivity is straightforward, we move on to   verify surjectivity. Let $\varphi \in Z^1(\mathfrak h, M)$ be an arbitrary  cocyle.  If $\varphi(h) = y_0 + \sum_{i\neq 0} y_i$ with $y_i \in M_i$ then, after replacing
    $\varphi$ by $\varphi - d(\sum_{i\neq 0} i^{-1} y_i)$, we can assume that $\varphi(h) \in M_0$. 
    Since $\varphi\in \mathrm{Der}(\g h,M)$, we have 
    \begin{align*}2\varphi(e) &=\varphi([h,e])\\&= 
  h_M\left( \varphi(e)\right)    -\underbrace{ e_M \left(\varphi(h)\right)}_{\in M_2},
\end{align*}
and    we deduce that     $\varphi(e) \in M_2$; indeed, writing  $\varphi(e)=\sum x_i$ with $x_i\in M_i$, then $\sum (i-2)x_i\in M_2$, which proves that $x_i=0$ unless $i=2$. Under this light,   equation $2\varphi(e)=h_M(\varphi(e))-e_M(\varphi(h))$   shows that $\varphi(h)$ in addition belongs to $\mathrm{Ker}\,e_M$, so that $\varphi(h)\in M^{\g h}$.
 Since the composition 
 \[
 \g g_2\stackrel{ \mathrm{ad}_f }\longrightarrow    \g g_0\stackrel{\mathrm{ad}_e}\longrightarrow \g g_2
 \] is multiplication by 2, the map    $e_M :M_0 \to M_2$ is  surjective. Hence, if $x\in M_0$ satisfies   $e_Mx=\varphi(e)$, then $\varphi - dx \in Z^1(\mathfrak h, M)$ is such that $(\varphi-dx)(h) =\varphi(h)\in M^{\g h}$ and $(\varphi-dx)(e) = 0$. It is then clear that
    $\delta:M^{\g h}\to Z^1(\g h,M)$ induces an  isomorphism onto $H^1(\g h,M)$. 

Let $\pi:\g g\to M$ be the projection. Since $\pi|_{\g g_0}:\g g_0\to M_0$ is surjective (with kernel $\Bbbk h$), there exists a subspace $V\subset \g g_0$  such that $\pi|_V:V\to M_0$ induces an isomorphism   $V\stackrel\sim\to M^{\g h}$. This implies that for each $\varphi\in Z^1(\g h,M)$, there exists $v\in V$ and  $x\in M$ such that $\varphi=\delta_{\pi(v)}+dx$. 

Now, if $v\in V$, it is not difficult to see that $[v,e]\in   \Bbbk e$ so that $\Bbbk e+\Bbbk (h+v)$ is a  subalgebra. Let $\psi:G\times V\to \SLie(2,\g g)$ be the morphism  $(g,v)\mapsto \Bbbk g(e)+\Bbbk g(h+v)$;   if $[v,e]+2e\not=0$, then $\mathrm{Orb}(\psi(g,v))$ is $Ge$, the orbit of $e$. Let $V'=\{[v,e]+2e\not=0\}$ and let   $\psi'$ be the restriction to $G\times V'$.  It follows that $\mathrm{Im}(\psi')\subset B_n$  and that $\psi'(1_G,0)=\g h$.  Also, the differential $\mathrm D_{(1_G,0)}\psi :T_{1}G\times T_0V\to T_{\g h}\mathbf{SLie}(2,\g g)$ is surjective. Since $G\times V'$ is a smooth $\Bbbk$-scheme, then $\mathbf{SLie}(2,\g g)$ is also smooth over $\Bbbk$ at $\g h$, and $\psi'$ is smooth at $(1_G,0)$. (These last claims are follow from an exercise which we were unable to find explicitly stated in the literature;  we move it to Lemma \ref{20.06.2024--1} below.) 
\end{proof}

\begin{lemma}\label{20.06.2024--1}Let $f:Y\to X$ be a morphism   between  algebraic $\Bbbk$-schemes and let $y\in Y$ be a $\Bbbk$-rational point of $Y$ with image $x$.  Assume that $Y$ is smooth and that  $D_yf:T Y|_{y}\to T X|_{x}$ is surjective. Then $X$ is smooth at $x$ and $f$ is smooth at $y$.
\end{lemma}
\begin{proof}
 Indeed, let $(R,\g m)$ and $(S,\g n)$
be the complete local rings at $x$ and $y$ respectively, and denote by $g:R\to S$ the associated map. By hypothesis, the induced $\Bbbk$-linear map $\g m/\g m^2\to\g n/\g n^2$ is injective. Let $\boldsymbol u=\{u_i\}\subset\g m$ induce a basis of $\g m/\g m^2$ and let $\boldsymbol v=\{v_j\}\subset \g n$ be defined in such a manner that $\{g(u_i),v_j\}$ induces a basis for $\g n/\g n^2$. Now, $\Bbbk\llbracket \boldsymbol u,\boldsymbol v \rrbracket=S$ and the natural map $\pi:\Bbbk\llbracket \boldsymbol U  \rrbracket\to R$ defined by $U_i\mapsto u_i$ is surjective. But since $g\pi$ is injective, we can be sure that $\pi$ is an isomorphism. Since $f$ is of finite type and  both $X$, and $Y$ are algebraic, all claims are proved.  \end{proof}

As is well-known, smooth morphisms are always open (see \cite{ega}, $\mathrm{IV}_4$, 7.5.1 and $\mathrm{IV}_2$, 2.4.6) and we have   concluded the proof of Theorem \ref{20.06.2024--2}. \end{proof}

\section{Irreducible components of $\mathrm{Fol}$ for  a projective space of dimension at least 4}

In what follows, all schemes in sight are algebraic over  $\C$ and points on these schemes are assumed  closed. Also, $\g{sl}_n(\C)$ is abbreviated to $\g{sl}_n$.

Using certain foliations on $\mathbb P^4$ of dimension 2,  we shall study the scheme of  foliations on $\mathbb P^{n}$ for   $n\ge4$ by employing the results of   Section \ref{Variety_of_Lie_subalgebras} and  \cite[Corollary 6.1]{cukierman-pereira}.

\subsection{Partitions and nilpotent classes}\label{21.06.2024--1} We begin by  recalling   well-known facts. Let $n\ge1$ be an integer and denote by $\mathcal P_n$ the set of partitions of $n$. Agreeing to write   
$J_\mu$ for   the Jordan block
\[
\begin{pmatrix}0&1& & & 
\\
&0&1& &
\\
&&\ddots&\ddots&
\\
&&&&1
\\&&&&0
\end{pmatrix}\in\g{sl}_\mu , 
\] 
let us then define,   for each partition        $\lambda=[\lambda_1,\ldots,\lambda_s]$ of $n$, 
the Jordan matrix   
\begin{equation}\label{04.05.2024--2}J_\lambda:=
\begin{pmatrix}J_{\lambda_1}&&&
\\
&J_{\lambda_2}&&
\\
&&\ddots&
\\
&&&J_{\lambda_s}\end{pmatrix}\in\g{sl}_n .
\end{equation}
With these notations, the map 
\[
\mathcal P_n\longrightarrow \begin{array}{c}\text{Conjugacy classes}\\ \text{ in $\left(\g {sl}_n \right)_{\rm nil}$}\end{array}
\]
\[\lambda\longmapsto \text{class of $J_\lambda$}\]
is a bijection \cite[3.1]{collingwood-mcgovern93}.  
Similarly,  let us introduce 
\[
H_\mu:=\begin{pmatrix}
\mu-1&&&
\\
&\mu-3&&
\\
&&\ddots&
\\
&&&-\mu+1\end{pmatrix}\in\g{sl}_\mu ,
\]
and,  write  
\[H_\lambda=\begin{pmatrix}H_{\lambda_1}&
\\
&\ddots
\\
&&H_{\lambda_s}
\end{pmatrix} \in\g{sl}_{n}
\]
for each $\lambda=[\lambda_1,\ldots,\lambda_s]\in\mathcal P_n$.
A direct computation shows that  $[H_\lambda , J_\lambda]=2J_\lambda$. 
It should be noted here that, for each $\mu\ge2$, the couple $\{J_\mu,H_\mu\}$  can be completed to an $\g{sl}_2$-triple $\{J_\mu,H_\mu,K_\mu\}$;
in like fashion, for each partition $\lambda$ of $n$, we arrive at an $\g{sl}_2$-triple $\{J_\lambda,H_\lambda,K_\lambda\}$. 

\subsection{Vector fields in projective space associated to certain partitions}\label{07.05.2024--4} In what follows, let su denote the tangent sheaf of $\PP^n$ by $T$. 
Letting $\mathrm{SL}_{n+1}$ act on the left of $\PP^n$ in standard fashion, we can construct from   $v\in \g{sl}_{n+1}$  a vector field   \[v^\natural\in H^0(\PP^n,T ); \]  the fundamental vector field associated to $v$. Concretely, if $p\in \PP^n$ is a point, then $v(p)\in T(p)$ is the image of $v$ under the orbit map $\mathrm{orb}_p:\mathrm{SL}_{n+1}\to \PP^n$. Once we endow $\PP^n$ with homogeneous coordinates $\{x_i\}_{i=1}^{n+1}$, 
it is not difficult to see that the canonical   matrices $E_{ij}\in\g{sl}_{n+1}$ give rise to the fields $\displaystyle x_i\frac{\partial}{\partial x_j}$. From the Euler exact sequence, we know that \[\g {sl}_{n+1}\stackrel\sim\longrightarrow H^0(\PP^n,T);\]
we shall then identify $H^0(P,T)$ with $\g {sl}_{n+1}$.

Let now $n\ge4$      and denote by  $\delta$ the difference $n-4$. Inside $\mathcal P_{n+1}$, let   
\[
\mathcal P_{n+1}'=\{\lambda\in\mathcal P_{n+1}\,:\,\text{$\lambda$ has one part equal to 5}\}.\]  Note that   $\#\mathcal P_{n+1}'=\#\mathcal P_{\delta}$. We shall be concerned with the fields   $J_\lambda^\natural$ and $H_\lambda^\natural$, where $\lambda\in \mathcal P_{n+1}'$ and  $J_\lambda$ and $H_\lambda$ are as in Section \ref{21.06.2024--1}.
 
\begin{lemma}\label{03.05.2024--1}
For any partition $\lambda\in\mathcal P'_{n+1}$ and any non-zero $(\alpha,\beta)\in\mathbb C^2$, we 
have 
\[
\mathrm{dimension}\,\mathrm{Sing}(\alpha J_\lambda^\natural+\beta H_\lambda^\natural)\le\delta.
\]  
\end{lemma}
\begin{proof}We deal first with the case  $\delta=0$. After fixing homogeneous  coordinates $x_0,\ldots,x_4$, we can write 
\[
J_5^\natural=x_1\frac{\partial}{\partial x_0}+x_2\frac{\partial}{\partial x_1}+x_3\frac{\partial}{\partial x_2}+x_4\frac{\partial}{\partial x_3}
\]
and 
\[H_5^\natural=4x_0\frac{\partial}{\partial x_0}+2x_1\frac{\partial}{\partial x_1}- 2x_3\frac{\partial}{\partial x_3}-4x_4\frac{\partial}{\partial x_4}.\]
It is a simple matter to check that $(1:0:\ldots:0)$ is the only singular point of $J_5^\natural$ and that the image of the five canonical vectors of $\C^5$ is the singular set of $H_5^\natural$. 

Let us then study the singular set of $\alpha J_5^\natural+H_5^\natural$ for varying $\alpha$. For that, we introduce the vector field on $\mathbb P_{\mathbb Q[t]}^4$ defined by 
\[
t J_5^\natural+H_5^\natural = \sum_{0}^4A_i(t)\frac{\partial}{\partial x_i}. 
\]
Let  
\[
I=\text{$2\ti2$ minors of } \begin{pmatrix}x_0&\cdots&x_4
\\
A_0(t)&\cdots&A_4(t)\end{pmatrix},
\]
and denote by $Z\subset\mathbb P_{\mathbb Q[t]}$ the closed subscheme cut-out by $I$. 
With the help of computer algebra, we verify that the comma ideal $(I:t)\subset\mathbb Q[t,x_0,\ldots,x_4]$ is identically zero and hence that $Z$ is a flat $\mathbb Q[t]$-scheme.  Another application of computer algebra shows that $\dim Z\otimes\mathbb Q(t)=0$ and hence $\dim Z_\alpha=0$ for all $\alpha\in\mathrm{Spec}\,\Q[t]$ \cite[$\mathrm{IV}_3$, 14.2.4]{ega}. The Lemma is then verified in the case where $\delta=0$.

Let us now suppose that $\delta\ge1$.
We pick homogeneous  coordinates $x_0,\ldots,x_4$, $y_1,\ldots,y_\delta$ on $\mathbb P^n$ such that  
\[
J^\natural_\lambda=x_1\frac{\partial}{\partial x_0}+x_2\frac{\partial}{\partial x_1}+x_3\frac{\partial}{\partial x_2}+x_4\frac{\partial}{\partial x_3}+\sum_{i=1}^\delta M_i \frac{\partial}{\partial y_i}, 
\]
with $M_i\in\C[y_1,\ldots,y_\delta]$ homogeneous and linear, and 
\[
H^\natural_\lambda=4x_0\frac{\partial}{\partial x_0}+2x_1\frac{\partial}{\partial x_1}- 2x_3\frac{\partial}{\partial x_3}-4x_4\frac{\partial}{\partial x_4}+\sum_{i=1}^\delta L_i \frac{\partial}{\partial y_i},
\]
with $L_i\in\C[x_0,\ldots,x_4]$ homogeneous and linear. Then $\alpha J_\lambda^\natural+\beta H_\lambda^\natural$ is of the form 
 \[\sum_{i=0}^4A_i\frac{\partial}{\partial x_i}+\sum_{i=0}^\delta B_i\frac{\partial}{\partial y_i}\]
 with $A_i\in\C[x_0,\ldots,x_4]$  and $B_i\in\C[y_1,\ldots,y_\delta]$ linear. 
In addition, the singularity set of the field $\sum A_i\frac{\partial}{\partial x_i}$ in $\mathbb P^4$ is zero dimensional, as we showed above.

 Let $Z$ be an irreducible component of $\mathrm{Sing}(\alpha J_\lambda^\natural+\beta H_\lambda^\natural)$,   $\Pi$ be the linear subspace  $\{x_0=\cdots=x_4=0\}$ and  \[\Phi:\mathbb P^n\setminus\Pi\longrightarrow \mathbb P^4
 \] 
be the projection centred at  $\Pi$.
If $Z\subset\Pi$, then $\dim Z\le \delta-1$. Otherwise, $Z^\circ=Z\setminus\Pi$ is irreducible of dimension $\dim Z$. Now, given $p\in Z^\circ$, it follows that $\Phi(p)$ is a point of the singular set of $\sum A_i\frac{\partial}{\partial x_i}$. Hence, $\Phi(Z^\circ)$ is a single point and $Z^\circ$ is contained in a fibre. Now the fibres of $\Phi$       are isomorphic to $\mathbb A^{\delta}$ and hence     $\dim Z^\circ\le\delta$. In conclusion, $\dim Z\le\delta$ in each scenario. 
\end{proof}

\begin{cor}\label{06.05.2024--1}For each partition $\lambda\in \mathcal P_{n+1}'$, let $\g h_{\lambda}\subset \g{sl}_{n+1}$ be the subalgebra  $\C J_\lambda+\C H_\lambda$. Let 
\[
\gamma:  \co\ot\g h_\lambda  \longrightarrow T_{\mathbb P^n}
\] 
be deduced from the natural morphism $\g {sl}_{n+1}\to H^0(\mathbb P^n,T_{\mathbb P^n})$. Then the singular set of $\gamma$ 
\[
\mathrm{Sing}(\gamma)=\left\{p\in\mathbb P^n\,:\,\begin{array}{c}\text{$\gamma(p):\left(\co\ot\g h_\lambda \right)(p)\to T(p)$}
\\ 
\text{fails to be injective}
\end{array}
\right\} 
\] 
has dimension at most $n-3$.
\end{cor}
\begin{proof} 
Denoting   the punctured affine space $\g h_\lambda\setminus\{0\}$ by $H$, we consider the incidence variety with its projections 
\[
\xymatrix{&\ar[dl]_-{\varphi}\ar[dr]^-{\psi} I=\{(v,p)\in H\times \mathbb P^n\,:\,v^\natural(p)=0\} & 
\\
H&&\mathbb P^n.} 
\]
(It is by direct verification that we assure that $I$ is a closed and reduced subscheme of the product.)
  Note that   $\psi (I)=\mathrm{Sing}(\gamma)$ and   that $\psi^{-1}(p)$ is either  a linear subspace deprived of its origin or  empty.
Let $Z$ be an irreducible component of $\mathrm{Sing}(\gamma)$ and let $J\subset I$ be an irreducible component such that $Z=\overline{\psi(J)}$.  Then $\dim J\ge\dim Z+1$ \cite[$\mathrm{IV}_2$, 5.6.8]{ega}. In particular, $\dim I\ge\dim \mm{Sing}(\gamma)+1$. Now, $\dim \varphi^{-1}(v)\le n-4$ for each $v\in H$ (Lemma \ref{03.05.2024--1}) and hence \cite[$\mathrm{IV}_2$, 5.6.7]{ega} assures that $\dim I\le   n-2$. Hence
$\dim \mm{Sing}(\gamma)\le n-3$.\end{proof}

\begin{cor}The morphism $\gamma:\co\ot\g h_\lambda \to T$ is injective and defines a singular foliation  on $\PP^n$ whose singular locus   has dimension at most $n-3$.
\end{cor}
\begin{proof}If $p\in U:=\PP^n\setminus\mathrm{Sing}(\gamma)$, then $(\co\ot\g h_\lambda)_p\to T_p$ is injective because of \cite[Theorem 22.5]{matsumura}.  If then follows that $\mathrm{Ker}(\gamma)$ is supported on a subscheme of dimension inferior to $n$, so that it must be zero because $\co\ot\g h_\lambda$ is pure. This proves that $\gamma$ is injective and we are left with verifying that $\mathrm{Im}\,\gamma$ is saturated. This follows from Proposition \ref{16.11.2023--1}-(\ref{16.11.2023--1-2}) because $U$ is of big and $\mathrm{Im}\,\gamma$ is certainly saturated over $U$ \cite[Theorem 22.5]{matsumura}.
\end{proof}

\subsection{Irreducible components of spaces of foliations}\label{06.05.2024--3}
As in Section \ref{07.05.2024--4} above, we keep on writing $T$ for the tangent sheaf of $\PP^n$. Also, for the sake of clarity in notation, we let $P$ stand for $\PP^n$ and   $\Omega^q$  for the sheaf of $q$-forms on $P$.  

 We set out to   study irreducible components of $\mathrm{Fol}_2(P,2)= \mathrm{Fol}_2(P ,\co(n+1))$. (The definition of this quasi-projective scheme is given in Section \ref{20.05.2024--3}.)  
This shall be done by combining Theorem \ref{THM:A} and one of the main results in \cite{cukierman-pereira}, which we like to call ``an openness theorem'': 

\begin{thm}[See {\cite[Theorem 2]{cukierman-pereira}}]\label{CP-theorem} 
Let $\varphi\in \mathrm{Fol}_r(P,m)$ be such that:
\begin{enumerate}
\item The $\co_{P}$-module $\mm{Ker}_\varphi$ is {\it free}. (For $\mathrm{Ker}_\varphi$, see Section \ref{formstodistributions}.)
\item The dimension of $\mathrm{Sing}(\varphi)$ is at most $n-3$. 
\end{enumerate} 
Then for each $\omega$ in a neighbourhood of $\varphi$ in $\mathrm{Fol}_r(P,m)$, the tangent sheaf $\mm{Ker}_\omega$ is also free.  
\end{thm}

Let now $n\ge4$ and recall from Section \ref{07.05.2024--4}  that $\mathcal P'_{n+1}$ stands for  the set of all partitions having a summand equal to five.

\begin{lemma}\label{07.05.2024--1}If   $\lambda\in\mathcal P_{n+1}'$, then $J_\lambda^\natural\wedge H_\lambda^\natural\not=0$.\end{lemma}
 \begin{proof}
We assume otherwise; then for each $p\in P$,   there exists $(\alpha,\beta)\in\C^2\setminus\{0\}$ such that $(\alpha J_\lambda^\natural+\beta H_\lambda^\natural)^\natural(p)=0$. This contradicts Corollary \ref{06.05.2024--1}.
\end{proof}

Let us keep assuming that $\lambda$ is an element of $\mathcal P_{n+1}'$, and let us write 
\[\g h_\lambda=\C J_\lambda+\C H_\lambda;\]
this is a non-abelian subalgebra of $\g{sl}_{n+1}$. 

We define a linear map 
\[\Psi:
\wedge^2  \g {sl}_{n+1}\aro H^0(P,\wedge^2T )
\]
  by $v\wedge w\to v^\natural\wedge w^\natural$, and  from this we obtain a morphism  
\[
\mathbb P\Psi: \mathbb P\left(\wedge^2\g {sl}_{n+1} \right)\setminus\mathbb P(\mathrm{Ker}(\Psi))\aro \mathbb P\left(H^0(P,\wedge^2T )\right).
\]
 Let  
\[\pi:\mathrm{Gr}(2,\g {sl}_{n+1})\aro    \mathbb P\left(\wedge^2\g {sl}_{n+1}\right)\]
stand   for the Pl\"ucker immersion 
 and define the morphism $\psi$ by means of the diagram  
\[
\xymatrix{\mathrm{Gr}(2,\g {sl}_{n+1})\setminus\pi^{-1}\PP(\mathrm{Ker}(\Psi))
\ar[rr]^-{\pi}\ar@/_2pc/[drr]_\psi&&\PP\left(\wedge^2\g {sl}_{n+1}\right)\setminus\PP(\mathrm{Ker}(\Psi))\ar[d]^{\PP\Psi}
\\
&&\PP H^0(P,\wedge^2T).
}
\]
Lemma \ref{07.05.2024--1} tells us that 
\[
\g h_\lambda \in \mathrm{Gr}(2,\g {sl}_{n+1})\setminus\pi^{-1}\PP(\mathrm{Ker}(\Psi)). \]
Through the identification   $\wedge^2T\stackrel\sim\to \Omega^{n-2} \ot \det T$, the morphism $\psi$  gives rise to a morphism, denoted likewise, 
\[
\psi:\mathrm{Gr}(2,\g {sl}_{n+1})\setminus\pi^{-1}\PP(\mathrm{Ker}(\Psi)) \aro\mathbb P  H^0(P,\Omega^{n-2}\ot \det T ).
\]
Let now $\mathrm{IF}_2(P,\det T)$ be the closed subscheme of $\mathbb P  H^0(P,\Omega^{n-2}\ot \det T )$ consisting of integrable forms. 
A  simple verification shows that the restriction of  $\psi$ to $\SLie(2,\g {sl}_{n+1})$ (recall that this is a closed and reduced subscheme of $\mm{Gr}(2,\g{sl}_{n+1})$) 
factors through $\mathrm{IF}_2(P,\det T)$
and we arrive at a morphism 
\[\psi:\SLie(2,\g{sl}_{n+1})\setminus\pi^{-1}\PP(\mm{Ker}(\Psi))\aro
\mathrm{IF}_{2}(P,\det T ).
\]
We let $\SLie(2,\g{sl}_{n+1})^\circ$ be the open subset of $\SLie(2,\g{sl}_{n+1})\setminus\pi^{-1}\PP(\mm{Ker}(\Psi))$ obtained as the inverse image of $\mathrm{Fol}_2(P,\det T)$ via $\psi$. 
Note that $\g h_\lambda\in \SLie(2,\g{sl}_{n+1})^\circ$. Given $\g h\in \SLie(2,\g{sl}_{n+1})^\circ$, the form $\psi(\g h)\in\mathrm{Fol}_2(P,\det T)$ is the   twisted $(n-2)$-form with values on $\det T$ associated to the foliation $\g h\ot\co\subset T$. Note that in this case $\g h=H^0(P,\mathrm{Ker}_{\psi(\g h)})$.

From the fact that $(J_\lambda,H_\lambda)$ is part of an $\g{sl}_2$-triple $(J_\lambda,H_\lambda,K_\lambda)$, we know from Proposition \ref{26.06.2024--1}
that $\g h_\lambda$ belongs to a unique irreducible component $\Sigma_\lambda$ of $\SLie(2,\g {sl}_{r+1})$. 
Let $\Sigma_\lambda^\circ=\SLie(2,\g{sl}_{n+1})^\circ\cap \Sigma_\lambda$.

\begin{thm}\label{07.05.2024--2}For each $\lambda\in\mathcal P_{n+1}'$, the closed subscheme    $\overline{\psi(\Sigma^\circ_\lambda)}$ is an irreducible component of $\mathrm{Fol}_{2}(P,\det T)=\mm{Fol}_2(P,2)$. 
\end{thm}
\begin{proof}Let $Z$ be an irreducible component of $\mm{Fol}_2(P,2)$ containing $\overline{\psi(\Sigma_\lambda^\circ)}$. 
We want to show that $\overline{\psi(\Sigma_\lambda^\circ)}$ contains an open   and non-empty subset of $Z$. 

We know that $\codim\mm{Sing}(\psi(\g h_\lambda))\ge3$ and hence there exists an open neighbourhood $U$ of $\psi(\g h_\lambda)$  such that   $\mathrm{Ker}_\omega\simeq\co^2$  for every $\omega\in U$ \cite[Corollary 6.1]{cukierman-pereira}. 
For any  $\omega\in  U$, the space $H^0(P,\mathrm{Ker}_\omega)$ is a subalgebra of $H^0(P,T)\simeq \g{sl}_{n+1}$ of dimension 2. 
It is not difficult to see that $H^0(P,\mathrm{Ker}_\omega)\not\in  \pi^{-1}\PP(\mathrm{Ker}(\Psi))$, which implies that   
\[
U\subset\psi(\SLie(2,\g{sl}_{n+1})^\circ).
\] 
In addition, if $\g h\in\SLie(2,\g{sl}_{n+1})^\circ$ belongs to $\psi^{-1}(U)$, then $H^0(P,\mathrm{Ker}_{\psi(\g h)})=\g h$.

Because $Z\cap U$ is an irreducible subspace, there exists an irreducible component $I$ of $\SLie(2,\g{sl}_{n+1})^\circ$ such that \[\tag{$\star$}Z\cap U\subset\overline{\psi(I)};\] in particular $\ov{\psi(\Si^\circ_\la)}\cap U\subset\overline{\psi(I)}$. Then 
  $\ov{\psi(\Sigma_\lambda^\circ)}\cap \psi(I)$ is   dense in $\ov{\psi(\Sigma_\lambda^\circ)}$ and hence  the constructible set $\psi(I)$ must contain an open and non-empty subset of $\ov{\psi(\Si_\la^\circ)}$ \cite[$0_{\rm III}$, 9.2.3]{ega}. Hence, there exists an open
subset  $V\subset\mm{Fol}_2$ such that  $\psi(I)\supset V\cap \psi(\Si_\la^\circ)\not=\varnothing$. Since $U\cap V\cap \psi(\Si_\la^\circ)\not=\varnothing$, we   assume that $V\subset U$. Now, for each $\g h\in \psi^{-1}(V)\cap \Sigma_\la^\circ$, the 
form $\psi(\g h)$ belongs to $\psi(I)$, and since $\psi(\g h)\in U$, it follows that $\g h\in I$. This shows that $I=\Sigma_\la^\circ$. From $(\star)$, we conclude that $\ov{\psi(\Sigma_\la^\circ)}$
contains the open and non-empty subset $Z\cap U$. 
\end{proof}

Using the theory of Hardy-Ramanujan (\cite[Corollary 3.1]{maroti} gives an interesting and explicit lower bound), we conclude that:

\begin{cor}For   $n\ge4$ we have 
\[\begin{split}\mathrm{ic}_2(\PP_\C^n,2)&\ge\#\,\text{partitions of $n-4$}
\\
&>\frac{e^{2\sqrt{n-4}}}{14}. 
\end{split}\]\qed
\end{cor}

From \cite[Corollary 5.1]{cukierman-pereira} we conclude that 
\begin{cor}Let $r\ge2$ be fixed. Then 
\[\lim_{n\to\infty}  \mm{ic}_r(\PP_\C^n,2)=+\infty.\]\qed
\end{cor}

\begin{question}Let $q\ge2$. Does $\displaystyle
\limsup_{n\to\infty}\,\mm{ic}^q(\PP^n_\C,2)$ exist?  
\end{question}

\begin{rmk}We constructed in this section a morphism of reduced schemes from an open subset of $\SLie(2,\g{sl}_{n+1})$ to $\mm{Fol}_2(\PP^n_\C,2)$ and it is instructive to note that it is not usually possible to  extend this morphism to the whole of $\SLie(2,\g {sl}_{n+1})$. Indeed, let $E_{ij}\in\g{gl}_{n+1}$ be defined by 
\[
E_{ij}(\vec  e_h)=\left\{
\begin{array}{ll}0&\text{if $j\not=h$,}
\\
\vec e_i&\text{if $j=h$},
\end{array}\right.\]
It is readily seen   that $\g h:=\C E_{11}+\C E_{12}$ is a non-commutative subalgebra but that the natural $\co$--linear morphism  $\gamma:\co\ot\g h  \to T$ is {\it not} injective. Indeed, we easily see that $E_{11}^\natural\wedge E_{12}^\natural=0$.  Also, it is possible that $\co\ot\g h\to T$ be injective but $\dim\mathrm{Sing}(\gamma)=n-1$.
\end{rmk}

\section{Distributions with a locally free tangent sheaf}\label{20.05.2024--1}

Let $f:X\to S$ be a smooth   morphism of Noetherian schemes and $\mathcal{V}$ a relative distribution on $X/S$ (Definition \ref{10.06.2024--3}). 
In what follows, for each Noetherian $S$-scheme $\tilde S$, we let $\mathcal V_{\tilde S}$ stand for the distribution on $X\times_S\tilde S/\tilde S$ obtained as the pull-back of $\mathcal{V}$ via the projection $\mathrm{pr}_X:X\times_S\tilde S\to X$ (Definition \ref{10.06.2024--4}). Similarly, for a coherent  $\mathcal{O}_X$-module  $\mathcal M$, we let $\restr{\mathcal M} {\tilde S }$ be $\mathrm{pr}_X^*\mathcal M$. Beware that $T_{\cv|_{s}}\not=\restr{T_\cv}s$.

\subsection{Constructibility results for set defined by the tangent sheaf}

\begin{prop}\label{31.10.2023--2}The set $\Phi$ of points $s$ in  $S$ where the tangent sheaf of the distribution $\cv_s$, the pull-back of $\cv$ to the fibre $X_s$, is \emph{locally free}
is constructible in $S$. 
\end{prop}
\begin{proof}
We employ \cite[$0_{\rm III}$, 9.2.3]{ega} in order to verify constructibility. For each integral and closed subscheme  $\tilde S\subset S$, we need to prove that $\Phi\cap \tilde S$ contains an open and non-empty subset, or that it is nowhere dense. 
Using  Corollary  \ref{transitivity_fibres}  and working with $\restr{\mathcal{V}} {\tilde S}$, we can, and do assume, that {\it $S$ is  itself integral} and set out to  verify that $\Phi$ either  contains an {\it  open and dense subset}, or that it is contained in a {\it proper closed subset}. 
Let us remark that under the present assumptions, $X$ is the disjoint union of integral schemes. Working with connected components, there is no loss of generality in supposing $X$ to be itself integral. In this case,   the $\co_X$-module $Q_\cv$ is torsion-free.

By generic flatness \cite[$\mathrm{IV}_3$, 8.9.4]{ega}, there exists an open and dense  subset $S'\subset S$ where  $\restr{Q_{\mathcal{V}}} {S'}$ is $\mathcal O_{S'}$-flat. (For notation, the reader should consult Definition \ref{10.06.2024--3}.) Hence, for $s\in S'$, we have   an exact sequence 
\[0\longrightarrow \restr{T_\cv}{s}\longrightarrow T_{X_s/\boldsymbol k(s)} \longrightarrow \restr{Q_\cv}{s}\longrightarrow 0 .
\]

Let $\eta\in S'$ be the generic point. Since $\restr{Q_\cv}{\eta}$ is a torsion-free $\co_{X_\eta}$-module, there exists an open neighbourhood  $S''$ of $\eta$ such that for all $s\in S''$ the $\co_{X_s}$-module $\restr{Q_\cv}{s}$ is torsion-free.  (The details, which are difficult, are worked out on p. 66, item $3^o$ of \cite[$\mathrm{IV}_3$]{ega}.)
Hence, for $s\in S''$, we have 
\[\tag{$\dagger$}T_{\mathcal{V}_{s}}=\restr{T_{\mathcal{V}}}s. 
\]
In particular, the following equality holds true:  
\[\tag{$\ddagger$}
\Phi\cap S''=\left\{s\in S''\,:\,\text{$\restr{T_\mathcal{V}}s$ is locally free}\right\}.
\]

 If        $\eta\in\Phi$ then, because of $(\ddagger)$ and \cite[$\mathrm{IV}_3$, 8.5.5, p. 23]{ega}, there exists an open and dense subset  $S^\circ\subset S''$ such that $T_\mathcal{V}|_{S^{\circ}}$ is a locally free $\mathcal O_{X_{S^{\circ}}}$--module. Hence, $S^{\circ} \subset \Phi\cap S''$.

 Otherwise, the point $\eta$ is not in $\Phi$; 
let $\Phi^c$ stand for $S\setminus \Phi$, so that  $\eta\in \Phi^c$. In this case, according to \cite[$\mathrm{IV}_3$, 9.4.7.1, p. 64]{ega},  there exists an open and dense subset $S^{\circ}$ of $S''$ such that for each $s\in S^\circ$, the $\co_{X_s}$-module $T_{\cv}|_{s}\simeq T_{\cv_s}$ fails to be flat. Hence $S^\circ\subset \Phi^c$, which proves that 
$\Phi\subset S\setminus S^\circ$.
\end{proof}

We now want place ourselves in the following {\it setting}:    $S$ is Noetherian  scheme over  $\Bbbk$ (recall that $\Bbbk$ is an arbitrary field) and  there exists a smooth $\Bbbk$-scheme  $M$ such that       $X=M\times_\Bbbk S$, and $f:X\to S$ is just the projection. 
Let $E$ be a vector bundle over $M$. Following a notation similar to the one in the beginning of the section, given a morphism $\tilde S\to S$,  we denote by $E_{\tilde S}$ the vector bundle $\mathrm{pr}_M^*(E)$ on $M\ti_\Bbbk \tilde S=X\times_S\tilde S$. In particular, for each point $s\in S$,  we have  $E_s\simeq E\otimes_\Bbbk\boldsymbol k(s)$.

\begin{prop}\label{28.11.2023--1}Assume $M$ to be \textbf{proper}  over $\Bbbk$. Then the set 
\[
\Phi_E=\{s\in S\,:\,T_{\cv_s}\simeq E_s\}.
\]
 is constructible. 
\end{prop}

\begin{proof} As in the proof of Proposition \ref{31.10.2023--2}, we  assume  that $S$ and $X$ are themselves integral and set out to   verify that $\Phi_E$ either  contains an {\it  open and dense subset}, or that it is contained in a {\it proper closed subset}.

By generic flatness \cite[$\mathrm{IV}_3$, 8.9.4]{ega}, there exists an open and dense  subset $S^\circ\subset S$ such that  $\restr{Q_{\mathcal{V}}}{S^\circ}$  and  $\restr{T_{\mathcal{V}}}{S^\circ}$  are $\mathcal O_{S^\circ}$-flat. Hence, for each $s\in S^\circ$ we have  
\[T_{\mathcal{V}_{s}}=\restr{T_{\mathcal{V}}}s. 
\]
In particular, $\Phi_E\cap S^\circ$ coincides with 
\[
\{s\in S^\circ\,:\,\text{$\restr{T_\mathcal{V}} {s}\simeq E\ot_k\boldsymbol k(s)$}\}.
\]
Let $\eta\in S^\circ$ be the generic point. If $\eta\in \Phi_E$, then
$\restr{T_\mathcal{V}}{\eta}\simeq E_\eta$ and 
 there exists an open and dense subset  $S^{\circ\circ}\subset S^\circ$ such that $\restr{T_\mathcal{V}}{S^{\circ\circ}}\simeq E_{S^{\circ\circ}}$ \cite[$\mathrm{IV}_3$, 8.5.2.5]{ega}. 
Consequently,   $S^{\circ\circ} \subset \Phi_E\cap S^\circ$.
Otherwise,    $\eta\not\in\Phi_E$; 
let $\Phi_E^c$ stand for $S\setminus \Phi_E$, so that  $\eta\in \Phi_E^c$. 
Let \[p:{\bf Isom}(T_\cv,E_S)\aro S\] stand for the scheme of isomorphisms between $T_\cv$ and $E_S$; this is an affine scheme over $S$, as can be deduced from the proof of  \cite[4.6.2.1]{laumon-moret-bailly00}. Then, $p^{-1}(\eta)=\varnothing$, which shows that there exists a   closed subscheme $S_1\not=S$ such that $p$ factors through the immersion  $S_1\to  S$. Hence, $p^{-1}(S^\circ)\to S^\circ$ factors through $S_1\cap S^\circ$ showing that $\Phi_E\cap S^\circ\subset S_1$. In conclusion, $\Phi_E\subset S_1\cup (S\setminus S^\circ)$. 
\end{proof}

\subsection{``Openness'' of the locus of local freeness: first case}
Let $R$ be a discrete valuation ring with uniformizer $t$ and field of fractions $K$. We write $R_n$ for the quotient $R/(t^{n+1})$ and,  given an $R$-scheme  $W$, we  let     $W_n$, respectively $W_K$,  denote the scheme $W\otimes R_n$, respectively   $W\otimes_RK$. 
We assume that   $S$ is $\Spec R$ so that we have a smooth morphism 
\[
f:X\longrightarrow \mathrm{Spec}\,R. 
\] 

\begin{thm}\label{07.11.2023--1}Let $F$ be a reflexive coherent $\co_X$-module. 
Let   \[Z=\{x\in X\,:\,\text{$F_x$ is not  free}\}.\] Assume that: 
\begin{enumerate}[(i)]\item  the $\co_{X_0}$-module
$\left({F}|_{X_0}\right)^{\vee\vee}$ is locally  free  and 
\item the following inequalities 
\[
\mathrm{codim}(Z\cap X_0,X_0)\ge3\quad \text{and}\quad\mathrm{codim}(Z\cap X_K,X_K)\ge3
\]  
hold true. 
\end{enumerate}
Then ${F}$ is locally free on an open neighbourhood of the special fibre $X_0$.
\end{thm}

\begin{proof}
There is no loss of generality in assuming that   $X=\mathrm{Spec}\,A$ with $A$ a domain. The proof is structured with a sequence of Lemmas. 
 
The following is a simple consequence of the codimension formula \cite[$\mathrm{IV}_2$, 5.1.3]{ega};   we omit the details.

\begin{lemma}$\mathrm{codim}(Z,X)\ge3$. \qed
\end{lemma}

Using \cite[Proposition 1.1]{hartshorne80}, we may assume that $F$ fits into an exact sequence $0\to F\to E\to Q\to0$, where $E$ is locally free and $Q$ is torsion-free. 
We abbreviate  ${F}|_{X_n}$ to  $F_n$ and   let $U$ be the the open subset $X\setminus Z$. One of the main actors in what follows is the module of ``formal functions'' $\varprojlim_nH^0(U_n,F_n)$ of $F$ on $U$.

\begin{lemma}The  projection map  
\[
\varprojlim_nH^0(U_n,F_n)\longrightarrow H^0(U_0,F_0)
\]
is surjective. 
\end{lemma}
\begin{proof}
We shall abuse notation and let $F_n$ stand for the sheaf of abelian groups on $X_0$ obtained from the fact that $X_n$ and $X_0$ underlie the same topological space. From the fact that $F$ is free of $t$-torsion, we have an exact sequence of sheaves 
\[
0\longrightarrow F_0\longrightarrow F_{n+1}\stackrel{\mathrm{pr}}\longrightarrow F_n\longrightarrow0.
\]
Hence, in order for $H^0(U_{n+1},F_{n+1})\to H^0(U_{n},F_{n})$ to be surjective it is sufficient  that $H^1(U_{0},F_{0})=0$. Now, since $F$ is locally free on $U$, we know that 
$H^1(U_{0},F_{0})\simeq H^1(U_{0}, F_0^{\vee\vee})$. (During this proof,  we shall write $F_0^{\vee\vee}$ instead of    $(F_0)^{\vee\vee}$.) From the long cohomology exact sequence  with supports of $F_0^{\vee\vee}$ and the fact that $X_0$ is affine, we obtain an isomorphism $ H^1(U_{0},F^{\vee\vee}_{0})\simeq H^2_{Z_0}(X_0,F^{\vee\vee}_0)$. We can therefore rely on the fact that $F_0^{\vee\vee}$ is locally free. Indeed,    for every $z\in Z_0$,    we have 
\[\begin{split}\mathrm{depth} \left(F^{\vee\vee}_{0}\right)_z&=\mathrm{depth}\,\mathcal O_{X_0,z}\\&=\dim\mathcal O_{X_0,z}\\&\ge3
\end{split}\]
and hence $H^2_{Z_0}(X_0,F_0^{\vee\vee})=0$,   by Lemma 3.1 and Proposition 3.3 of  \cite[Exp. III]{sga2}.
\end{proof}

\begin{lemma}\label{06.10.2023--1}
The restriction map \[\rho:H^0(X_0,F_0)\longrightarrow H^0(U_0,F_0)\] 
is surjective. 
\end{lemma}

\begin{proof}Let $A'$ stand for the $t$-adic completion of $A$ and let  $X'=\mathrm{Spec}\,A'$. Denoting  by $h:  X'\to X$ the canonical {\it flat} morphism, we introduce the following notations:  $Z'=h^{-1}(Z)$,  $U'= X'\setminus   Z'$ and $F'=h^*F$. Let us note that  $\mathrm{codim}(Z',X')\ge\mathrm{codim}(Z,X)\ge3$ \cite[$\mathrm{IV}_2$, 2.3.4-5]{ega}. 

{\it Step 1.} We wish to  show that the natural map 
\begin{equation}\label{06.10.2023--2}
H^0(U',F')\longrightarrow \varprojlim_nH^0(U_n,F_n) 
\end{equation}
is bijective. For that we will employ Proposition  1.4 of \cite[Exp. IX]{sga2}. 

Let $x'\in U'$ be such that $\mathrm{codim}(\overline{\{x'\}}\cap Z',\overline{\{x'\}})=1$. If $\mathrm{codim}(\overline{\{x'\}},X')\le1$, then $\mathrm{codim}(\overline{\{x'\}}\cap Z',X')\le2$, which is excluded. Hence, $\dim\co_{X',x'}=\mathrm{codim}(\overline{\{x'\}},X')\ge2$ for such $x'$. 
Since $F$ is reflexive so is $F'$ and hence  $\mathrm{depth} \, F'_{x'}\ge2$ \cite[Proposition 1.4.1]{bruns-herzog93}.  
In conclusion, by the aforementioned result of \cite{sga2}, we are right to affirm that the map  in \eqref{06.10.2023--2} is bijective. 

{\it Step 2.}
Since $\mathrm{codim}(Z',X')\ge\mathrm{codim}(Z,X)\ge3$ and $F'$ is reflexive, we know that   $\mathrm{depth}\,F'_{z'}\ge2$ for each $z'\in Z'$ \cite[Proposition 1.4.1]{bruns-herzog93}. Then,  the map   \[H^0(X',F')\longrightarrow   H^0(U',F')\] is an isomorphism because $\underline H_{Z'}^0(F')=\underline H_{Z'}^1(F')=0$ (cf. Proposition  3.3 in   \cite[Exp. III]{sga2} and Corollary 4 in  \cite[Exp. II]{sga2}). Using the commutative diagram 
\[
\xymatrix{
H^0(X',F')\ar[d]^\sim\ar[r]^-{\sim} & \varprojlim_n H^0(X_n,F_n)\ar@{->>}[rr] \ar[d]_\sim& &H^0(X_0,F_0)\ar[d]^{\textcircled{\tiny1}}
\\
H^0(U',F')\ar[r]_-{\sim}& \varprojlim_n H^0(U_n,F_n)\ar@{->>}[rr]&& 
H^0(U_0,F_0)}
\]
we assure that the map $\textcircled{\tiny1}$ is indeed surjective.
\end{proof}

\begin{lemma}\label{08.10.2023--1}The $\mathcal O_{X_0}$-module $F_0$ is   reflexive (and hence $F_0\stackrel\sim\to F_0^{\vee\vee}$). 
\end{lemma}
\begin{proof} We wish to apply \cite[Proposition 1.3]{hartshorne80} (or  \cite[Proposition 1.4.1]{bruns-herzog93}).   
Let $x\in X_0$ be such that $\dim \mathcal O_{X_0,x}\ge2$: it is to be shown that $\mathrm{depth}\,F_{0,x}\ge2$.  If $x\in U_0$, the we know that $F_{0,x}$ is free, so that $\mathrm{depth}\,F_{0,x}=\mathrm{depth}\,\mathcal O_{X_0,x}\ge2$. Hence, all difficulty lies in the case where $x\in Z_0$. We then rely on the cohomological characterization of depth. 

Because of Lemma \ref{06.10.2023--1} and the exact sequence of local cohomology 
\[
0\longrightarrow H^0(X_0,F_0)\longrightarrow    H^0(U_0,F_0)\longrightarrow H^1_{Z_0}(X_0,F_0)\longrightarrow\underbrace{H^1(X_0,F_0)}_0
\]
we conclude that $H^1_{Z_0}(X_0,F_0)=0$. Since $H^0_{Z_0}(X_0,F_0)=0$ because $F_0$ is torsion-free, we conclude from   Proposition 3.3 of \cite[Exp. III]{sga2} in conjunction with Corollary 4 of \cite[Exp. II]{sga2}  that $\mathrm{depth}\,F_{0,x}\ge2$ for all $x\in Z_0$.
\end{proof}

We then finish the proof of the theorem: by Lemma \ref{08.10.2023--1}, we know that   $F_0$ is a flat $\mathcal O_{X_0}$-module. Since $F$ is $R$-flat, the fibre-by-fibre flatness criterion  \cite[$\mathrm{IV}_3$,  11.3.10]{ega} assures that $F$ is $\mathcal O_X$-flat on each point of $X_0$ and hence on an open neighbourhood of $X_0$. Now, over a local Noetherian ring, each flat module of finite type is actually free \cite[Theorem 7.10, p.51]{matsumura}.
\end{proof}

We now state a variant of  Theorem \ref{07.11.2023--1} in terms more akin to the theory of  distributions. The definitions needed to grasp what follows are made in Appendix \ref{generalitiesbis}.

\begin{cor} \label{31.10.2023--3} 
Let $\cv$ be a distribution on $X/S$, where $S$ is the spectrum of the discrete valuation ring $R$. Assume that

\begin{enumerate}
\item[{\bf H1.}] The pull-back distribution $\cv_0= \cv|_{X_0}$ has a locally free tangent sheaf. 

\item[{\bf H2.}]
The inequality  
\[\mathrm{codim}(\mm{Sing}(\cv)\cap X_0,X_0)\ge2
\]
holds true.

\item[{\bf H3.}]
The inequalities  
\[
\mathrm{codim}(\mm{TSing}(\cv)\cap X_0,X_0)\ge3\quad \text{and}\quad\mathrm{codim}(\mm{TSing}(\cv)\cap X_K,X_K)\ge3
\]  
hold true.  
\end{enumerate}
Then the $\co_X$-module $T_\cv$ is locally free on an open neighbourhood of the special fibre $X_0$. Finally,   if $X$ is proper, then $T_\cv$ is locally free.  
\end{cor}

\begin{proof}It is enough  to prove that $\left(T_{\cv}|_{X_0}\right)^{\vee\vee}$ is locally free and apply Theorem \ref{07.11.2023--1}.
Indeed, $T_\cv$ is reflexive \cite[Proposition 1.1]{hartshorne80} and hypothesis (ii) in the statement of Theorem \ref{07.11.2023--1} is guaranteed by $\textbf{H3}$. 

 Consider the exact sequence of $\co_{X_0}$-modules 
 \[0\aro T_{\cv}|_{X_0}\aro T_{X_0}\aro Q_{\cv}|_{X_0}\aro 0.\]
 Letting $U=X\setminus\mm{Sing}(\cv)$, we know that  $\left(T_\cv|_{X_0}\right)|_{U_0}$ is saturated in $T_{X_0}|_{U_0}$. 
By   ${\bf H2}$, the open subset $U_0\subset X_0$ is big, and hence $\left(T_{\cv}|_{X_0}\right)^{\mm{sat}}\simeq (T_{\cv}|_{X_0})^{\vee\vee}$,  by Proposition \ref{16.11.2023--1}-(2).
Now, the definition of $\cv|_{X_0}$ (cf. Definition \ref{10.06.2024--4}) jointly with  $\textbf{H1}$ say  that   $(T_{\cv}|_{X_0})^{\mm{sat}}$ is locally free. Hence $\left(T_\cv|_{X_0}\right)^{\vee\vee}$ is locally free.

The final statement concerning the case where $X$ is proper is simple: if $Z=\{x\in X\,:\,\text{$T_{\cv,x}$ is not free}\}$, 
then $Z$ is a closed subset contained in $X_K$, which is only possible if $Z'=\varnothing$.  
\end{proof}

We now end with a simple example showing the necessity of {\bf H3} in Corollary \ref{31.10.2023--3}. 

\begin{ex}\label{counter_example} The notations in the beginning of this section are in force.  Let 
$A=R[x,y,z]$,  $X=\mathrm{Spec} \,A$. In what follows,   $\{\partial_x,\partial_y,\partial_z\}$ is the 
basis of $T_{X/R}$ dual to $\{\mathrm dx,\mathrm dy,\mathrm dz\}$. We shall use the theory of Fitting ideals \cite[20.2]{eisenbud95}. 

  Consider the closed 1-form \[\varphi=x \mathrm dy+y \mathrm dx+t \mathrm dz,\] and let $\mathcal{V}$ be the relative distribution on $X$ defined by   $T_{\mathcal V} =\mathrm{Ker}\, \varphi$. We have   
\[\begin{split}
T_{\mathcal V}&=\mm{Im}\, \begin{pmatrix}x&t&0\\-y&0&t\\0&-y&-x\end{pmatrix}
\\&= \co_{X}(x\partial_x-y\partial_y)+\co_{X}( t\partial_x-y\partial_z)+\co_{X}(t\partial_y-x\partial_z).\end{split}
\]
By construction, $T_{\mathcal V}$ is saturated in $T_{X/R}$ and  its   sequence of Fitting ideals is 
\[
0\subset \underbrace{\mm{Fitt}_{0}(T_{\mathcal V})}_{0}\subset \underbrace{\mm{Fitt}_1(T_{\mathcal V})}_{0}\subset \underbrace{\mm{Fitt}_2(T_{\mathcal V})}_{(x,y,t)}\subset \underbrace{\mm{Fitt}_3(T_{\mathcal V})}_{(1)}.
\]
Consequently  $T_{\mathcal V}$
is not $\mathcal O_X$-flat \cite[Proposition  20.8]{eisenbud95}   and $T_\cv|_{X_0}$ is not $\co_{X_0}$-flat either  \cite[Corollary 20.5, p.494]{eisenbud95}. 

Let us   analyse the singularities of $\mathcal V$,  which amounts to studying singularities of  $Q_{\mathcal V}$. By definition we have the following exact sequence: 
\[\tag{$\dagger$}
\xymatrix{
\mathcal O_X^3\ar[rrr]^-{\begin{pmatrix}x&t&0\\-y&0&t\\0&-y&-x\end{pmatrix}}&&&
\mathcal O_X^3\ar[r]& Q_{\mathcal V} \ar[r]&0.}
\]
The sequence of Fitting ideals of $Q_\cv$ is then   
\[ \tag{$\ddagger$}
0\subset \underbrace{\mm{Fitt}_{0}(Q_{\mathcal V})}_{0}\subset \underbrace{\mm{Fitt}_1(Q_{\mathcal V})}_{(y^{2},xy,ty,x^{2},tx,t^{2})}
\subset 
\underbrace{\mm{Fitt}_2(Q_{\mathcal V})}_{(x,y,t)}
\subset \underbrace{\mm{Fitt}_3(Q_{\mathcal V})}_{(1)} 
\]
so that the closed {\it subset} $\sing(\cv)$ is $\{x=y=t=0\}$; it
is contained   in $X_0$. Obviously,    $\mathrm{codim}(\mathrm{Sing}(\mathcal V)\cap X_0,X_0)=2$.

Let $\cv_0$ be the pull-back of $\cv$ to the special fibre $X_0$.  
By definition, $T_{\cv_0}$ sits in   an exact sequence  
\[
0\longrightarrow T_{\mathcal{V}_0}\longrightarrow T_{X_0}\longrightarrow \left( Q_{\cv}|_{X_0}\right)_\mm{tof}\longrightarrow0.\]
(For the notation ``$\mathrm{tof}$'', see Definition \ref{15.11.2023--1}.  Note that  there is no distinction here between torsion-free and torsion-less modules.) 
Using the presentations in   ($\dagger$),   we have 
\[
\xymatrix{
\co_{X_0}^3\ar[rrr]^-{\begin{pmatrix}x&0&0\\-y&0&0\\0&-y&-x\end{pmatrix}}&&&
\co_{X_0}^3\ar[r]^-{\overline\pi}& \restr{Q_{\mathcal V}}{X_0} \ar[r]&0}
\]
so that $\overline\partial_y,\overline\partial_z\in  Q_{\mathcal V}|_{X_0}$ are annihilated by non-zero divisors of $A_0[x,y,z]$.   Hence, the quotient sheaf $Q_{\cv_0}$ is described by the exact sequence
\[
\xymatrix{
\co_{X_0}^3\ar[rrr]^-{\begin{pmatrix}x&0&0\\-y&0&0\\0&1&1\end{pmatrix}}&&&
\co_{X_0}^3\ar[r]& Q_{\mathcal{V}_{0}}\ar[r]&0.}
\]
This implies  that $T_{\cv_0}=\co_{X_0}(x\partial_x-y\partial_y)+\co_{X_0}\partial_z \simeq\co_{X_0}^2$ is locally free,  while $T_\cv|_{X_0}$ is not. Finally, we have sequence of Fitting ideals: 
  \[\underbrace{\mathrm{Fitt}_0(Q_{\cv_0})}_{0}\subset\underbrace{\mathrm{Fitt}_1(Q_{\cv_0})}_{(x,y)}\subset\underbrace{\mathrm{Fitt}_2(Q_{\cv_0})}_{(1)},\]
  which proves that $\mathrm{Sing}(\cv_0)$ is $\{x=y=0\}\subset X_0$. 
Also, we have from ($\ddagger$)
\[
\underbrace{\mathrm{Fitt}_0\left(\restr{Q_{\mathcal V}}{X_0}\right)}_{0} \subset \underbrace{\mathrm{Fitt}_1\left(\restr{Q_{\mathcal V}}{X_0}\right)}_{(x^2,xy,y^2)} 
\subset \underbrace{\mathrm{Fitt}_2\left(\restr{Q_{\mathcal V}}{X_0}\right) }_{(x,y)}\subset 
\underbrace{\mathrm{Fitt}_3\left(\restr{Q_{\mathcal V}}{X_0}\right) }_{(1)}.
\]
This has the interesting consequence that ``the'' Fitting ideal (the first non-zero) of  $Q_\cv|_{X_0}$ is the second, while ``the'' Fitting ideal of $Q_{\cv_0}$ is   the first. 
\end{ex}

\subsection{``Openness results'': general case}
We now abandon the assumption that the base scheme $S$ is the spectrum of a DVR and go back to the general setting of the start: $f:X\to S$ is a smooth morphism of Noetherian schemes. 
Joining Corollary \ref{31.10.2023--3} and Proposition \ref{31.10.2023--2}, we arrive at:

\begin{cor}[Theorem \ref{T:locally free}]\label{21.03.2024--1} Let $f:X\to S$ be a proper and  smooth morphism of Noetherian schemes. For each $s\in S$, let $\cv_s$ stand for the pull-back of $\cv$ to the fibre $X_s$ above $s$.  Suppose that for each $s\in S$, we have
\[
\codim(\sing(\cv)\cap X_s,X_s)\ge3.
\] 
Then the set 
\[\Phi=\{s\in S\,:\,\text{$T_{\mathcal{V}_s}$ is locally free}\}
\]
is open in $S$. 
\end{cor}

\begin{proof}
Let $\sigma\rightsquigarrow s$ be a specialization in $S$ with $s\in \Phi$. 
According to \cite[$\mathrm{II}$, 7.1.9]{ega}, there exists a scheme $D$ which is  the spectrum of a discrete valuation, has  generic point $\eta$  and closed point $o$, and a map   $\varphi:D\to S$ sending $\eta$ to $\sigma$ and $o$ to $s$, and, in addition, is such that   $\boldsymbol k(\sigma)\stackrel\sim\to\boldsymbol k(\eta)$.

Let now $Y:=X\times_SD$ and  define     $\mathcal W$ to be the  the pull-back of $\cv$ to $Y/D$. 
Because $T_{\mathcal V_s}$ is locally-free, so is   $T_{\mathcal W_o}$  (cf. the transitivity of the pull-back, Corollary \ref{transitivity_fibres}). Since $\mathrm{Sing}(\cw)$ is contained in the inverse image of $\mathrm{Sing}(\cv)$   in $Y$ (see Proposition \ref{24.03.2024--2}), we have 
\[
\codim(\sing(\cw)\cap Y_o,Y_o)\ge3\quad\text{and}\quad\codim(\sing(\cw)\cap Y_\eta,Y_\eta)\ge3.
\] (The zealous reader here may wish to note that the inverse image of $\sing(\cv)$ in $X_o=X_s\otimes_{\boldsymbol k(s)}\boldsymbol k(o)$ also has codimension at least 3 \cite[$\mathrm{IV}_2$, 6.1.4]{ega}.)
By  Corollary \ref{31.10.2023--3},   $T_{\mathcal W_\eta}$ has a locally free tangent sheaf. But $\mathcal W_\eta=\cv_\sigma$ and  hence    $\sigma\in \Phi$. 
We then conclude that $\Phi$ is stable under generalisations and, being constructible (Proposition \ref{31.10.2023--2}), must be open. 
\end{proof}

For the next result, we shall require some basic facts concerning Kupka singularities. These are worked out in Section \ref{kupka_singularities}. We also employ the notations of the mentioned section in the following result. 

\begin{cor}[Theorem \ref{THM:Kupka}]\label{21.03.2024--1} Let $f:X\to S$ be a proper and  smooth morphism of Noetherian schemes. Assume that for   certain $q\ge1$ and line bundle $L$, 
the distribution $\cv$ is 
\begin{itemize}  
\item  defined by an LDS twisted $q$-form with values on $L$ and 
\item involutive.  \end{itemize}

For each $s\in S$, let $\cv_s$ stand for the pull-back of $\cv$ to the fibre $X_s$ above $s$.  Suppose that for each $s\in S$, we have
\[
\codim(\sing(\cv)\cap X_s,X_s)\ge2.
\] 
and 
\[\codim\left( {\rm  NKup}(\cv)  \cap X_s,X_s\right)\ge3.
\]
Then the set 
\[\Phi=\{s\in S\,:\,\text{$T_{\mathcal{V}_s}$ is locally free}\}
\]
is open in $S$. 
\end{cor}

\begin{proof}
Let $\sigma\rightsquigarrow s$ be a specialization in $S$ with $s\in \Phi$. 
According to \cite[$\mathrm{II}$, 7.1.9]{ega}, there exists a scheme $D$ which is  the spectrum of a discrete valuation, has  generic point $\eta$  and closed point $o$, and a map   $\varphi:D\to S$ sending $\eta$ to $\sigma$ and $o$ to $s$, and, in addition, is such that   $\boldsymbol k(\sigma)\stackrel\sim\to\boldsymbol k(\eta)$.

Let now $Y:=X\times_SD$ and  define     $\mathcal W$ to be the  the pull-back of $\cv$ to $Y/D$. 
Because $T_{\mathcal V_s}$ is locally-free, so is   $T_{\mathcal W_o}$  (cf. the transitivity of the pull-back, Corollary \ref{transitivity_fibres}). 
Since $\mathrm{Sing}(\cw)$ is  the inverse image of $\mathrm{Sing}(\cv)$   in $Y$, we have 
\[
\codim(\sing(\cw)\cap Y_o,Y_o)\ge2\quad\text{and}\quad\codim(\sing(\cw)\cap Y_\eta,Y_\eta)\ge2.
\] (The zealous reader here may wish to note that the inverse image of $\sing(\cv)$ in $X_o=X_s\otimes_{\boldsymbol k(s)}\boldsymbol k(o)$ also has codimension at least 3 \cite[$\mathrm{IV}_2$, 6.1.4]{ega}.)
Since $\mm{NKup}(\cw)\supset\mm{TSing}(\cw)$, see Proposition \ref{23.05.2024--1}, we have 
\[
\codim(\mm{TSing}(\cw)\cap Y_o,Y_o)\ge3\quad\text{and}\quad\codim(\mm{TSing}(\cw)\cap Y_\eta,Y_\eta)\ge3.
\]

By  Corollary \ref{31.10.2023--3},   $T_{\mathcal W_\eta}$ has a locally free tangent sheaf. But $\mathcal W_\eta=\cv_\sigma$ and  hence    $\sigma\in \Phi$. 
We then conclude that $\Phi$ is stable under generalisations and, being constructible (Proposition \ref{31.10.2023--2}), must be open. 
\end{proof}

It is sometimes useful to formulate Corollary \ref{21.03.2024--1} in a slightly different manner. 

\begin{cor}
Let $f:X\to S$ be a proper and  smooth morphism of Noetherian schemes  having irreducible fibres. Adopt the same notations regarding  pull-backs as in Corollary \ref{21.03.2024--1}.  
Assume that for $o\in S$, the $\co_{X_o}$-module    $T_{\cv_o}$ is locally free and $\codim(\mathrm{Sing}(\cv)\cap X_o,X_o)\ge3$. Then there exists an open neighbourhood  $U$ of $o$ in $S$ such that $T_{\cv_s}$ is locally free for all $s\in U$. 
\end{cor}

\begin{proof}We only need to remark that 
\[S'=
\{s\in S\,:\,\codim(\mathrm{Sing}(\cv)\cap X_s,X_s)\ge3\}
\]
is an open subset of $S$ by Chevalley's semi-continuity theorem \cite[$\mathrm{IV}_3$, 13.1.5]{ega} and properness of $\mathrm{Sing}(\cv)$,  and then apply Corollary \ref{21.03.2024--1}. 
\end{proof}

We now wish to join Corollary \ref{31.10.2023--3} and Proposition \ref{28.11.2023--1} in order to show that, in certain cases, {\it the isomorphism class of the tangent sheaf of a distribution   remains constant} on an open subset of the base.
For that, we shall require some basic results of the theory of deformations.

Let us   place ourselves in the {\it setting of Proposition  \ref{28.11.2023--1}}: $S$ is a Noetherian scheme over  $\Bbbk$ (we place no hypothesis on the field $\Bbbk$), $M$ is a smooth and connected $\Bbbk$-scheme, and $X$ is $M\times_\Bbbk S$, while $f:X\to S$ is just the projection. 
Let   $E$ be a vector bundle  over $M$. We {\it assume  $E$ to be  rigid}, i.e.   $\mathrm{Ext}_{\co_M}^1(E,E)=0$. The reader having familiarity with deformation theory \cite{schlessinger68} will probably intuit that the terminology is given in order to express the fact that $E$ ``will not deform''. Unfortunately, we were unable to find a suitable reference for the technique behind this intuition, so that we   work out the details in a proposition. 

\begin{prop}Let us adopt the above conventions and notations. Fix $s\in S$ and let $\ce$ be a vector bundle over $X\ti_\Bbbk S$ such that $E\ot_\Bbbk\boldsymbol k(s)=\ce|_{M\ot_\Bbbk\boldsymbol k(s)}$. Then $\ce\ot_{\co_S}\co_{S,s}\simeq E\ot_\Bbbk\co_{S,s}$.
\end{prop}
\begin{proof}
 Let   $s\in S$ be a point and   then introduce the category $\mathcal C$ of local Artin $\wh \co_{S,s}$-algebras with residue field $\boldsymbol k(s)$ \cite[Section 1]{schlessinger68}. In this setup, we have    the deformation functor 
 \[
 \mathrm{Def}_E :\mathcal C \aro \mathrm{Set}
 \] 
associated to the vector bundle $E\otimes_\Bbbk \boldsymbol k(s)$ on the $\boldsymbol k(s)$-scheme $M\otimes_\Bbbk \boldsymbol k(s)$.   We then know that the tangent space to  $\mathrm{Def}_E$ is isomorphic to 
\[\begin{split}
\mathrm{Ext}^1(E\ot_\Bbbk \boldsymbol k(s),E\ot_\Bbbk \boldsymbol k(s))&\simeq \mathrm{Ext}^1(E,E)\ot_\Bbbk \boldsymbol k(s)\\&=0. 
\end{split}\]

Let $R$ be a hull for $\mathrm{Def}_E$ \cite[Theorem 2.11]{schlessinger68}. (That $\mathrm{Def}_E$ satisfies Schlessinger's conditions can be verified as in Section 3.1 of \cite{schlessinger68}.) This is a complete   local $\wh\co_{S,s}$-algebra with residue field $\boldsymbol k(s)$. 
By definition of $R$,  its  relative Zariski tangent space   over $\wh\co_{S,s}$ is  $\mathrm{Ext}^1 (E\ot_\Bbbk \boldsymbol k(s),E\ot_\Bbbk \boldsymbol k(s))=0$. Consequently,    the structural map $\wh\co_{S,s}\to R$ is {\it surjective} \cite[Lemma 1.1]{schlessinger68}. We conclude that   $\mathrm{Def}_E(A)$ has {\it at most one element for each $A\in\mathcal C$}. Now, it is clear that 
for every  such $A$, the class of   $E\ot_\Bbbk A$ is an element in $\mathrm{Def}_E(A)$, so that $\mathrm{Def}_E(A)$ is simply this class. 

Let then  $\ce$ be a vector bundle as in the statement. For each $\co_S$-algebra $A$, we put $\ce_A:=\ce\ot_{\co_S}A$ and   $E_A:=E\ot_\Bbbk A$; these are vector bundles on the $A$-scheme $M\ot_\Bbbk A$.  From the argument above,   there exists, for each $A\in\mathcal C$, an isomorphism   $u_A: \ce_A \stackrel\sim\to E_A$.

 Let us agree to denote by $\co_{s,n}$ the quotient of $\co_{S,s}$ by the $(n+1)$st power of its maximal ideal,   by $I_n$, respectively  $P_n$,  the group 
of {\it automorphisms}   of  $E_{\co_{s,n}}$, respectively set of {\it isomorphisms} 
  $\ce_{\co_{s,n}}\stackrel\sim\to E_{\co_{s,n}}$.  Then, $P_n \not=\varnothing$ and 
  $I_n$ acts on its  left {\it freely and transitively}. Since the restriction map $I_{n+1}\to I_n$ is surjective, it is possible to find a compatible family of isomorphisms $u_n:\ce_{\co_{s,n}}\stackrel\sim\to E_{\co_{s,n}}$. (For the careful reader: surjectivity of $I_{n+1}\to I_n$ can be checked using the fact that for any given local Noetherian ring $A$ with residue field $\boldsymbol k(s)$, a morphism  $u:E_A\to E_A$ is an isomorphism if and only if its restriction $u_0:E\ot_\Bbbk\boldsymbol k(s)\to E\ot_\Bbbk\boldsymbol k(s)$ is an isomorphism.)

From Grothendieck's Existence Theorem   \cite[$\mathrm{III}_1$, 5.1.4]{ega}, we obtain an isomorphism of vector bundles  $u:\ce_{\widehat \co_{s}} \stackrel\sim\to  E_{\widehat \co_{s}}$. Then, because 
\[
\mathrm{Hom}  (E_{\widehat \co_{s}},\ce_{\widehat \co_{s}} )\simeq \mathrm{Hom}  (E_{\co_s} ,\ce_{\co_s})\ot_{\co_s}{\widehat \co_s}\]
there exists $u^\circ:E_{\co_s}  \to\ce_{\co_s}$ such that $u^\circ\ot_{\co_s}{\rm id}\equiv u$ modulo the maximal ideal of $\wh \co_s$.
It follows that $u^\circ$ is an isomorphism. Hence, $E_{\co_s}\simeq \ce_{\co_s}$.  
\end{proof}

\begin{cor}\label{21.03.2024--2}
Let us adopt the above setting: $M$ is a smooth and proper $\Bbbk$-scheme, $S$ is a Noetherian $\Bbbk$-scheme,   $X=M\times_\Bbbk S$, $f:X\to S$ is the projection and $E$ is a rigid vector bundle on $M$.   
For each $s\in S$, let $\cv_s$ stand for the pull-back of $\cv$ to the fibre $X_s =M\ot_\Bbbk \boldsymbol k(s)$ and suppose that 
\[\codim(\sing(\cv)\cap X_s,X_s)\ge3\]
for each $s\in S$. 
Then the set 
\[\Phi_E=\{s\in S\,:\,\text{$T_{\mathcal{V}_s}$ is isomorphic to $E_s$}\}
\]
is open in $S$. 
\end{cor}

\begin{proof}
Let $\sigma\rightsquigarrow s$ be a specialization in $S$ with $s\in \Phi_E$. 
According to \cite[$\mathrm{II}$, 7.1.9]{ega}, there exists the spectrum of a discrete valuation ring   $D$    with generic point $\eta$ and closed point $0$, and a map   $\varphi:D\to S$  such that   $\varphi(\eta)=\sigma$,  $\varphi(0)=s$  and     $\boldsymbol k(\sigma)\stackrel\sim\to\boldsymbol k(\eta)$.

Let now $Y:=X\times_SD$ and  define     $\mathcal W:=\mathcal V|_{Y}$, the pull-back of $\cv$ to $Y/D$.
Because $T_{\mathcal V_s}\simeq E\ot_k\boldsymbol k(s)$, 
we have $T_{\mathcal W_0}\simeq E\otimes_k\boldsymbol k(0)$. Note, in particular, that this $\co_{Y_0}$-module is locally free.  
 By Corollary \ref{31.10.2023--3},  
$T_{\mathcal W}$ is locally free over $\co_Y$. By the above discussion, we conclude that  $T_\cw\simeq E\ot_k\co_D$. Hence, $T_{\cw_\eta}\simeq (T_\cw)_\eta\simeq E_\eta$.  
We then conclude that $\Phi_E$ is stable under generalisations and, being constructible (Proposition \ref{28.11.2023--1}), must be open. 
\end{proof}

\section{Vector fields on the Borel variety   of a simple algebraic group}\label{22.02.2024--3}

In this section we set out to construct a class of examples of distributions, in fact foliations, on Borel varieties having a singular set of  codimension at least three.  This will enable us  to apply Theorem \ref{T:locally free}. 

\subsection{Notations and fundamental results}\label{22.02.2024--2}

We  set up some notations and terminology which are {\it in force in   Section \ref{22.02.2024--3}}. The field  $\Bbbk$ is to be  algebraically closed and of characteristic zero. We shall work only with reduced algebraic $\Bbbk$-schemes; points are always assumed to be closed.

Let $G/\Bbbk$ be a {\it semi-simple and    adjoint} linear algebraic group with Lie algebra $\g g$. (So $\g g$ is semi-simple \cite[13.5]{humphreys-ag}.)
The subset of nilpotent, respectively semi-simple, elements in $\g g$ shall be denoted by $\g g_{\rm nil}$, respectively $\g g_{\rm ss}$. 
 
Given $x\in \g g$, we let $Z_G(x)$ stand for the centralizer of $x$ in $G$, i.e. the stabilizer of $x$ for the adjoint action. 
We let also $Z_{\g g}(x)$ stand for the centralizer of $x$ in $\g g$, which is $\{y\in\g g\,:\,[x,y]=0\}$. 
We fix a Borel subgroup $A$ with maximal torus $T$ and unipotent part $U$. The Lie algebra of $A$ is denoted by $\g a$ and that of $T$ by $\g t$. 

Let $\mathcal{B}$ be the {\it Borel variety} of $G$: its points correspond simply to the Borel subgroups of $G$ and, as these are all conjugated to $A$, we have  $\mathcal{B}=G/A$ \cite[\S11]{borel}, \cite[23.3]{humphreys-ag}. In particular, $\mathcal B$ is a smooth and connected rational variety of dimension $\dim U$. 
 
Among subvarieties of $\mathcal B$, the {\it Springer fibres} have been the centre of much attention.  These are defined, once chosen $x\in \g g$, as
\[
\cb_x=\{B\in\mathcal{B}\,:\,x\in\mathrm{Lie}\, B\}.
\]
(The name is derived from the fact that for   $x\in\g g_{\rm nil}$, each $\mathcal B_x$ appear as the fibre in the Springer resolution of the singularities of $\g g_{\rm nil}$ \cite{springer69}.)
Let us now state a deep result in the theory   on which much of our arguments hinge: Steinberg's ``dimension formula.'' 
 
\begin{thm}[{\cite{steinberg76}, \cite[\S7.4]{mcgovern02}}] Assume that $G$ is semi-simple and let $r=\dim T$ be its rank. Then, for any $x\in\g g$,   the following formula holds:
\[
\dim Z_G(x)=r+2\dim\mathcal{B}_x.
\] \qed
\end{thm}

As a consequence, we have 
\[\begin{split}
2\,\codim(\mathcal{B}_x,\mathcal B)&=2\dim U-2\dim\mathcal{B}_x\\&=2\dim U+r-\dim Z_G(x)\\&=\mm{codim}\,Z_G(x).
\end{split}.  
\]
(Note in particular that $\codim Z_G(x)$ is always even.) Since we are in characteristic zero, we know that $\mathrm{Lie}\, Z_G(x)=Z_{\g g}(x)$ \cite[Lemma 7.4]{borel} and we shall work with $Z_{\g g}(x)$ in place of $Z_G(x)$. In a nutshell, the convenient formula hods:
\begin{equation}\label{03.04.2024--3}
\codim(\cb_x,\cb)=\frac12\codim Z_{\g g}(x). 
\end{equation}

The following shall be useful to us: its simple proof is omitted.  
\begin{lemma}\label{03.04.2024--4}Let  $x\in \g g$ have  Jordan-Chevalley decomposition $x=x_s+x_n$. Then   $Z_{\g g}(x)=Z_{\g g}(x_s)\cap Z_{\g g}(x_n)$. In particular
\[\mm{codim}\,Z_{\g g}(x)\ge\max(\mm{codim}\,Z_{\g g}(x_s),\mm{codim}\,Z_{\g g}(x_n)).\] \qed
\end{lemma}

\subsection{The variety of Borel subgroups fixed by a   subalgebra}\label{24.02.2024--4}
Let now $\g h\subset\g g$ be a non-zero subalgebra. 
We  now define the object to be studied in this section: it is the set (soon enough we shall endow it with more structure)
\[
\mathcal B_{\g h} =  \bigcup_{x\in\g h\setminus\{0\}}\mathcal B_{x}. 
\] 
Clearly $\mathcal B_x=\mathcal B_y$ if $x,y\in\g g\setminus\{0\}$ are proportional, so, for each $\ell\in \PP(\g h)$,  the set $\mathcal B_\ell$ is well defined. 
Following Springer \cite{springer69}, we shall analyse $\mathcal B_{\g h}$ through the lenses of the incidence set, which is: 
\[
V_{\g h} =\{(\ell,B)\in\mathbb P(\g h)\times\mathcal{B}\,:\,\ell\subset \mm{Lie}\,B\}.
\]


\begin{prop} The set $V_{\g h}$ is  closed in $\mathbb P(\g h)\ti\cb$.\end{prop}
\begin{proof} We Recall that $A$ stands for a Borel subgroup of $G$ with Lie algebra     $\g a$. 
Consider the action  morphism $\alpha:  \mathbb P(\g h)\times G \to \mathbb P(\g g)$ defined by $(\ell,g)\mapsto g^{-1}(\ell)$. Then $\alpha^{-1}(\mathbb P(\g a))=\{(\ell,g)\in\PP(\g h)\times G\,:\,\ell\subset  g(\g a)\}$, which is $\{(\ell,g) \,:\,\ell\subset   \mathrm{Lie}\,gAg^{-1}\}$.  
Let  $\pi:\mathbb P(\g h) \times G\to \mathbb P(\g h)\times \mathcal B$ be the morphism derived from the projection $G\to \mathcal B$. 
Clearly $\pi(\alpha^{-1}(\PP(\g a)))=V_{\g h}$. Now, although $\pi$ is not a proper morphism, it is the quotient morphism for the action of $A$ on the right of  $\mathbb P(\g h)\times G$ given by $(\ell,g)*a=(\ell,ga)$  and hence the closed and $A$-invariant subset $ \alpha^{-1}(\PP(\g a))$ is taken to a closed subset. 
\end{proof} 

Let $\mathrm{pr}_1:\mathbb{P}(\g h)\ti\cb\to\mathbb{P}(\g h)$ and $\mathrm{pr}_2:\mathbb{P}(\g h)\ti\cb\to\mathcal{B}$ be the  projections. Then  
\[
\mathcal B_{\g h}:=\mathrm{pr}_2(V_{\g h}),
\]
so that $\mathcal B_{\g h}$ now has the structure of a closed subset of $\mathcal B$ and moreover $\dim \mathcal B_{\g h}\le\dim V_{\g h}$  \cite[$\mathrm{IV}_2$, 4.1.2(i)]{ega}.  Note that  the fibre of  $(\mathrm{pr}_2)|_{V_{\g h}}$ above     $B\in \cb_{\g h}$ is $\mathbb P(\g h\cap\mathrm{Lie}\,B)$. 
Our strategy from this point on is to bound $\codim\mathcal B_\ell$  from below, that is, 
 find a certain   $c\in\mathbb N$ satisfying  
\[
\codim (\mathcal B_\ell,\mathcal B)\ge c,\quad\forall\ell\in\mathbb P\g (\g h).
\]  
This will then allow us to conclude  that 
\begin{equation}\label{20.02.2024--1}\codim( \mathcal B_{\g h},\mathcal B)\ge c-\dim\mathbb P(\g h).
\end{equation}
(We apply \cite[$\mathrm{IV}_2$, 5.6.7]{ega}  to $\mathrm{pr}_1$.)
To arrive at this goal, equation \eqref{03.04.2024--3} and  Lemma \ref{03.04.2024--4} allow us to consider separately  centralisers of semi-simple and nilpotent elements, which is done in 
Sections \ref{centralisers_ss} and \ref{centralisers_nil}

\subsection{Centralisers of semi-simple elements}\label{centralisers_ss}
Recall that $T$ stands for a maximal torus of $G$;  denote by $\g t$ its  Lie algebra and   let  $\Phi\subset \g t^\vee$ be the associated root system \cite[Section  8]{hum}. As we have chosen a Borel subgroup $A$,  we have in our hands a base  $\Delta$ and a set of positive roots $\Phi^+$.
As customary, given $\alpha \in\Phi$, we let $\g g_\alpha=\{x\in \g g\,:\,\text{$[sx]=\alpha(s)x$ for each $s\in \g t$} \}$.   
 These conventions are now employed in the remainder of Section \ref{22.02.2024--3}. 

For any given $s\in \g t$, let  
\[
s^\perp=\{\lambda\in\g t^\vee\,:\,\lambda(s)=0\}.
\]

\begin{lemma}The following are true
\begin{enumerate}[(1)]\item Let $s\in \g t\setminus\{0\}$ be given. 
Then 
$Z_{\g g}(s)=\g t\oplus \bigoplus_{\alpha\in \Phi \cap s^\perp  }\g g_\alpha$ 
 and $\Phi \cap s^\perp$
is its root system (eventually reductible). In particular  
\[
\codim Z_{\g g}(s)=\#(\Phi\setminus s^\perp  ). 
\]
\item Let $s'\in \g g_{\rm ss}\setminus\{0\}$. Then there exists an element $\sigma\in \g t$ such that $\codim Z_{\g g}(s')=\#(\Phi\setminus\sigma^\perp)$. 
\end{enumerate}
\end{lemma}

\begin{proof}(1) See \cite[Lemma 2.1.2,p.20]{collingwood-mcgovern93}.

(2) This follows from the fact that any two maximal tori are conjugate.  \end{proof}

The relation between the root system $\Phi$ and the root system $\Phi\cap s^\perp$ is explained by:
  
\begin{prop}\label{25.06.2024--1}Each base of $\Phi\cap s^\perp$ can be extended to a base of $\Phi$. In particular, the Dynkin diagram of $\Phi\cap s^\perp$ is obtained from the Dynkin diagram of $\Phi$ by removal of a number of vertices and the edges touching them. 
\end{prop}
\begin{proof}\cite[VI.1.7, Proposition 24]{bourbaki_lie}. 
\end{proof}

From a direct analysis of the Dynkin diagrams and the tables of \cite{bourbaki_lie}, we derive. 

\begin{cor}\label{23.02.2024--5}We suppose that $\g g$ is simple. Let $s\in \g g_{\mathrm{ss}}\setminus\{0\}$.
\begin{enumerate}[(1)]
\item Assume that $\Phi$ is of type $\mathrm A_r$ for some $r\ge1$. Then $\Phi\cap s^\perp$ a direct sum of systems of type $\mathrm A$. In particular, $\#(\Phi\setminus s^\perp)\ge2r$. 

\item Assume that $\Phi$ is of type $\mathrm B_r$ for some $r\ge2$, respectively $\mathrm C_r$ for some $r\ge3$. Then $\Phi\cap s^\perp$ is a direct sum of systems of type $\mathrm A$ or $\mathrm B$, respectively     $\mathrm A$ or $\mathrm C$.   
In particular, $\#(\Phi\setminus s^\perp)\ge 4r-2$.

\item Assume that $\Phi$ is of type $\mathrm D_r$ for some $r\ge4$. Then $\Phi\cap s^\perp$ is a direct sum of systems of type $\mathrm A$ or $\mathrm D$.    
In particular, $\#(\Phi\setminus s^\perp)\ge 4r-4$.

\item Assume that $\Phi$ is of type $\mathrm E_6$. Then $\Phi\cap s^\perp$ is either $\mathrm E_6$ or a direct sum of systems of type $\mathrm A$,    $\mathrm D$. In particular, $\#(\Phi\setminus s^\perp)\ge36$. 

\item Assume that $\Phi$ is of type $\mathrm E_7$. Then $\Phi\cap s^\perp$ is   a direct sum of systems of type $\mathrm A$,    $\mathrm D$ and   $\mathrm E$. In particular, $\#(\Phi\setminus s^\perp)\ge54$. 

\item Assume that $\Phi$ is of type $\mathrm E_8$. Then $\Phi\cap s^\perp$ is a direct sum of systems of type $\mathrm A$,    $\mathrm D$ and   $\mathrm E$. In particular, $\#(\Phi\setminus s^\perp)\ge114$.

\item Assume that $\Phi$ is of type $\mathrm F_4$. Then $\Phi\cap s^\perp$ is $\mathrm F_4$ or  a direct sum of systems of type $\mathrm A$,   $\mathrm B$ or  $\mathrm C$.    
In particular, $\#(\Phi\setminus s^\perp)\ge 30$.

\item Assume that $\Phi$ is of type $\mathrm G_2$. Then $\#(\Phi\setminus s^\perp)\ge16$.
\end{enumerate}
\end{cor}
\begin{proof}
The idea of proof is   simple and requires constant observation of the tables in \cite{bourbaki_lie}. We shall leave to the reader the verification, with the help of Proposition \ref{25.06.2024--1}, that the only possible types for the system $\Phi\cap s^\perp$ are the ones stated. Also, we note that $\#\Phi\cap s^\perp\not=\Phi$, since $\Phi$ spans $\g t^\vee$. 

Let us adopt the following convention. For a root system of type $\mathrm A_r$, $\mathrm B_r$,\ldots, we shall let $a_r$, $b_r$,\ldots stand for the number of roots in it. Hence $a_r=r(r+1)$, $b_r=c_r=2r^2$, $d_r=2r(r-1)$, $e_6=72$, $e_7=126$, $e_8=240$, $f_4=48$ and $g_2=12$. At this point, we note 
that if $\tau\in\{a,b,c,d\}$, then  $\tau_{i+j}\ge\tau_{i}+\tau_j$. 

Now  we move on to verify assertion (3) hoping that the reader will be able to verify the remaining by following the presented method. 
Let then  $\Phi$ be of type  $\mathrm D_r$.  It follows that $\Phi\cap s^\perp$ is of type $A_{r_1}+ \cdots +  A_{r_m}+ D_{r_{m+1}}+\cdots  +   D_{r_n}$, where $\sum r_i\le r-1$. Now, $a_i\le d_i$, so $\#\Phi\cap s^\perp\le d_{i}+d_j$ with $i+j=r-1$. Then $\#\Phi\setminus s^\perp\ge d_{r}-d_{i}-d_j$, and $d_{r}-d_{i}-d_j\ge4(r-1)$ as a simple verification shows. 

\end{proof}

\begin{cor}\label{03.04.2024--2}Suppose  that $\g g$ is simple and that $\Phi$ is not of type ${\rm A}_1,{\rm A}_2,{\rm A}_3$ or ${\rm B}_2$. Then  
\[
\frac12\codim Z_{\g g}(s) \ge4.  \qed
\]
\end{cor}

\begin{rmk}The expression of $\#(\Phi\setminus s^\perp )$ in terms of tabulated data is known in the case where  $s$ is a {\it co}-root and $\Phi$ is ADE: According to \cite[VI.1.11, Proposition 32]{bourbaki_lie} we have  $\#(\Phi\setminus s^\perp)=4\operatorname{Cox}(\Phi)-6$, where $\operatorname{Cox}(\Phi)$ is the Coxeter number of $\Phi$. 
\end{rmk}






\subsection{Centralisers of nilpotent elements}\label{centralisers_nil}
We      assume here that   $\g g$ is {\it simple}, so that $\Phi$ is irreducible. Let  $\widetilde \alpha$ be the longest root in $\Phi$. (For the existence of $\widetilde\alpha$, see \cite[VI.1.8, Proposition 25]{hum}.)  Then    

\begin{thm}For each $n\in\g g_{\rm nil}\setminus\{0\}$ we have 
\[
\mm{codim}\,Z_{\g g}(n) \ge\#(\Phi^+\setminus \widetilde{\alpha}^\perp )+1.
\]
In addition, the lower bound is attained. 
\end{thm}
\begin{proof}[Proof] See \cite[Theorem 4.3.3]{collingwood-mcgovern93}.
\end{proof}

In order to interpret $\Phi\setminus \widetilde{\alpha}^\perp$, we need the ``dual Coxeter''  number $\operatorname{Cox}^*(\Phi)$ \cite[7.6, p.119]{fuchs-schweigert} and a result   proved by R. Suter and W. Wang independently.

\begin{thm}[{\cite[Proposition 1]{suter98},\cite[Theorem 1]{wang99}}]We have 
\[\#(\Phi\setminus\widetilde{\alpha}^\perp)=4\operatorname{Cox}^*(\Phi)-6.\]
In particular 
\[\#(\Phi^+\setminus\widetilde{\alpha}^\perp)=2\operatorname{Cox}^*(\Phi)-3.\]\qed
\end{thm}


\begin{cor}Let $n\in{\g g}_{{\rm nil}}\setminus\{0\}$. The lower line in the following diagram gives a lower bound for  
\[
\frac12\codim Z_{\g g}(n).\]
\[
\begin{tabular}{|l|c|c|c|c|c|c|c|c|c|}
\hline 
\text{Type of root system $\Phi$}
&
$\mathrm A_r$ 
& 
$\mathrm B_r$
&
$\mathrm C_r$
&
$\mathrm D_r$ 
&
$\mathrm E_6$
&
$\mathrm E_7$
&
$\mathrm E_8$
&
$\mathrm F_4$
&
$\mathrm G_2$
\\
\hline
\text{lower bound} & r& 2r-2&r&2r-3&11&17&27&8&3.
\\
\hline\end{tabular} 
\]
\end{cor}
From eq. \eqref{03.04.2024--3} we obtain:

\begin{cor}\label{03.04.2024--1}If $\g g$ is not of type   $\mathrm A_1,\mathrm A_2,\mathrm A_3$,  $\mathrm B_2$, $\mathrm C_3$ or $\mathrm G_2$ then, for any $n\in{\g g}_{{\rm nil}}\setminus\{0\}$, we have   $\codim (\mathcal B_{n},\mathcal B)\ge4$.\qed\end{cor}    

\subsection{The codimension of $\cb_{\g h}$}\label{01.03.2024--1} 
Let us now assume that  $\dim\g h=2$. We then suppose that the root system $\Phi$ of $\g g$ is not of type  $\mathrm A_1,\mathrm A_2,\mathrm A_3$,  $\mathrm B_2$, $\mathrm C_3$ or $\mathrm G_2$. Then, by Corollary \ref{03.04.2024--2},   Corollary \ref{03.04.2024--1} and Lemma \ref{03.04.2024--4}, we have  
\[\codim(\cb_{\ell},\cb)\ge4\qquad\text{for any $\ell\in \mathbb P(\g h)$.}
\]
Consequently, 
\[
\codim(\cb_{\g h},\cb)\ge3
\]
in these cases because of eq. \eqref{20.02.2024--1}.

\subsection{Singularities of   fundamental vector fields on   the Borel variety and foliations associated to $\SLie(2,\g g)$}\label{01.04.2024--2}

Let us identify $\mathcal B$ and  $G/A$ 
 and let $\pi:G\to G/A$ stand for the natural projection. Under this identification,  the action of $G$ on the left of $\mathcal B$ by conjugation corresponds  to   the obvious action of $G$ on the left of $G/A$. In addition, letting $G$ act  on itself  by left-translations, it is obvious that $\pi$ is $G$-equivariant. For each $v\in \g g=\mathrm{Lie}\,G$, let $v^\natural$ be the fundamental vector field associated to $v$.

The action of $G$ on $\mathcal B$ allows us to construct a map of Lie algebras    $\g g\to H^0(\mathcal B,T_{\mathcal B})$, which shall be denoted by    $v\mapsto v^\natural$; the vector field $v^\natural$ is the  ``fundamental'' vector field on $\cb$. Let us briefly recall the definition of this classical object.  If $\vartheta:G\to \mathcal B$ is the orbit morphism sending $1\in G$ to $x\in\mathcal B$, then $v^\natural(x)=\mathrm D_{1}\vartheta(v)$. 
(Here $\mathrm D_x\pi$ is the derivative between tangent spaces.)

\begin{lemma}\label{26.11.2024--1}The vector field $v^\natural$ has a zero at $\pi(x)$ if and only if $v\in \mathrm{Ad}_x(\g a)=\mathrm{Lie}\,(xAx^{-1})$. Otherwise said, $v^\natural$ has a zero at the Borel subgroup $B\in \cb$ if and only if $v\in \mathrm{Lie}\,B$.
\end{lemma}

\begin{proof}For $y\in G$, let $R_y$ and $L_y$ be respectively the right and left translations by $y$ on $G$. 
Let $\vartheta:G\to \mathcal B$ be the orbit morphism sending the neutral element $e$ to $x$; we have $v^\natural(\pi(x))=\mathrm D_e\vartheta (v)$. Since $\pi  R_x=\vartheta$, we conclude that 
\[
v^\natural(\pi(x)) = \mathrm D_x\pi\left( \mathrm D_eR_x( v)\right).
\]
Now, the map $\mathrm D_x\pi:T_xG\to T_{\pi (x)}\cb$ is a surjection and its kernel is $\mathrm D_eL_x(\g a)$ because $\pi^{-1}(\pi(x))=xA$. 
Hence, $v^\natural(\pi(x))=0$ if and only if $\mathrm D_eR_x(v)\in \mathrm D_eL_x(\g a)$, which is the case if and only if  $v\in \mathrm D_e(R_{x^{-1}}L_x)(\g a)=\mathrm{Ad}_x(\g a)$.

To verify the last claim, we observe that the point $\pi(x)$ of $\cb$ corresponds to the Borel subgroup $xAx^{-1}$. 
\end{proof}

Consider the $\co_\cb$-linear morphism 
\[\gamma:
\co_\cb\otimes \g g \longrightarrow T_\cb
\]
determined by sending $v\in\g g$ to the field $v^\natural$. 

\begin{prop}Let $\g h$ be a {\it two-dimensional} subalgebra   of $\g g$ and suppose that   $\g g$ is not of type $\mathrm A_1,\mathrm A_2,\mathrm A_3,\mathrm  B_2,\mathrm C_3$ or $\mathrm G_2$. 
\begin{enumerate}[(1)]
\item The morphism $\gamma:\co_\cb\ot\g h\to T_\cb$ is injective and its image is saturated. 
\item If $\mathcal A(\g h)$ is the   (involutive) distribution determined by $\gamma(\co_\cb\ot\g h)$, then $\codim\sing\ca(\g h)\ge3$.
\end{enumerate}
\end{prop}
\begin{proof}From the assumptions of $\g h$, we know from Section \ref{01.03.2024--1} that  
$\codim\cb_{\g h}\ge3$. 
Now, for a closed point $x\in\cb$, the map of $\Bbbk$-spaces
\[\gamma(x): \left(\co_{\cb}\ot\g h  \right)(x)\aro T_\cb(x)\]
will fail to be injective precisely when for a certain $v\in \g h\setminus\{0\}$ we have $v^\natural(x)=0$, since 
\[1\otimes{\rm id} :\g h\aro \left(  \co_{\cb}\ot\g h\right)(x)\]
is an isomorphism.   Hence, letting 
\[
\Sigma:=\left\{x\in\cb\,:\,\begin{array}{c}\text{$\gamma(x):\left(\co_{\cb}\ot\g h\right)(x)\to T_\cb(x)$}
\\ 
\text{fails to be injective}
\end{array}
\right\},
\]
it follows that $\cb_{\g h}$ is the set of closed points in $\Sigma$. 
As a consequence,    $\gamma :\co_{\cb}\ot\g h\to T_\cb$ is injective over $\cb\setminus\Sigma$, and on this open subset   $\gamma(\co_{\cb}\ot\g h)$ is a subbundle of  $T_\cb$  \cite[Theorem 22.5, p.176]{matsumura}. From Proposition \ref{16.11.2023--1}, we see that $\gamma(\co_{\cb}\ot\g h)$ is a strongly saturated $\co_\cb$-submodule of $T_\cb$ and $\gamma(\co_{\cb}\ot\g h)$ thus defines a distribution. 
Moreover, we see that $\cb_{\g h}$ is none other than the set of closed points in  $\mathrm{Sing}(\mathcal A(\g h))$.
\end{proof}

\section{Irreducible components of the space of  foliations in the case of a Borel variety}
 
\subsection{The scheme of distributions}
We fix base field   $\Bbbk$ and let   $X$ be a smooth, proper and geometrically connected $n$-dimensional $\Bbbk$-scheme. In what follows, unadorned direct products  are taken over $\Spec\Bbbk$ and, for a given $\Bbbk$-scheme $S$, or $\co_S$-module $\cm$ over it, we shall write $X_S$, respectively $\cm_S$, for the product $X\ti S$, resp. for the pull-back of $\cm$ to $X\ti S$.

Let $L$ be an invertible sheaf on $X$ and   $q$ an integer in $\{1,\ldots,n\}$. 
Denote by $A$ the affine scheme of twisted $q$-forms with values on $L$  so that we have an universal $q$-form 
\[\Phi: \wedge^qT_{X_A/A} \aro L_A.\]
Let $B\subset A$ be the open subscheme  of $A$ consisting of all points $a\in A$ such that 
\[
\codim(\mathrm{Sing}(\Phi)\cap X_a,X_a)\ge2. 
\]
Note that $\codim\mathrm{Sing}(\Phi|_{X_B})\ge2$.  For the sake of simplicity, let us abuse notation and denote by $\Phi$ also the pull-back of $\Phi$ to $X_B\subset X_A$. 
Let $\mathrm{Ker}_{\Phi}$ stand for the $\co_{X_B}$-submodule of   $T_{X_B/B}$ defined by the form $\Phi$  as in eq. \eqref{11.05.2024--1}. Consider the {\it closed} subscheme $D$ of $B$ with the following universal property: For each morphism $S\to B$, the sheaf $\mm{Ker}_\Phi\ti_B S$ is  free of rank $n-q$ on the generic point of each fibre $X\ti\{s\}$    if and only  $S\to B$ factors through $D$ \cite[Part 1, Lemma 4.3.2]{raynaud-gruson}.  

\begin{lemma}Let $u:S\to B$ be a morphism and let $\psi:X_S\to X_B$ be the natural map. Then $\psi^*(\omega)$ is LDS if and only if $u$ factors through $D$. 
\end{lemma}
\begin{proof}Let $\xi$ be the generic point of $X$.  For each $S$ and each $s\in S$,  we let $\xi_s$ stand for the generic point of $X\times\{s\}$. Note that for each $b\in B$, we have $\xi_b\in\mathrm{Reg}(\Phi)\cap X\times \{b\}$ by construction. Now, if $s\in S$ is sent to $b\in B$, then   $\psi(\xi_s)=\xi_b$ and $\xi_s\in\mathrm{Reg}(\psi^*\omega)$.
We note that $\psi^*(\omega)$ is very irreducible because   $\mathrm{Reg}(\psi^*(\omega))=\psi^{-1}(\mathrm{Reg}(\omega))$ and hence $\mathrm{Reg}(\psi^*\omega)\cap X\times\{s\}$ is always big in $X\times\{s\}$. 
 Consequently,  
\[\left(\psi^*(\mathrm{Ker}_{\omega})\right)_{\xi_s}\simeq\left( \mathrm{Ker}_{\psi^*(\omega)}\right)_{\xi_s}.\]

Suppose that $\psi^*(\omega)$ is LDS. We must show that $\mathrm{Ker}_{\psi^*(\omega)}$ is free of rank $n-q$ on $\xi_s$. Note that, since $u$ takes values in $B$, the form $\psi^*(\omega)$ is immediately very irreducible.  Then, $\left(\mathrm{Ker}_{\psi^*(\omega)}\right)_{\xi_s}$ has rank $n-q$ (Lemma \ref{04.06.2024--1}). It then follows that the {\it free} $\co_{X_S,\xi_s}$-module   $\left(\psi^*(\mathrm{Ker}_{\omega})\right)_{\xi_s}$ has rank $n-q$.

Conversely, suppose that for each $s\in S$,   the $\co_{X_S,\xi_s}$-module $\left(\psi^*(\mathrm{Ker}_{\omega})\right)_{\xi_s}$ is free of rank $n-q$ so that $\left( \mathrm{Ker}_{\psi^*(\omega)}\right)_{\xi_s}$ is likewise. 
 Since $\mathrm{Ass}(X\times S)=\{\xi_s\,:\,s\in\mathrm{Ass}(S)\}$, we conclude that $\mathrm{Ker}_{\psi^*(\omega)}$ is always free of rank $n-q$ on the associated points of $X\times S$. By Lemma \ref{04.06.2024--1}, we conclude that $\psi^*(\omega)$ is LDS. 
\end{proof}

\begin{dfn}The scheme $D$ is called the scheme of LDS forms on $X$ of codimension $q$ and determinant $L$. It will be denoted by $\mathbf{Dist}^q(X,L)^+$. The quotient of $\mathbf{Dist}^q(X,L)^+$ by the natural action of $\mathbb G_m$ is the {\it scheme of distributions of codimension $q$ and determinant $L$} and is denoted by $\mathbf{Dist}^q(X,L)$.  For the sake of brevity, $\mathbf{Dist}_r(X,L)$ shall stand for $\mathbf{Dist}^{n-r}(X,L)$.
\end{dfn}


Let finally $\mathbf{Fol}^q(X,L)^+$ stand for the scheme of foliations on $X$. This is the closed subscheme of $D:=\mathbf{Dist}^q(X,L)^+$ on which the $\co_{X\times D}$--linear map 
\[\wedge^2\mm{Ker}_\Phi\aro \wedge^2T_{X\ti D/D} \aro
\frac{T_{X\ti D / D}}{\mm{Ker}_\Phi} \]
defined by the bracket vanishes. The quotient of $\mathbf{Fol}^q(X,L)^+$
by the natural action of $\mathbb G_m$ is the scheme of foliations of determinant  $L$ and codimension $q$ on $X$ and is denoted by $\mathbf{Fol}^q(X,L)$. 

For the sake of brevity, we let $\mathbf{Fol}_r(X,L)$ stand for $\mathbf{Fol}^{n-r}(X,L)$. It is also convenient to 
denote  the underlying reduced schemes by dropping the boldface typeset: 
 $\mathrm{Dist}^q(X,L)$, $\mathrm{Dist}_r(X,L)$, etc.

\subsection{Irreducible components of the space $\mathbf{Fol}$}  

Let us place ourselves in the situation described in Section \ref{22.02.2024--2}: $G/\Bbbk$ is an adjoint linear algebraic group with {\it simple} Lie algebra $\g g$. Instead of the more traditional $\cb$, we shall let  $X$ stand for  the Borel variety of $G$ and denote by $n$ its dimension. Finally, we exclude the cases where    $\g g$ is   of type $\mathrm A_1,\mathrm A_2,\mathrm A_3,\mathrm  B_2,\mathrm C_3$ or $\mathrm G_2$.

Let 
\[\Psi:
\wedge^2\g g\aro H^0(X,\wedge^2T_X)
\]
be the  linear map defined by $v\wedge w\to v^\natural\wedge w^\natural$. 
Assume that  $v^\natural\wedge w^\natural=0$; then for each $x$ closed point, there exists $\lambda,\mu\in\Bbbk$ such that $(\lambda v+\mu w)^\natural(x)=0$. Consequently, if $\Bbbk \cdot v\wedge w=\wedge^2\g h$ for some $\g h\in\SLie(2,\g g)(\Bbbk)$, then $x\in X_{\g h}$. This contradicts the bound obtained in Section \ref{01.03.2024--1}, unless $v\wedge w=0$. Hence, no element of $\wedge^2\g g$ of the form $v\wedge w$ with $\Bbbk v+\Bbbk w$ a two-dimensional subalgebra of $\g g$, belongs to $\mathrm{Ker}(\Psi)$. 

Let now 
\[
\mathbb P\Psi: \mathbb P\left(\wedge^2\g g\right)\setminus\mathbb P(\mathrm{Ker}(\Psi))\aro \mathbb P\left(H^0(X,\wedge^2T_X)\right)
\]
be the morphism obtained from $\Psi$ by passing to the associated projective spaces. 
From what we explained above, the Pl\"ucker immersion 
\[ 
\mathrm{Gr}(2,\g g)\aro    \mathbb P\left(\wedge^2\g g\right)
\]
sends $\SLie(2,\g g)$ into $\mathbb P\left(\wedge^2\g g\right)\setminus\mathbb P(\mathrm{Ker}(\Psi))$ and we then obtain a morphism 
\[\psi:\SLie(2,\g g)\aro \mathbb P\left(H^0(X,\wedge^2T_X)\right),\]
with the property that for any two generators $v,w$ of $\g h\in\SLie(2,\g g)(\Bbbk)$, we have $\psi(v\wedge w)=\Bbbk\cdot v^\natural\wedge w^\natural$. The fact that the natural pairing $\wedge^2T_X\otimes_\co\wedge^{n-2}T_X\to \det T_X$ is non-degenerate gives us an  isomorphism 
\[
H^0(X,\wedge^2T_X)\simeq H^0(X,\Omega_X^{n-2}\ot \det T_X)
\]
and consequently a morphism 
\[
\SLie(2,\g g)\aro\mathbb P\left( H^0(X,\Omega_X^{n-2}\ot \det T_X)\right).
\]
Another simple verification shows that $\psi$ factors through the reduced subscheme $\mathrm{Fol}_2(X,\det T_X)$ and that for each $\g h\in \SLie(2,\g g)(\Bbbk)$, the element $\psi(\g h)$ is the twisted $(n-2)$-form with values on $\det T_X$ associated to the action distribution (in fact a foliation) $\mathcal A(\g h)$ introduced in Section \ref{01.04.2024--2}. 

\begin{thm}\label{10.06.2024--2}
The morphism $\psi$ is open. In particular, if $\Sigma$ is an irreducible component of $\SLie(2,\g g)$, then $\overline{\psi(\Sigma)}$ is an irreducible component of $\mathrm{Fol}_2(X,\det T_X)$. 
\end{thm}
\begin{proof}We start by explaining some preliminary material. First,  due to  \cite[Lemma 1]{serre59},  a trivial vector bundle on $X$ is  rigid, since $X$ is a smooth and rational variety. Then, the natural morphism of Lie algebras   $\g g\to H^0(X,T_X)$ 
obtained from the left action of $G$ on $X$ is an isomorphism, see Theorem 1, Proposition 1 and the remark preceding the beginning of Section 3 in  \cite{demazure}. 

Let $\g h\in\SLie(2,\g g)(\Bbbk)$ be given. Then the foliation $\mathcal A(\g h)$ on $X$ is such that 
\[
\codim\left(X\setminus\mathrm{Sing}(\mathcal A(\g h))\,,\, X\right)\ge3
\]
because of what is explained in Section \ref{01.04.2024--2}.
Consider now the universal relative foliation $\mathcal U$ on $X\times \mathrm{Fol}_2(X,\det T_X)$.  Then 
\[
\mathcal U_{\psi(\g h)}\simeq \mathcal A(\g h).
\]
Because the closed subsets 
$\mathrm{Sing}(\mathcal A(\g h))$ and $\mathrm{Sing}({\mathcal U})\cap X_{\psi(\g h)}$ agree, we conclude that the open subset 
\[B=\{b\in \mathrm{Fol}_2(X,\det T_X)\,:\,\codim(\sing(\mathcal U)\cap X_b,X_b)\ge3\}\]
is non-empty.    Then, there exists an open neighbourhood $B'$ of 
$\psi(\g h)$ in    $B$ such that   $T_{ \mathcal U_b}$ is free of rank two for each $b\in B'$  (Corollary \ref{21.03.2024--2}). Let $\g i_b=H^0(X,T_{{\mathcal U}_b})$; this is a two-dimensional subalgebra of $\g g=H^0(X,T_X)$.  
Consequently, for each $b\in B'(\Bbbk)$, the distribution ${\mathcal U}_b$ is of the form $\psi(\g i_b)$ for some $\g i_b\in \SLie(2,\g g)(\Bbbk)$, which shows that $\psi$ is open. 
\end{proof}

\begin{cor}Assume that $\g g$ is of type $\mathrm A_r, \mathrm B_r, \mathrm C_r$ or $\mathrm D_r$.  Then the number of irreducible components of $\mathrm{Fol}_2(X,\det T_X)$ tends to infinity when $r\to+\infty$.  
\end{cor}
\begin{proof}This can be proved using the theory of ``special'' partitions,  see  Theorems 5.1.2--4 in \cite{collingwood-mcgovern93}. 
\end{proof}

\appendix
\section{Generalities on Commutative Algebra }\label{generalities}

In certain cases, it is important to work with distributions on relative schemes and, it turns out, even in schemes having complicated underlying algebra: non-reduced, having embedded components, with sheaves of meromorphic functions which are not coherent, etc. In order to gain ground on this theme, we thought useful to interpose the following sections dealing with basic, but less widespread, commutative algebra. 
 
\subsection{Torsion theories}
Throughout this section and Section \ref{orthogonals}, we fix a Noetherian ring $A$. Its total ring of fractions is denoted by $K$. 
Under this generality, the notion of {\i torsion} in an $A$-module is  not as simple as in the case of a domain, and, it turns out, there are two relevant paths to follow.

\begin{dfn}\label{15.11.2023--1}
Let $M$ be a finite $A$-module. 

\begin{enumerate}[i)]
\item   We say that $M$ is   {\it torsion-free} if $M\to K\otimes M$ is injective.  
We say that  $M$ is {\it torsion-less} if the canonical map $M\to M^{\vee\vee}$ is injective (cf. \cite[1.4]{bruns-herzog93}).

\item 
The torsion-free, respectively torsionless, quotient of $M$ is the image of $M$ in $K\otimes M$, respectively in $M^{\vee\vee}$. These shall respectively be denoted by $M_{\rm{tof}}$ and $M_{\rm{tol}}$. 
\end{enumerate}
\end{dfn}

The  above definition has a   slight imprecision  since we did not establish that $M_{\rm{tof}}$, respectively    $M_{\rm{tol}}$, is torsion-free, respectively torsion-less. That this is the case for $M_{\rm{tof}}$ is simple to verify, while for  $M_{\mm{tol}}$ this is a consequence of:

\begin{prop}[{cf. \cite[Proposition 16.31]{bruns-vetter80}}]\label{15.11.2023--2}A finite module  $M$ is torsion-less if and only if it  is a submodule of a free module of finite rank. \qed
\end{prop}

It comes as no surprise that $M\to M_{\mathrm{tof}}$, resp. $M\to M_{\rm{tol}}$, is ``universal'' with respect to morphisms to torsion-free, resp. torsion-less, modules. 
Here is another fundamental property:

\begin{lemma}\label{24.06.2024--1}Let $M$ be a finite  $A$-module. Then the epimorphism $M\to M_{\rm{tol}}$  factors through the   epimorphism $M\to M_{\rm{tof}}$. In particular, each torsion-less module is also torsion-free. 
\end{lemma}

With the previous results in sight, the following definition is justified.

\begin{dfn}Let $M$ and $F$ be $A$-modules and $M\to F$ an injection. Then   the {\it strong saturation} of $M$ in $F$ is  \[M^{\rm ssat}=\mathrm{Ker}\,F\to(F/M)\to(F/M)_{\rm{tol}}.\] 
The {\it saturation} is \[M^{\rm sat}=\mathrm{Ker}\,F\to(F/M)\to(F/M)_{\rm{tof}}.\]
\end{dfn}

Here is a simple working property around the previous definition.  

\begin{lemma}\label{24.06.2024--3}Let $M$ be a submodule of $F$. Then 
\[M^{\rm{ssat}}=\bigcap_{\substack{\lambda\in F^\vee\\\lambda|_M=0}}\mathrm{Ker}\,\lambda \quad\text{and}\quad M^{\rm{sat}}=\left\{x\in F\,:\,\begin{array}{c}\text{$ax\in M$ for some regular} \\ \text{    element $a\in A$}\end{array}\right\}.\qed
\]
\end{lemma}

A fundamental result of Vasconcelos allows us to determine when the notions of torsion-free and torsion-less agree thus providing a converse to Lemma \ref{24.06.2024--1}. This result employs two fundamental properties  of Commutative Algebra:  Gorenstein rings  (cf. \cite[\S18]{matsumura}, specially Theorem 18.1) and Serre's conditions $\mathrm S_n$ \cite[\S23,  p.183]{matsumura}. 
The reader is asked to bear in mind the following summary scheme:   
\[
\begin{split} \mathrm S_0&\text{ is always satisfied,}
\\ \mathrm S_1&\Leftrightarrow\text{there are no embedded primes,}
\end{split}
\]
\[\text{$A$ is reduced }\quad\Leftrightarrow\begin{array}{c}\text{$A$ is $\mathrm S_1$ and for}
\\
\text{  each  minimal $\g p$} 
\\ 
\text{the ring $A_{\g p}$ is a field.}
\end{array}\]
If $(A,\g m,k)$ is a zero dimensional Noetherian local ring, then 
\[\text{$A$ is Gorenstein} \quad\Leftrightarrow\quad\mathrm{Hom}_A(k,A)\simeq k.\]

It is expedient to make the following definition. 

\begin{dfn}[{\cite{hartshorne94}}]We say that a Noetherian scheme $X$    satisfies property  $\mathrm G_n$ if for all $x\in  X$ of codimension $\le n$, i.e.   $\dim\mathcal O_{X,x}\le n$, the local ring $\mathcal O_{X,x}$ is Gorenstein. A Noetherian ring satisfies $\mm G_n$ if its spectrum satisfies $\mm G_n$. 
\end{dfn}

We can then clearly state Vasconcelos' result.

\begin{thm}[{ \cite[Theorem A.1]{vasconcelos68}}]
The following conditions are equivalent. 
\begin{enumerate}[i)]\item Every finite torsion-free $A$-module is torsion-less. 
\item The ring  $A$ enjoys $\mathrm G_0$  and   $\rm S_1$. 
\end{enumerate}
In particular, if $A$ is reduced, then each torsion-free module is automatically torsion-less and vice-versa. 
\end{thm}

\begin{rmk}
A fundamental fact which lies at the heart of the previous result is that {\it all} finite torsionless modules are reflexive over a   local Gorenstein ring of dimension at most one \cite[Theorem 6.2]{bass}. 
\end{rmk}

We record another simple property 
around torsion-less modules. (It is a consequence of the ``commutation of flat base-change and duals'' \cite[Theorem 7.11]{matsumura}.)

\begin{lemma}Let $A\to B$ be a flat morphism of noetherian rings and   $M$ a finite $A$-module. If $M$ is torsion-less, then   $B\otimes_AM$ is likewise. \qed
\end{lemma}

Let us then end this section with some illustrations. 
\begin{ex}[On the rank] In opposition to the case of integral domains, torsion-less modules  do not necessarily have a ``reasonable   rank''. 
Let $P=\C[x,y]$ and let $f$ be an irreducible  polynomial. Define $A=P/(f^2)$ and   $I= fA$. Then, for each $\g p\in\mathrm{Spec}\,A$, the $A_{\g p}$-module  $I_{\g p}$ is annihilated by the non-zero element  $f/1 \in A_{\g p}$.  Hence, $I$ has  ``rank zero'' although being   torsion-less. 
\end{ex}

\begin{ex}[Failure of Lemma \ref{24.06.2024--1}]Let $A=\Q[x,y]/(x^2,xy,y^2)=\Q1\oplus\Q x\oplus \Q y$; this is a local ring with maximal ideal $\g m=\Q x\oplus\Q y$. Let $M=A/(x)$. 
It is not difficult to see that $M$ is torsion-free but not torsion-less. Note here that $A$ {\it is not Gorenstein}.
\end{ex}

\subsection{Orthogonals and strong saturation}\label{orthogonals}
We fix a Noetherian ring $A$. 
Let $F$   be a projective $A$-module of finite rank and  $\langle-,-\rangle:F  \times F^\vee\to A$  the canonical pairing. For $M\subset F$, we write $M^\perp=\{\theta\in F^\vee\,:\,\langle m,\theta\rangle=0,\,\,\forall\,\,m\in M\}\simeq (F/M)^\vee$.

\begin{prop}\label{13.11.2023--1}Let $M\subset   F$ be a submodule as above. 
\begin{enumerate}[(1)]
\item The natural inclusions $M\subset  M^{\mm{sat}}\subset  M^{\mm{ssat}}$ induce equalities 
$(M^{\mm{ssat}})^\perp=  (M^{\mm{sat}})^\perp= M^\perp$.
\item 
The submodule $M^\perp$ of $F^\vee$ is strongly saturated. 
\item The natural monomorphism $u:M\to M^{\perp\perp}$ defined by $u_m(\varphi):=\langle  m ,\varphi\rangle$ is an isomorphism if and only if $M$ is strongly saturated in $F$. 

\item We have $M^{\perp\perp}=M^{\mm{ssat}}$.

\end{enumerate} 
\end{prop}

\begin{proof}{\it Item (1)}. Clearly $M^\perp\supset (M^{\mm{sat}})^\perp\supset(M^{\mm{ssat}})^\perp$. Let now $\theta:F\to A$ belong to $M^\perp$. By definition, $\theta|_M=0$ and hence $M^{\rm ssat}\subset\mm{Ker}(\theta)$ (Lemma \ref{24.06.2024--3}), so that $\theta\in (M^{\rm ssat})^\perp$. 

\noindent{\it Item (2)}. By Lemma \ref{24.06.2024--3} we have 
\[
\begin{split}(M^\perp)^{\rm ssat}&=\bigcap_{\substack{x\in F\\ \langle x,M^\perp \rangle=0}}\mm{Ker}\langle x ,-\rangle.
\end{split}
\]
Now $x\in F$ is annihilated by each $\lambda\in M^\perp$ if and only if $x\in M^{\mathrm{ssat}}$ by the same Lemma. Hence, $(M^\perp)^{\mathrm{ssat}}=\bigcap_{x\in M^{\mathrm{ssat}}}\mathrm{Ker}\,\langle x,-\rangle$. But by definition the former intersection is just $(M^{\mathrm{ssat}})^\perp$, which equals $M^\perp$ by (1).

\noindent{\it Item (3)}.  Consider an exact sequence 
\[
0\longrightarrow M\stackrel i\longrightarrow   F\stackrel p\longrightarrow    Q\longrightarrow0.
\] 
Then, we have a commutative diagram with exact rows
\[
\xymatrix{
0\ar[r]&M\ar[r]^i\ar[d]_{u}&F\ar[rr]^p\ar@{=}[d]&&Q\ar[d]^v\ar[r]&0 
\\
0\ar[r]&M^{\perp\perp}\ar[r]&F\ar[rr]_-{(p^\vee)^\vee}&&Q^{\vee\vee}
}
\]
and from the Snake Lemma, we have 
\[
\mathrm{Ker}(u)=0\quad\text{and}\quad \mathrm{Coker} \,u \simeq \mathrm{Ker}\,v.
\]
Hence, $M\to M^{\perp\perp}$  is an isomorphism if and only if $\mathrm{Ker}\,v=0$, which is equivalent to $Q$ being torsionless.
 
\noindent{\it
Item (4)}. Follows from (1) and (3). 
\end{proof}

\subsection{Torsion in the case of sheaves}Let $X$ be a Noetherian scheme with sheaf of meromorphic functions     $\ck_X$ \cite{kleiman79}. 
In analogy with Definition \ref{15.11.2023--1}-(i),   EGA puts forth the following.

\begin{dfn}\label{strictly-torsion-free} An  $\co_X$-module $\cm$ is  {\it strictly torsion-free} \cite[$\mathrm{IV}_4$, 20.1.5]{ega} if the natural map $\cm\to\ck_X\ot\cm$ is injective.
\end{dfn}
As $X$ is Noetherian, we know that $\ck_{X,x}$ is simply the total ring of fractions of $\co_{X,x}$ \cite[Lemma 7.1.12]{liu02} and hence $\cm$ is strictly torsion-free if and only if each $\cm_x$ is torsion-free. 
Now, the $\co_X$-module $\ck_X$ is {\it not  necessarily quasi-coherent} \cite[p.205]{kleiman79}, which renders the above definition unwieldy   in general: indeed, in order to construct the ``strictly torsion-free quotient of $\cm$'' we would need to look at the image of $\cm$ in $\cm\ot\ck_X$. (We note here that if $X$   satisfies Serre's $\mathrm S_1$ property \cite[\S23, p.183]{matsumura},  then  $\ck_X$ {\it is} quasi-coherent \cite[Proposition 2.1(d)]{hartshorne94}.)   
Therefore, the notion of torsion-less   is better suited to this level of generality than that of torsion-free. 

Let $\cm$ be a coherent sheaf and $\cf$ a locally free sheaf of finite rank throughout.

\begin{dfn}We say that $\cm$ is {\it torsion-less} if the natural map $\chi:\cm\to\cm^{\vee\vee}$ is injective. The {\it torsion-less quotient} of $\cm$, denoted $\cm_{\rm{tol}}$, is  $\mathrm{Im}(\chi)$. 
\end{dfn}
Needless to say, in the above definition, $\mathrm M_{\rm{tol}}$ is torsion-less, since it is a submodule of the torsion-less $\co_X$-module    $\cm^{\vee\vee}$.

We suppose from now on that   $\cm   \subset \cf$.

\begin{dfn} We define the {\it strong saturation} of $\cm$ in $\cf$ as being the kernel of the composition $\cf\to(\cf/\cm)\to(\cf/\cm)_{\mathrm{tol}}$. Notation: $\cm^{\rm{ssat}}$. 
\end{dfn}

Note that each torsion-less coherent $\co_X$-module is automatically {  strictly torsion-free}. Also, observe that, by construction, the strong saturation of $\cm$ in $\cf$ is always coherent. 

Let us now state Proposition \ref{13.11.2023--1} in terms of sheaves. Write $\langle-,-\rangle:\cf\ti\cf^\vee\to\co_X$ for the canonical pairing and  define   $\cm^\perp$ by $\cm^\perp=\mathrm{Ker}\,\cf^\vee\to\cm^\vee$. Clearly, if $U$ is an affine open subset of $X$, then $\cm^\perp(U)=\cm(U)^\perp$, in the notation of Section \ref{orthogonals}. 
Proposition \ref{13.11.2023--1} now easily implies the following result.

\begin{prop}\label{17.11.2023--1}
\begin{enumerate}\item 
The submodule $\cm^\perp$ of $\cf^\vee$ is strongly saturated. 
\item The natural monomorphism   $\cm\to \cm^{\perp\perp}$
  is an isomorphism if and only if $\cm$ is strongly saturated in $\cf$. 

\item We have $\cm^{\perp\perp}=\cm^{\mm{ssat}}$.

\end{enumerate} 
\end{prop}

Let us end this section with  simple results which allow us to recognise what kind of $\co_X$-module the saturation shall be. This is particularly useful when $X$ is a quasi-normal scheme (see Remark \ref{quasi-normal})  and when $\cm$ is already strongly saturated  on some open subset.


\begin{prop}\label{16.11.2023--1}Let $U\subset X$ be a non-empty open subset and write $\overline \cm$ for the canonical extension of $\cm|_U$ inside $\cf$ \cite[I, 9.4.1-2]{ega}, that is, $\Gamma(V,\overline\cm)=\{s\in \Gamma(V,\cf)\,:\,s|_{V\cap U}\in\cm(V\cap U)\}$.  
\begin{enumerate}[(1)]
\item \label{16.11.2023--1-1}Assume that $\mathrm{Ass}(X)\subset U$ (which means that $U$ is schematically dense) and that $\cm|_U$ is strongly saturated in $\cf|_U$. Then $\cm^{\rm{ssat}}=\overline\cm$.   

\item\label{16.11.2023--1-2} Suppose that $X$  is quasi-normal, i.e. it  satisfies $\mathrm G_1$ and $\mathrm S_2$. Also, suppose that  $U$ is big in $X$ and that     $\restr{\cm}{U}$ is strongly saturated in $\restr{\cf}U$. Then the $\co_X$-module $\cm^{\rm{ssat}}=\cm^{\rm{sat}}$ is isomorphic to $\cm^{\vee\vee}$. In particular, if $\cm$ is reflexive, then $\cm=\cm^{\rm{ssat}}=\cm^{\rm{sat}}$. 
\end{enumerate}
\end{prop}
\begin{proof}
{\it Item (1)}
Let $j:U\to X$ be the inclusion. By definition, $\overline\cm$ is the kernel of the natural map \[\cf\aro j_*\left(\restr{\cf}  U\right)\aro  j_* \left[\restr{(\cf/\cm)} U\right], \]
 and hence 
$\cm^{\rm{ssat}}\subset\overline\cm$ because $\cm^{\rm{ssat}}$ is taken to zero under the above map. Also, 
$\cf/\overline\cm$ is strictly torsion-free (cf. Definition \ref{strictly-torsion-free}). 
Indeed, 
if $V$ is any open subset of $X$ and  $a\in\co(V)$ annihilates $\restr{(\cf/\overline\cm)} V$,  then $a|_{U\cap V}$ annihilates $\restr{(\cf/\cm)} {V\cap U}$ which implies that $a=0$ since $\restr{(\cf/\cm)} U$ is strictly torsion-free  --- each torsion-less is strictly torsion-free --- and 
$V\cap U\supset\mathrm{Ass}(V)$. (Here we are using that $\co_X$ injects into $j_*\co_U$.)
  Because  $\cf/\cm^{\rm{ssat}}$ is torsion-less, it is strictly torsion-free. Consequently, $\overline\cm/\cm^{\rm{ssat}}$ is also strictly torsion-free and hence \cite[$\mathrm{IV}_4$, 20.1.6]{ega}
 \[\mathrm{Ass}\left(\overline\cm/\cm^{\rm{ssat}}\right)\subset \mathrm{Ass}(X)\subset U.\]

 Since $\overline\cm/\cm^{\rm{ssat}}$ vanishes on all its associated points, it follows that $\overline\cm/\cm^{\rm{ssat}}$ vanishes all-over \cite[$\mathrm{IV}_2$, 3.1.1.]{ega}.

{\it Item (2)}.
We know that $\cm^{\rm{ssat}}$ is reflexive \cite[Proposition 1.7]{hartshorne94}. We then obtain a factorization $\cm\to \cm^{\vee\vee}\to\cm^{\rm{ssat}}$. From the assumption, we get isomorphisms  \[\restr\cm U\stackrel\sim\aro \restr{\cm^{\vee\vee}}U\stackrel\sim\aro\restr{\cm^{\rm{ssat}}}U.\] We then conclude via \cite[Theorem  1.12]{hartshorne94} that $\cm^{\vee\vee}\simeq \cm^{\rm{ssat}}$.
\end{proof}

\begin{rmk}\label{quasi-normal}Schemes satisfying $\mathrm G_1$ and  $\mathrm S_2$     are called quasi-normal schemes in \cite[Definition 1.2]{vasconcelos68}. {\it Warning}: There exists in the literature the  notion of ``semi-normal'', which seems unrelated.  
\end{rmk}

\section{Generalities on the theory of distributions} \label{generalitiesbis}
Let $f:X\to S$ be a smooth morphism of Noetherian schemes. The relative tangent sheaf, which is a vector bundle, is denoted by $T_f$ in what follows. 

\begin{dfn}\label{10.06.2024--3}
A  distribution $\mathcal V$ on $X/S$ is a   coherent $\co_X$-submodule $T_{\cv}$ of   $T_{f}$ such that the quotient $T_f/T_\cv$  is   torsion-less. The sheaf $T_\cv$ is called the {\it tangent sheaf} of the distribution. The quotient $Q_{\cv}=T_f/T_\cv$ is called the {\it   quotient sheaf}.    
\end{dfn}
The reader must have noticed that we could have easily abandoned from the notation the symbol $T_\cv$: we avoided this because when studying restrictions, we want to avoid a notational conflict between the restriction of the distribution and the restriction of the tangent sheaf. 

The next definition recalls some standard    terminologies and introduces a new convenient one.

\begin{dfn}[Singularities of distributions]\label{def_of_singularities}
Let $\cv$ be a distribution on $X/S$. 

\begin{enumerate}[(a)]
\item A point $x\in X$ is called a {\it singularity} of the distribution $\cv$ if $Q_{\cv,x}$ fails to be a free $\co_{X,x}$-module. The set of all singularities is denoted by $\mathrm{Sing}(\cv)$. Similarly, a point which is {\it not} a singularity is called a {\it regular point}. The set of all regular points shall be denoted by $\mathrm{Reg}(\cv)$. 
\item  A point $x\in X$ is called a {\it tangent singularity}, or is said to be {\it tangent-singular},  if $T_{\cv,x}$ is  {\it  not} a free $\co_{X,x}$-module. The set of tangent singularities is denoted by $\mathrm{TSing}(\cv)$.
\end{enumerate}\end{dfn}

Using    standard Commutative Algebra (see for example \cite[Theorem 22.5, p.176]{matsumura}), it is a simple matter to prove the following result. 

\begin{lemma}\label{08.11.2023--1}
Let $\cv$ be a distribution defined by the submodule $\alpha: T_\cv\to T_f$. Then 
\begin{enumerate}[(i)]\item We have the equality 
\[
\mm{Sing}(\cv):=\left\{x\in X\,:\,\begin{array}{c}\text{$\alpha(x):T_\cv(x)\to T_{f}(x)$} \\ \text{is not injective}\end{array}\right\}.\]
\item Each tangent-singularity is also a singularity. In symbols: \[\mathrm{TSing}(\cv)\subset\mm{Sing}(\cv).\]
\end{enumerate}
\end{lemma}

  Lemma \ref{08.11.2023--1} shows that the  term ``tangent singularity'' is unambiguous. 
Notice that, in general, the inclusion envisaged in Lemma \ref{08.11.2023--1}-(ii) is strict.  For instance, if $T_{\mathcal V}= \mathcal O_X \cdot v$ for some vector field
$v$, then $\mathrm{TSing}(\mathcal V)= \varnothing$ while $\mathrm{Sing}(\mathcal V)$ coincides with the zero set of $v$, i.e. the set where $v_x\in\mathfrak m_xT_{X,x}$.

The definition  together with Proposition \ref{15.11.2023--2} show that, locally, a distribution   is constructed as follows. Let $U\subset X$ be an open subset where    $Q_\cv$ can be injected into a free module $\co_U^m$, cf. Proposition \ref{15.11.2023--2}. Then, letting   $T_{f}|_U\to Q_{\cv}|_U\to \co_U^m$ the composition be  defined by relative 1-forms $(\omega_1,\ldots,\omega_m):T_{f}|_U\to \co_U^m$, we conclude that $T_\cv|_U=\cap_i\mathrm{Ker}\,(\omega_i)$.

\subsection{Distributions and  twisted exterior forms}\label{twisted_form}

Let $f:X\to S$ be a smooth morphism of Noetherian schemes. Imposing certain regularity conditions on $S$, there exists a {\it fundamental and very concise way of expressing a distribution  employing  a twisted differential form}; the case of codimension one being rather well-known \cite{jouanolou79}, while in higher codimensions being due to  A.  de Medeiros, cf. \cite[Section 1]{medeiros77} and  \cite[Section 1]{medeiros}. 
The translation of the concept to the associated form serves then as a means to study the ``moduli'' of distributions. 

\subsubsection{From forms to distributions}\label{formstodistributions}
\begin{dfn}
Let $q\ge1$ be an integer and $L$ be an invertible sheaf on $X$. A global section $\om$ of $\Omega_f^q\ot L$ is called a twisted $q$-form with values on $L$.  
\end{dfn}

We being by explaining how to associate a distribution to certain  twisted forms. Let $q$ be a positive integer, $L$ an invertible sheaf on $X$ and  $\om$ a twisted $q$-form with values on $L$. The  $\co_X$-linear map $\wedge^qT_f\to L$   obtained from $\omega$ will be denoted likewise. 
Let now $\mathrm{Ker}_\om$ be the $\co_X$-submodule of $T_f$ defined 
by decreeing that on an  affine open subset  $U$ we have:  
\begin{eqnarray}\label{13.12.2023--1}
\mathrm{Ker}_\omega(U)&=&\{v\in T_f(U)\,:\, \boldsymbol i_v (\omega)=0\}
\\\nonumber
&=&\bigcap_{v_1,\ldots,v_{q-1}\in T_f(U)}\{v\in T_f(U)\,:\,  \omega( v_{1},\ldots,v_{q-1},v)=0\}.  
\end{eqnarray}
Otherwise said, 
\begin{equation}\label{11.05.2024--1}
\mathrm{Ker}_\omega=\mathrm{Ker}\left(\xymatrix{
 T_f\ar[r]^-{\mathrm{id}\ot\omega}& T_f\otimes \Omega^q_f\ot L\ar[rr]^{\text{contract}\ot\mathrm{id}}&&\Omega_f^{q-1}\ot L
}\right).
\end{equation}
Note that,    $T_f/\mathrm{Ker}_\omega$ is a submodule of a (locally) free module and hence is  torsion-less (Lemma \ref{15.11.2023--2}). Consequently:

\begin{lemma}\label{09.05.2024--1}The $\co_X$-submodule $\mathrm{Ker}_\om\subset T_f$
is strongly saturated.\qed
\end{lemma}

\begin{dfn}We let   
${\mathcal Z}(\omega)$ be the distribution on $X/S$ defined by  $T_{{\mathcal Z}(\omega)}=\mathrm{Ker}_\omega$. 
\end{dfn}
\newcommand{\contr}{ \,{\mathlarger{\mathlarger \lrcorner}}\, }

\begin{ex}\label{09.05.2024--2}Let $\{t_i\}_{i=1}^n$, respectively $\lambda$, be a basis of $T_{f}$, respectively $L$, over an affine and open subset $U$  and $\{\theta_i\}_{i=1}^n$ be the dual basis of $\{t_i\}$. 
 If  $\omega|_U=\theta_1\wedge\cdots\wedge\theta_q\otimes \lambda$, then $\mathrm{Ker}_{\omega }|_U=\oplus_{i=q+1}^{n}\co_U t_i$. 
\end{ex}

\begin{ex}Let $\cv$ be a distribution on $X/S$ such that $T_\cv$ is free on the basis $\{v_i\}_{i=1}^r$. 
  Now, if   $n$  is the  relative dimension of $X/S$ and $q=n-r$, we have a canonical isomorphism 
 \[\wedge^rT_f\arou\sim \Omega^q_f\ot \det T_f;
\]
let     $\omega_\cv \in H^0(X,\Omega^q_f\ot\det  T_f)$ correspond to $v_1\wedge\cdots\wedge v_r$. 
Assume that   $\mathrm{Ass}(X)\subset\mathrm{Reg}(\cv)$. Since  $\mathrm{Ker}_{\omega_\cv,x}=T_{\cv,x}$ for $x\in\mm{Ass}(X)$ --- as explained in Example \ref{09.05.2024--2} --- and since 
$\mathrm{Ker}_{\omega_\cv}$ and $T_{\cv}$ are both strongly saturated in $T_f$ (cf. Lemma \ref{09.05.2024--1}), we conclude that $\mm{Ker}_{\omega_\cv}=T_\cv$ because of Proposition \ref{16.11.2023--1}. 
\end{ex}

We now move on to relate   $\sing(\mathcal Z (\omega))$ with the more classical version of singular locus of a form. 

\begin{dfn}The regular locus $\mathrm{Reg}(\omega)$ of $\omega$ is 
\[
\{x\in X\,:\,\omega(x)\not=0\}
\]
and the  singular locus, $\mathrm{Sing}(\omega)$, is $X\setminus \mathrm{Reg}(\omega)$. 
\end{dfn} 

At this level of generality, the definitions are too weak. Indeed, $\mathrm{Sing}(\omega)$ may contain a generic point or   $\mathrm{rank}\,Q_{{\mathcal Z}(\omega)}\not=q$. (The notion of rank, as in  \cite{bruns-herzog93}, will be review briefly below.)
To deal with the first of the aforementioned difficulties, we only need the following definition (taken essentially from \cite{jouanolou79}).
\begin{dfn}We say that the form $\omega$ is {\it irreducible}  if $\mathrm{Reg}(\omega)$ is a big open subset of $X$. The form is said to be {\it very irreducible} if in addition to being irreducible, we have $\mathrm{Reg}(\omega)\supset\mathrm{Ass}(X)$.
\end{dfn}  
To deal with the second, 
the notion of    {\it decomposability}      is required. This is a well-known notion in the case of base fields (cf. \cite[p.210]{griffiths-harris}, \cite[3.4]{jacobson96}); it is less clear in the case of an arbitrary   base ring  unless one follows \cite{jacobson96}.

\begin{dfn}Let $A$ be a Noetherian local ring and $E$ a free $A$-module of   rank   $r$ over it. We say that $\varphi\in \wedge^qE$ is decomposable if there exists $\{e_1,\ldots,e_q\}\subset E$ which can be {\it completed to a basis} of $E$ such that $\varphi=e_1\wedge\cdots\wedge e_q$. 
\end{dfn}

In the notation of the previous definition, we see that $\varphi\in\wedge^qE$ is decomposable  only if $\varphi$ is part of a basis of $\wedge^qE$. 

\begin{dfn}[{\cite[Definition 1.2.1]
{medeiros}}]
We say that $\omega$ is {\it decomposable} at $x\in \mathrm{Reg}(\omega)$  if for
each identification $L_x\simeq \co_{X,x}$,
 the element $\omega_x\in\wedge^q\Omega^1_{f,x}$ is decomposable. We say that $\omega$ is LDS if  it is decomposable for each  $x\in\mathrm{Reg}(\omega)$. 
\end{dfn}

Using the Grassmann-Pl\"ucker equations \cite[3.4.10]{jacobson96}, the following is true. 

\begin{lemma}  The twisted form   $\omega$ is LDS if and only if $\omega_x\in\Omega_{f,x}^q$ is decomposable for each $x\in\mathrm{Ass}(X)\cap\mathrm{Reg}(\omega)$. 
\end{lemma}
\begin{proof}
Let   $x\in \mm{Reg}(\omega)$ and let $U$ be an affine and open neighbourhood of $x$ such that $T_f|_U$ comes with a basis   $\{t_i\}_{i=1}^n$ and where $L$ is trivial. Let   $\{\theta_i\}_{i=1}^n$ be the dual basis and, after picking an isomorphism $L|_U\simeq\co_U$, write     $\omega=\sum P_{h_1\ldots h_q}\theta_{h_1}\wedge\cdots\wedge\theta_{h_q}$. Then $\omega_x$ is decomposable if and only if, for any given $1\le i_1<\ldots<i_{q-1}\le n$ and $1\le j_1<\ldots <j_{q+1}\le n$,  the Grassmann-Pl\"ucker relations hold 
\[
\sum_{\ell=1}^{q+1}(-1)^{\ell}P_{i_1\ldots,i_{q-1} j_\ell} P_{j_1\cdots \hat{j}_\ell\ldots j_{q+1}}=0
\]
in $\co_x$. Since $a\in\co(U)$ vanishes if and only if $a_\xi=0$ for each $\xi\in\mathrm{Ass}(U)=\mathrm{Ass}(X)\cap U$ \cite[$\mathrm{IV}_2$, 3.1.8]{ega}, we are done.  
\end{proof}

\begin{lemma}\label{03.06.2024--1}Suppose that $\om$ is irreducible   and  LDS. Then    $ \mathrm{Reg}(\omega)= \mathrm{Reg}({\mathcal Z}(\omega))$. 
\end{lemma}
\begin{proof}Let $x\in\mathrm{Reg}(\omega)$. Using the LDS condition, let $\{t_i\}\subset T_{f,x}$,  and   $\{\theta_i\}\subset \Omega^1_{f,x}$ be bases in duality such that $\omega_x=\theta_1\wedge\cdots\wedge\theta_q$.    Then $\omega(t_{i_1},\ldots,t_{i_q})=0$ if $\{i_1,\ldots,i_q\}\not=\{1,\ldots,q\}$ and  $\pm1$ otherwise. From this it is easy to see that  $\oplus_{i=q+1}^n\co_xt_i=\mathrm{Ker}_{\omega,x}$ so that $x\in\mathrm{Reg}({\mathcal Z}(\omega))$ and $\mathrm{Reg}(\omega)\subset\mathrm{Reg}(\mathcal Z(\omega))$.

Suppose now that $x\in\mathrm{Reg}({\mathcal Z}(\omega))$. Then, for a certain open affine neighbourhood $U$ of $x$, the $\co_U$-module $T_{f}|_U$ possesses  a basis $\{t_i\}_{i=1}^n$ such that $\oplus_{i=1}^m\co_Ut_i=\mathrm{Ker}_{\omega}|_U$. If $\xi\in U$ is a generic point specializing to $x$, the fact that $\xi\in\mathrm{Reg}(\omega)$ allows us to  use the previous decomposition to show that $m=n-q$. 
Let $\{\theta_i\}$ be a dual basis to $\{t_i\}$. 
Using the standard formulas for contraction \cite[III.11.10]{bourbaki_algebra}, we see that  $\omega|_U=a\cdot\theta_{m+1}\wedge\cdots\wedge\theta_{n}$ for some $a\in\co(U)$.  But then $\mathrm{ Sing}(\omega)\cap U$ contains $\{a=0\}$; since $\codim\{a=0\}\le1$,   by the Hauptidealensatz \cite[$0_{\rm IV}$, 16.3.2]{ega}, this contradicts the irreducibility of $\omega$ unless $a\in\co(U)^\times$. In that case, clearly $x\in\mathrm{Reg}(\omega)$.
\end{proof}

For the next result, we require the notion of  {\it rank of a coherent $\co_X$-module},
which is not very common in the literature unless $X$ is integral, but which can be generalized following Scheja and Storch. 

\begin{dfn}[The rank {\cite[Proposition 1.4.3(c)]{bruns-herzog93}}] Say that a coherent $\co_X$-module $\cm$ has rank $r$ if  for each $x\in\mathrm{Ass}(X)$ the module $\cm_x$ is free of rank $r$ over $\co_{X,x}$. 
\end{dfn}

\begin{lemma}\label{04.06.2024--1}Suppose that $\omega$ is very irreducible. Let $n$ be the rank of $T_f$, that is, the relative dimension of $X$.  Then $\omega$ is LDS if and only if $\mathrm{Ker}_\omega$ has rank $n-q$. 
\end{lemma}
\begin{proof}Assume that $\omega$ is LDS. Let   $x\in\mathrm{Ass}(X)\subset\mathrm{Reg}(\omega)$. We then let $\{t_i\}$ be a basis of $T_{f,x}$ and $\{\theta_i\}$ be the dual basis such that $\omega_x=\theta_1\wedge\cdots\wedge\theta_q$.  We know that in this case, $\mathrm{Ker}_{\omega,x}=\oplus_{i=q+1}^n\co_xt_i$. Hence $\mathrm{Ker}_\omega$ is of rank $n-q$. 

Suppose that $\mathrm{Ker}_\omega$ has rank $n-q$ so that  $Q :=T_f/\mathrm{Ker}_\omega$ has rank $q$ \cite[Proposition 1.4.5]{bruns-herzog93}.  Let $x\in \mathrm{Ass}(X)\subset\mathrm{Reg}(\omega)$ so that   $\mathrm{Ker}_{\omega,x}$ and $Q_x$ are free. Consequently, there exists an affine open neighbourhood $U$ of $x$, a basis $\{t_i\}$ of $T_{f}|_U$ such that 
$\mathrm{Ker}_\omega|_U=\oplus_{i=q+1}^n\mathcal O_Ut_i$. Using contractions, we see that $\omega|_U=a\cdot\theta_1\wedge\cdots\wedge\theta_q$, where $\{\theta_i\}$ is the dual basis of $\{t_i\}$. Since $\codim\{a=0\}\le1$ and we assume that $\mathrm{Reg}(\omega)$ is big, we conclude that $\omega|_U=a\cdot\theta_1\wedge\cdots\wedge\theta_q$ with $a\in\co(U)^\times$.  
\end{proof}

\begin{remark}Suppose that $\mathrm{Ker}_\omega$ has rank $m$ so that $Q=T_f/\mathrm{Ker}_\omega$ has rank $n-m$ \cite[1.4.5]{bruns-herzog93}. Let $x\in\mathrm{Ass}(X)$. It follows that for a certain open neighbourhood $U$ of $x$, we have  $\mathrm{Ker}_\omega|_U=\oplus_{i=1}^m\co_Ut_i$, where $\{t_i\}_{i=1}^n$ is a basis of $T_f|_U$. Now, let $\{\theta_i\}$ be the dual of $\{t_i\}$ and let $\omega|_U=\sum_IP_I\theta_I$. If $I\cap \{1,\ldots,m\}\not=\varnothing$, we conclude that $P_I=0$. Hence, unless $\omega|_U=0$, it must be the case  that $m\le n-q$.   
\end{remark}



\subsubsection{From distributions to   forms}

Let us now move on  by explaining how to obtain a twisted form from a distribution $\cv$ on $X/S$. For that, we shall require  
  assumptions on either \begin{enumerate}[(a)]\item the ``singularities'' of $X$,   or on the \item  distribution itself. 
  \end{enumerate}
Condition (a), respectively (b),  shall be treated by imposing property (LF), respectively   (GR) and (D).  Note that, as $f$ is smooth, imposing   properties on $X$ amounts, in many cases, to imposing properties on $S$. See Remark \ref{11.03.2024--1} below for information on this. 
 
{\it Case (a).} We suppose that  
\[\tag{LF}
\begin{array}{c}\text{The scheme $X$ is} \\ 
\text{integral and locally factorial.}
\end{array}
\]
Let $\cv$ be a distribution on $X/S$ and let $q$ be the generic rank of the torsion-free sheaf  $Q_\cv$. Since        $\mathrm{Reg}(\cv)$ is big, we have an isomorphism 
\[
\mathrm{Pic}(X)\stackrel\sim\longrightarrow\mathrm{Pic}(\mathrm{Reg}(\cv))
\] 
\cite[$\mathrm{IV}_4$, 21.6.10, 21.6.12]{ega}. Consequently,  $\restr{\wedge^qQ_\cv}{{\rm Reg}(\cv)}$ is the restriction of a {\it unique invertible sheaf $L_\cv$} which can be, in addition, defined as  
\begin{equation}\label{28.02.2024--1}L_\cv=\begin{array}{c}\text{push forward of}\\
\text{$\restr{\wedge^qQ_\cv}{{\rm Reg}(\cv)}$ to $X$.}
\end{array}
\end{equation} (Recall that   $X$ is normal.) 
Alternatively, we can also define 
\[
L_\cv=\left(\wedge^qQ_\cv\right)^{\vee\vee};
\]
this is because on $X$ (which is integral and locally factorial), any reflexive sheaf of rank one must be invertible \cite[Proposition 1.9]{hartshorne80}. Following  the terminology  of \cite[Definition, p.129]{hartshorne80}, we can more expediently write 
\begin{equation}\label{01.04.2024--1}
\boxed{L_\cv=\det Q_\cv.}
\end{equation}  

\begin{rmk}In certain cases, it is interesting to refrain from looking for an invertible sheaf and deal solely with the {\it rank-one reflexive sheaf} $\det Q_\cv$. Indeed, if $X$ enjoys $\mathrm G_1+\mathrm S_2$ then we have at our disposal the group of  {\it generalized} divisors; see the Definition on p. 300 of \cite {hartshorne94}.  
\end{rmk}

 Note that we have a natural morphism of $\co_X$-modules $\wedge^qQ_\cv\to L_\cv$ which gives rise to a morphism 
\[\omega:\wedge^qT_f\aro L_\cv.\]
The morphism $\omega$ is surjective when restricted to ${\rm Reg}(\cv)$ since $\restr{L_\cv}{{\rm Reg}(\cv)}=\restr{\wedge^qQ_\cv}{{\rm Reg}(\cv)}$.

{\it Case $(b)$.}  Let us    {\it drop assumption (LF)} and give ourselves  a distribution $\cv$ on $X/S$ such that 
\[
\tag{GR} \text{$Q_\cv$ has   rank $q$. } 
\]
See \cite[Definition 1.4.2]{bruns-herzog93}. 
This being so,  we have, by \cite[Proposition 1.4.3]{bruns-herzog93},  that $\mathrm{Reg}(\cv)\supset\mathrm{Ass}(X)$. 
Now, assume in addition that 
\[\tag{D}\begin{split}
 L_\cv&:=\det Q_\cv
 \\&= (\wedge^qQ_\cv)^{\vee\vee}
 \\& \text{is locally free.}\end{split}
\]
(Clearly, $L_\cv$ is a reflexive sheaf of rank one on $X$, but this does not assure that it be invertible.) 
Then, as before, we obtain a twisted form \[\omega:\wedge^qT_f\aro L_\cv\]
and, as before, we have 
\[\mathrm{Ker}_\omega=T_\cv\] because of Proposition  \ref{16.11.2023--1}-(1). 
Consequently, $\cv=\mathcal Z(\omega)$. 

\begin{dfn}[The twisted form of a distribution]The   global section of $\Omega^q_f\otimes L_\cv$ associated to the morphism $\omega$ is called the twisted $q$-form  of $\cv$.
\end{dfn}



\begin{remark}\label{11.03.2024--1}In this section we made certain ``regularity''  hypothesis on  $X$. Now, since $f$ is assumed to be smooth, regularity assumptions on $X$ and on $S$ are  closely related, for example:  if $X$ is regular or normal at $x$, then $S$ is regular or normal at $f(x)$, and conversely  \cite[$\mathrm{IV}_2$, 6.5]{ega}, a point $x\in X$ is associated only if $f(x)$ is likewise, etc.  So   the reader is required to regard ``regularity'' assumptions on $X$ as really assumptions on $S$; the reason we avoid putting assumptions on $S$ lies in our wish to  reduce  the number of references to deep  Commutative Algebra. 
\end{remark}

\subsection{Pull-back of distributions}\label{S:pullback}

Let us give ourselves a commutative diagram of Noetherian schemes
\[
\xymatrix{X_1\ar[r]^\psi\ar[d]_{f_1}& X\ar[d]^f
\\
S_1\ar[r]_\varphi&S,}
\]
where $f$ and $f_1$ are smooth. 

\begin{dfn}\label{10.06.2024--4}
Given a distribution $\cv$ on $X/S$, its pull-back via the morphism $\psi$, denoted by $\psi^\star\cv$, is the distribution on  $X_1/S_1$ whose tangent sheaf $T_{\psi^{\star}\cv}$ is 
\[
\left(\mathrm{Kernel}\,\,\,T_{f_1}\xymatrix{\ar[r]^-{\mathrm D\psi}& \psi^*T_{f}\ar[r]&\psi^*Q_\cv} \right)^{\mm{ssat}}.
\]
If no confusion is likely, we shall write $\cv|_{X_1}$ in place of $\psi^\star\cv$.
\end{dfn}

It should be noted that if the derivative $\mathrm D\psi:T_{f_1}\to \psi^*T_{f}$ is surjective, then $Q_{\psi^\star \cv }$ is isomorphic to the torsion-less quotient $(\psi^*Q_\cv)_{\mm{tol}}$. In general, we can only say that $Q_{\psi^\star \cv}$ will be a quotient of an $\co_{X_1}$-submodule of $\psi^*Q_{\cv}$.
The following result is helpful in grasping the tangent sheaf of a pull-back. 

\begin{prop}\label{24.03.2024--2}
Let \[
\xymatrix{X_1\ar[r]^\psi\ar[d]_{f_1}& X\ar[d]^f
\\
S_1\ar[r]_\varphi&S,}
\]
be a commutative diagram of Noetherian schemes with $f$ and $f_1$ smooth. Let $\cv=(T_\cv,\alpha)$ be a distribution on $X/S$ and write $\cv_1$  for its pull-back via $\psi$. 

Let $U\subset X$   be the set of regular points of $\cv$ and write $U_1=\psi^{-1}(U)$. Assume that the derivative 
\[\mathrm D\psi:T_{f_1}\longrightarrow\psi^* T_f\]
is surjective. 
Then: 
\begin{enumerate}[(1)]
\item   The set of regular points ``only increases'': $U_1\subset \mathrm{Reg}(\cv_1)$. 
\item Suppose that $U_1$ is schematically dense, that is, $\mathrm{Ass}(X_1)\subset U_1$. Then, the restriction 
\[
\restr{\psi^*(\alpha)}{U_1}: \restr{\psi^*T_\cv}{U_1}\longrightarrow \restr{\psi^*T_{f}}{U_1}\]
is injective and $T_{\cv_1}$ is the canonical extension of 
\[(\mathrm D\psi)^{-1}\left(\restr{\psi^*T_\cv}{U_1}\right)\subset \restr{T_{f_1}}{U_1}
\]
to $X_1$.
\end{enumerate}
In particular, if the diagram above is cartesian, i.e. $X_1=X\times_SS_1$, then both items (1) and (2) hold. 
\end{prop}

\begin{proof}(1)  We have a commutative diagram 
\[
\xymatrix{
\psi^*T_\cv\ar[rr]^{\psi^*(\alpha)} && \psi^*T_f\ar@{->>}[rr]^{\psi^*(q)} && \psi^*Q_{\cv}\ar@{->>}[r]^-{\chi}&(\psi^*Q_{\cv})_{\rm{tol}}
\\
T_{\cv_1}\ar@{^{(}->}[rr]_{\alpha_1}&& T_{f_1} \ar@{->>}[u]^{\mathrm D\psi}\ar@{->>}[rr]_{q_1}&&Q_{\cv_1}\ar[ru]_\sim&
}
\]
Now, over $U_1$, we know that $\chi$ is an isomorphism and that $\psi^*Q_\cv$ is locally free. Hence, over $U_1$, $\cv_1$ is regular. 
To establish (2), we note that over $U_1$, the arrow $\psi^*(\alpha)$ is in addition a monomorphism. Hence, over $U_1$, the module $T_{\cv_1}$ coincides with $(\mathrm D\psi)^{-1}\left(\restr{\psi^*T_\cv}{U_1}\right)$. 
From Proposition \ref{16.11.2023--1}-(1), we see that $T_{\cv_1}$ is the canonical extension of the sheaf $\mathrm D\psi^{-1}\left(\restr{\psi^*T_\cv}{U_1}\right)$. 

To end the verification of all our claims, we note that if the diagram in the statement of the proposition is cartesian, then the derivative $\mathrm D\psi:T_{f_1}\to \psi^*T_f$ is in fact an {\it isomorphism} \cite[$\mathrm{IV}_4$, 16.5.11, p.30]{ega}.
\end{proof}

\begin{cor}[Transitivity of the pull-back]\label{transitivity_fibres}
Let  
\[
\xymatrix
{X_2\ar[r]^{\psi_1}\ar[d]_{f_2}&X_1\ar[r]^\psi\ar[d]_{f_1}& X\ar[d]^f
\\
S_2\ar[r]_{\varphi_1}\ar[r]&S_1\ar[r]_\varphi&S,}
\]
be a commutative diagram of Noetherian schemes with $f$,  $f_1$ and $f_2$ smooth. Assume in addition that $\mathrm D\psi$ and $\mathrm D\psi_1$ are surjective and that letting $U=\mathrm{Reg}(\cv)$, $U_1=\psi^{-1}(U)$ and $U_2=\psi_1^{-1}(U_1)$, we have $U_1$ and $U_2$ schematically dense in $X_1$ and $X_2$ respectively.  
Then 
\[
(\psi\psi_1)^\star(\cv)=\psi_1^\star(\psi^\star\cv)
\]
\end{cor}

\subsection{Families of distributions }
Let $f:X\to S$ be a smooth morphism of Noetherian schemes and, as usual, let $X_s$ stand for the schematic fibre above $s\in S$.  We want to single out certain distributions on $X/S$ to be called ``families''.   

\begin{dfn}[Families]\label{18.03.2024--1}
 {\it A family of singular distributions  on $X/S$}, is a distribution $\mathcal V$ on $X/S$
such that 
  for any $s \in S$,  the open subset $\mathrm{Reg}(\cv)\cap X_s$ is big in $X_s$, i.e. 
 \[\codim(\mathrm{Sing}(\cv)\cap X_s\,,\,X_s)\ge2.\]
\end{dfn}

\begin{rmk}One possible way of defining the notion of a family of foliations is to impose that the sheaf $Q_\cv$ be flat over $S$ \cite[p. 2713]{correa-jardim-muniz22}. On the other hand, as \cite[Section 2]{quallbrunn15}  explains, this imposition turns out to exclude natural cases. Indeed, consider the morphism of affine schemes defined by the  projection in the following diagram:
\[\xymatrix{\mathrm{Spec}\,\mathbb C[s,t,w,x,y,z]\ar[rr]^-{f}&&\mathrm{Spec}\,\mathbb C[s,t]
\\X\ar@{=}[u]&&S\ar@{=}[u].}
\]
  Then, if $\cv$ is the distribution on $X/S$ given as the kernel of  
\[
x\,\mathrm dx+y\,\mathrm dy+s z\, \mathrm dz+tw\,\mathrm dw, 
\]
Quallbrunn shows that $T_\cv$ is not flat over $S$, which ensures that $Q_\cv$ is likewise not flat. 
\end{rmk}

\begin{lemma} Let $\cv$ be a family of distributions on $X/S$ and denote, for each $s\in S$, the pull-back of the distribution $\cv$ to $X_s$ by $\cv_s$. (See Section \ref{S:pullback}.) The following are true. 
\begin{enumerate}[(1)]\item    We have  $\mathrm{Reg}(\cv)\supset\mathrm{Ass}(X)$.   
\item The function $s\mapsto \mathrm{rank}\,Q_{\cv_s}$ is locally constant.  
\item For each $s\in S$ we have 
\[\det(T_{\cv_s})=\det\left(\restr{T_{\cv}}  {X_s}\right)\quad\text{and}\quad  \det\left(Q_{\cv_s}\right)=\det\left( \restr{Q_\cv}{X_s}\right) .\]
\end{enumerate}
\end{lemma}  
\begin{proof}(1) We know that the associated points of $X$ are the generic points of the fibres $X_s$ \cite[$\mathrm{IV}_2$, 3.3.1]{ega}, where $s\in\mathrm{Ass}(S)$. No such point can belong to $\mathrm{Sing}(\cv)$. 

(2) We abbreviate $\mathrm{Reg}(\cv)$ to $U$. We know that   $\restr {Q_\cv}{U}$ is  locally free. Write $\sqcup_\lambda U(\lambda)$ for its decomposition into connected components and let $r_\lambda=\mathrm{rank}\,Q_\cv|_{U(\lambda)}$. Let us denote $U_s$, respectively $U(\lambda)_s$, for the open subscheme of $X_s$ defined by $U$, respectively $U(\lambda)$. Now  
 $U_s=\sqcup_\lambda U{(\lambda)}_s$ and since $U_s$ is irreducible, there exists a unique $\lambda_s$ such that $U(\lambda_s)\supset U_s$. As $f$ is an open map, $f(U(\lambda_s))$ is open and for each $s'\in f(U(\lambda_s))$ we have $U_{s'}\subset U(\lambda_s)$. This means that $Q_\cv$ has rank $r_{\lambda_s}$ on $f(U(\lambda_s))$.

(3)   Then $\restr{T_{\cv_s}}{U_s}\simeq \restr{T_\cv}{U_s}$ and $\restr {Q_{\cv_s}}{U_s}\simeq \restr{Q_{\cv}}{U_s}$. Since $\codim(X_s\setminus U_s,X_s)\ge2$, the result follows.  
\end{proof}
  
Another useful property of families is: 

\begin{lemma}Let $\tilde S\to S$ be a morphism of Noetherian schemes and let $\cv$ be a family of distributions on $X/S$. Let $\tilde X=X\times_S\tilde S$ and let   $\tilde \cv$ by the pull-back of $\cv$ to $\tilde X$. Then $\tilde \cv$ is a family of distributions on $\tilde X$. 
\end{lemma} 
\begin{proof}Letting $\psi:\tilde X\to X$ stand for the natural projection, we know that $\mathrm{Sing}(\tilde \cv)\subset \psi^{-1}(\mathrm{Sing}(\cv))$ (Proposition  \ref{24.03.2024--2}-(2)). Hence, $\mathrm{Sing}(\tilde \cv)\cap \tilde X_s\subset \mathrm{Sing}(\cv)\cap \tilde X_s$. 
\end{proof}

\begin{rmk}Let $\cv$ be a family of distributions with regular set $R$. Then, for $s\in S$, we have  $Q_{\cv_s}|_{R_s}=Q_\cv|_{R_s}$. Now, we know that $Q_\cv|_R$ is locally free, or rank $q$ say, and hence $Q_\cv|_{R_s}$ is locally free of rank $q$. In conclusion: The function $s\mapsto\mathrm{rank}\,Q_{\cv_s}$ is locally constant.   \end{rmk}


\subsection{Kupka singularities}\label{kupka_singularities}
A very important type of singularity for a foliation defined by a differential form  was unveiled by I. Kupka and A. de Medeiros and    now bears Kupka's  name. In order not to change too much the context of this paper, let us remain in the complex analytic world, although 
definitions can be made in the differentiable context. 
 
Let  $M$ be an $n$-dimensional  complex manifold and  $\omega$  a holomorphic  $q$-form on $M$. A point $x\in M$ where $\omega(x)=0$ is called a Kupka singularity when $\mm d\omega(x)\not=0$. 
The relevance of this definition comes from the following result, proved by I. Kupka in the case $q=1$ \cite{kupka} and A. Medeiros in general \cite[Proposition 1.3.1]{medeiros}.

\begin{thm}\label{kupka_theorem}Suppose that $x\in M$ is a Kupka singularity of the \textbf{integrable} $q$-form $\omega$. 
\begin{enumerate}[(1)]\item There exists a local system of coordinates $(z_1,\ldots,z_n)$ on an open neighbourhod $U$ of  $x$ such that $\omega|_U$ admits a ``reduction to $q+1$  variables'', viz.  \[\omega|_U=\sum_{\substack{I\subset \{1,\ldots,q+1\}\\ \#I=q}}a_I(z_1,\ldots,z_{q+1})\,dz_I.
\]
\item There exists an open neighbourhood $U$ of $x$, a local chart $(\varphi,\psi):U\to V\times W$, where  $V\subset \C^{q+1}$ and $W\subset \C^{n-q-1}$  are open sets,  and a holomorphic $q$-form $\zeta\in \Omega^q(V)$ such that $\omega|_U=\varphi^*\zeta$.
\item\label{kupka_theorem_simple} The tangent sheaf of the foliation  $\mathcal Z(\omega)$ is free at $x$. 
\end{enumerate}
\end{thm}

We note that conclusion (\ref{kupka_theorem_simple}) 
in Theorem \ref{kupka_theorem} is certainly weaker than the others, as the latter give  a clear structure result for forms. Our goal now is to recuperate, in a completely algebraic setting,   condition   (\ref{kupka_theorem_simple}) of Theorem \ref{kupka_theorem}.

Let $f:X\to S$ be a smooth morphism of Noetherian schemes; assume that $X$ is integral and locally factorial, i.e. we are in the setting of Section \ref{twisted_form}.  
Let $q$ be a positive integer and $L$
an invertible sheaf on $X$. We give ourselves a twisted $q$-form $\omega$ with coefficients in $L$ and denote by $\cv$ the associated distribution, i.e. $\cv=\mathcal Z(\omega)$ (cf. Section \ref{formstodistributions}).  
Let now $x\in \mathrm{Sing}(\cv)$ and write    $L_{_x}=\co_{X,x}\ell$. Then 
$\omega= \xi\otimes \ell$ with $\xi\in\Omega^q_{f,x}$,  and   since $x\in \mathrm{Sing}(\cv)$,  we have $\xi(x)=0$. It is a simple matter to prove that,  if $\ell'\in L_{\cv,_x}$ is another generator,  so that $\omega_x=\xi'\otimes \ell$ with $\xi'\in\Omega_{f,x}^q$, then 
\[\mm d  \xi(x)=0 \quad\text{  if and only if}\quad  \mm d\xi'(x)=0.\](Here the differential is the relative differential.) Consequently,  the following definition makes perfect sense: 
\begin{dfn}A point $x\in \mathrm{Sing}(\cv)$ is a Kupka singularity or a Kupka point,  if, in the above notation,   $\mm d\xi(x)\not=0$. The  set of Kupka points in $\mathrm{Sing}(\cv)$ shall be denoted by $\mathrm{Kup}(\cv)$. (Note that $\mathrm{Kup}(\cv)$ is open in $\mathrm{Sing}(\cv)$.) For expediency, we shall also let ${\rm NKup}(\cv)$ denote the  set $\sing(\cv)\setminus\mathrm{Kup}(\cv)$; this is the set of non-Kupka points. 
\end{dfn}

\begin{prop}\label{23.05.2024--1}Assume that $X$ is locally factorial and 
 $\omega$ is  \textbf{integrable}. Then, for each $x\in \mathrm{Kup}(\cv)$, the $\co_{X,x}$-module   $T_{\cv,x}$  is   free. Said differently, $\mathrm{TSing}(\cv)\subset {\rm NKup}(\cv)$.
\end{prop}

\begin{proof}Let $\eta$ be the generic point of $X$.  
Let $U$ be an open neighbourhood of $x$ where $L|_U=\co_U\cdot\ell$ and write $\omega=\xi\otimes\ell$. Let now $\pi_1,\ldots,\pi_q$ be rational 1-forms such that $\xi_\eta=\pi_1\wedge\ldots\wedge\pi_q$.  Then   $\mm d\xi_\eta=\sum_j\pm \pi_1\wedge\cdots\wedge \mm d\pi_j\wedge\cdots\wedge\pi_q$. Now $\mm d\pi_j=\sum_{i}\alpha_{ij}\wedge\pi_i$ with $\alpha_{ij}\in\Omega^1_{f,\eta}$ due to integrability of $\omega$. Consequently, $\mm d\xi_\eta=\sum_j\pm \alpha_{jj}\wedge\pi_1\wedge\cdots\wedge\pi_q$ is decomposable. Now, we consider the distribution on $U$ defined by $\mm d\xi$. From $\mm d\xi(x)\not=0$, we know that $x\in\mathrm{Reg}(\mathcal Z(\mm d\xi))$. The Proposition  is then a consequence of Proposition \ref{13.12.2023--2} below. 

\end{proof}

\begin{prop}\label{13.12.2023--2}We assume that $X$ is a locally factorial integral scheme. Let $\cv$ be a distribution  of codimension $q$. Assume that $\cv$ contains a distribution $\cw$ of codimension $q+1$. Then, $\mathrm{TSing}(\cv)\subset \mathrm{Sing}(\cw)$. 
\end{prop}
\begin{proof}Let $x\in\mathrm{Reg}(\cw)$, so that $T_{\cw,x}$ and $T_{f,x}/T_{\cw,x}$ are free $\co_{X,x}$-modules of finite rank.  From the exact sequence $0\to T_{\cv,x}/T_{\cw,x}\to T_{f,x}/T_{\cw,x}\to T_{f,x}/T_{\cv,x}\to0$ we conclude  that  $T_{\cv,x}/T_{\cw,x}$ is reflexive \cite[Proposition 1.1]{hartshorne80}. The hypothesis on the rank implies that $T_{\cv,x}/T_{\cw,x}$ is free \cite[Proposition 1.9]{hartshorne80}. Hence   $0\to T_{\cw,x}\to T_{\cv,x}\to T_{\cv,x}/T_{\cw,x}\to0$ splits  on the right and  $T_{\cv,x}$ is free. 
\end{proof}

\end{document}